\documentclass[12pt]{amsart}

\newcounter{ourcount}
\newcounter{myenumi}
\usepackage[vcentermath]{youngtab}
\usepackage{ulem}

\usepackage[utopia]{mathdesign} 
\usepackage{mathabx}

\usepackage{mathtools}
\usepackage[mathscr]{eucal}
\usepackage{amsthm}
\usepackage{mathtools}
\usepackage{ytableau}
\usepackage{lscape}

\usepackage{epstopdf}
\usepackage{ifpdf}
\ifpdf
\usepackage{t-angles}
\usepackage[matrix,arrow,curve,pdf]{xy} 
\else
\usepackage{t-angles}
\usepackage[matrix,arrow,curve,dvips]{xy} 
\UsePSspecials{dvips}
\fi

\usepackage{graphicx} 

\advance\textheight\topskip
\allowdisplaybreaks

\newcommand{\qwB}{\mathsf{qw}\kern-1.8pt\mathscr{B}}
\newcommand{\qwb}{\mathsf{qw}\kern-1.8pt\mathscr{B}}
\newcommand{\wb}{\mathsf{w}\kern-1.8pt\mathscr{B}}

\newcommand{\tensor}{\mathbin{\otimes}}

\newcommand{\EE}{\mathscr{E}}

\newcommand{\BareWt}{\kern2pt\overline{\kern-2pt\mathsf{Wt}\kern-2pt}\kern2pt}

\newcommand{\bref}[1]{\textbf{\ref{#1}}}

\newcommand{\mfrac}[2]{\raisebox{.8pt}{\mbox{\small$\displaystyle\frac{#1}{#2}$}}}
\newcommand{\ffrac}[2]{\raisebox{.5pt}{\mbox{\footnotesize$\displaystyle\frac{#1}{#2}$}}}
\newcommand{\fffrac}[2]{\raisebox{.9pt}{\mbox{\scriptsize$\displaystyle\frac{#1}{#2}$}}}

\newcommand{\oC}{\mathbb{C}}
\newcommand{\oZ}{\mathbb{Z}}

\newcommand{\qSL}[1]{\mathscr{U}_{\q} s\ell(#1)}
\newcommand{\SL}[1]{s\ell(#1)}
\newcommand{\Fi}{B}
\newcommand{\Fii}{F}

\newcommand{\Ei}{C}
\newcommand{\Eii}{E}
\newcommand{\q}{\textit{\textsf{q}}}
\newcommand{\qGL}[1]{\mathscr{U}_q g\ell(#1)}
\newcommand{\GL}[1]{g\ell(#1)}
\newcommand{\qint}[1]{{\textstyle[#1]}}
\newcommand{\one}{\boldsymbol{1}}

\newcommand{\rep}{\mathscr}

\newcommand{\repX}{\rep{X}}
\newcommand{\repZ}{\rep{Z}}
\newcommand{\repZold}{\rep{Z}}
\newcommand{\repZZ}{\widebar{\mathscr{Z}}}

\newcommand{\repR}{\rep{R}}
\newcommand{\repRR}{\widebar{\mathscr{R}}}

\let\BTL=\leftarrow
\let\BTR=\rightarrow
\let\BTD=\downarrow
\let\BTU=\uparrow
\newcommand{\ket}[1]{|#1\rangle}
\newcommand{\ZC}[1]{\ket{{\textstyle#1}}^{\!\BTL}}
\newcommand{\ZB}[1]{\ket{{\textstyle#1}}^\BTR}
\newcommand{\ZD}[1]{\ket{{\textstyle#1}}^{\!\BTD}}
\newcommand{\ZU}[1]{\ket{{\textstyle#1}}^\BTU}
\newcommand{\myboldsymbol}{\pmb}
\newcommand{\bExti}{\boldsymbol{\textup{Ext}}^{\myboldsymbol{1}}}
\newcommand{\Exti}{\mathop{\mathrm{Ext}^1}}
\newcommand{\myatop}[2]{{#1\atop\mbox{\small$(#2)$}}}

\newcommand{\TOP}{{}^{{\scriptscriptstyle\bigtriangleup}}}
\newcommand{\BOT}{{}_{{\scriptscriptstyle\bigtriangledown}}}

\newcommand{\three}{\myboldsymbol{3}}
\newcommand{\bthree}{\myboldsymbol{\overline{3}}}
\newcommand{\Chain}{\mathscr{T}}
\newcommand{\Chat}{T^{\text{at}}}
\newcommand{\la}{\lambda}

\newcommand{\emp}{\emptyset}

\newcommand{\res}{\text{res}}
\newcommand{\resN}{\res^{m,n}_{m,n-1}}
\newcommand{\add}{\text{add}}
\newcommand{\remov}{\text{rem}}
\newcommand{\Cr}{\mathscr{C}\text{r}}

\newcommand{\red}[1]{\textcolor{red}{#1}}

\newcommand{\myar}{\ar@{-->}@[|(1.0)]}
\newcommand{\Da}[2]{A^{#1}_{#2}}
\newcommand{\Db}[2]{B^{#1}_{#2}}
\newcommand{\Dc}[2]{C^{#1}_{#2}}
\newcommand{\Daa}[2]{\hat{A}^{#1}_{#2}}
\newcommand{\Dbb}[2]{\hat{B}^{#1}_{#2}}
\newcommand{\Dcc}[2]{\hat{C}^{#1}_{#2}}
\newcommand{\at}{\mathscr{A}\text{t}}
\newcommand{\bup}{\biguplus}
\newcommand{\operS}{\hat{G}}
\newcommand{\Alg}{\mathscr{X}}
\newcommand{\ppp}{\mathbb{P}}
\newcommand{\qqq}{\mathbb{Q}}

\numberwithin{equation}{section}

\makeatletter
\@addtoreset{equation}{section}
\@addtoreset{subsubsection}{section}

\def\@secnumfont{\bfseries}
\def\subsubsection{\@startsection{subsubsection}{3}%
  \z@{.5\linespacing\@plus.7\linespacing}{-.5em}%
  {\normalfont\bfseries}}
\def\paragraph{\@startsection{paragraph}{4}%
  \z@\z@{-\fontdimen2\font}%
  \normalfont\bfseries}
\def\subparagraph{\@startsection{subparagraph}{5}%
  \z@\z@{-\fontdimen2\font}%
  \normalfont\bfseries}

\makeatother

\swapnumbers
\newtheorem{Thm}[subsection]{Theorem}
\newtheorem{thm}[subsubsection]{Theorem}

\newtheorem{lemma}[subsubsection]{Lemma}

\newtheorem{prop}[subsubsection]{Proposition}

\newtheorem{cor}[subsubsection]{Corollary}

\newtheorem{conj}[subsubsection]{Conjecture}

\newtheorem{Conj1}[subsection]{Conjecture 1}
\newtheorem{Conj2}[subsection]{Conjecture 2}

\theoremstyle{definition}

\newtheorem{dfn}[subsubsection]{Definition}
\newtheorem{rem}[subsubsection]{Remark}

\begin{document}
\begin{flushright}
\vspace{1mm}
FIAN/TD/2017-9\\
\end{flushright}

\title[Bimodule structure of the mixed tensor product]{Bimodule structure of the mixed tensor product over $\qSL{2|1}$ and quantum walled Brauer algebra}

\author{D.\,V.\;Bulgakova, A.\,M.\;Kiselev and I.\,Yu.\;Tipunin}
\address{I.E.Tamm Department of Theoretical Physics, Lebedev Physical Institute,
Leninsky prospect 53, 119991, Moscow, Russia\hfill\mbox{}\linebreak \texttt{dvbulgakova@gmail.com}, \texttt{kiselevalexs@gmail.com}, \texttt{tipunin@gmail.com}}

\begin{abstract}
We study a mixed tensor product 
$\three^{\tensor m} \tensor \bthree^{\tensor n}$
of the three-dimensional fundamental representations of the Hopf algebra
$\qSL{2|1}$, whenever $\q$ is not a root of unity.
Formulas for the decomposition of tensor products of any simple and projective $\qSL{2|1}$-module with the generating modules $\three$ and $\bthree$ are obtained. 
The centralizer of $\qSL{2|1}$ on the mixed tensor product is calculated. It is shown to be the quotient 
$\Alg_{m,n}$ of the quantum walled Brauer algebra $\qwb_{m,n}$.
The structure of projective modules over $\Alg_{m,n}$ is written down explicitly.
It is known that the walled Brauer algebras form an infinite tower. We have calculated the corresponding restriction functors on simple and projective modules over $\Alg_{m,n}$. This result forms a crucial step in decomposition of the mixed tensor product as a bimodule over
$\Alg_{m,n}\boxtimes\qSL{2|1}$. We give an explicit bimodule structure for all $m,n$.\\[1cm]
\end{abstract}

\maketitle
\thispagestyle{empty}

\section{Introduction}
Over the course of the last twenty years Logarithmic conformal field theory (LCFT) has established itself as 
an area of extensive interaction between models of statistical physics such as
percolation, the sand pile model, dense polymers as well as other models with nonlocal
observables on the one hand, and modern topics in mathematics such as Nichols
algebras, quantum groups, braided categories, VOA theory 
and diagram algebras on the other.
One of the most developed approaches~\cite{FGST2006_1, FGST2006_3, FuchsHwangST} to
constructing LCFT is based on the intersection of screening operator kernels. 
In this approach one chooses a lattice vertex operator algebra (VOA) and fixes
a set of fields $v_i$, which correspond to representations of VOA and are called
screening currents. The zero modes of these currents $s_i = \oint v_i$ are called
screenings. Under certain integer valuedness conditions on scalar products of the 
screening currents momenta, the screenings form a finite-dimensional Nichols
algebra (see examples in \cite{SemTip2011, Simon2017}). 
Under these conditions the intersection of the
screening kernels is a vacuum module of a rational LCFT:
$\text{Vac} = \bigcap_i \text{Ker} \, s_i$. 
In this case, LCFT is a representation space of the rational $W$-algebra $\mathscr{W}=\text{Vac}$.

The algebra $\mathscr{W}$ has only a finite number of irreducible
representations. The set of simple and projective $\mathscr{W}$-modules is closed
under fusion and the characters of  the $\mathscr{W}$-irreducible modules
generate a finite-dimensional representation of the modular group.

Another source of LCFT is given by various lattice models 
\cite{PRZ, ReadSaleur2007, ReadSaleur2007_2, Ridout2007, Ridout2015, ga, AzatVasseur2013}. CFT appears
naturally as a scaling limit of lattice models in the critical point, see e.g.
\cite{DiFrancesco}. Then, a mathematically rigorous program on algebraic construction of the scaling limits was initiated in \cite{AzatReadSaleur2016}.
If one considers nonlocal observables (for example, the
cluster probability in percolation theory
\cite{DuplantierSaleur1987, FrancSaleurZuber1987, Cardy1991, Ruelle2013}) in
the lattice 
model, then in the scaling limit an LCFT is in general expected to appear, 
and in several models \cite{Rasmussen2007_1, Rasmussen2007_2, ga,
  AzatSaleur2016} its appearance is shown explicitly. 

The standard approach to studying the lattice models is the transfer-matrix
method 
\cite{Baxter, Quantum_Groups_2dim}. In this approach a connection with a spin
chain is established by the Hamiltonian limit. Another feature of lattice
models with nonlocal observables is that in the Hamiltonian limit there exist
nontrivial Jordan blocks in the Hamiltonian  
\cite{Dubail2010, VasseurJacobsen2011, MorinDuchesne2011, PearceRasmussen2013,
  BPPR_2014} 
(see discussion on the Jordan blocks problem in the algebraic Bethe ansatz
approach in \cite{AzatNepomechie2016}). 
From the side of LCFT the existence of
nontrivial Jordan blocks in the Hamiltonian is expressed in the fact that
the conformal dimension operator $L_0$ becomes non-diagonalizable 
and conformal blocks admit logarithmic terms.

In both approaches a quantum group plays a crucial role 
\cite{FGST2006_2, FGST2007, ReadSaleur2007_2, AzatVasseur2013}. In the first case quantum group appears as a double bosonization of the
algebra generated by screenings \cite{SemTip2013CFT}. In the second case the
spin-chain can be constructed as tensor product of fundamental representations
of 
the quantum group.

For the simplest case of $(1,p)$ LCFT models 
\cite{GaberdielKausch,  FHST2004, AdamovicMilas2007_1,AdamovicMilas2007_2},
the corresponding spin-chain $T_N$ is a tensor product of two-dimensional
representations of the quantum group
$\qSL{2}$, \cite{AzatSalTip2012}, and the $T_N$ is called the Heisenberg
spin-chain. 

An interesting generalization of the Heisenberg spin-chain is a spin-chain based on the algebra 
$\qSL{M|N}$ \cite{Candu2010}. Such spin-chains describe interaction between spin and other degrees of freedom. For
instance, $\qSL{2|1}$-spin-chain is related with (integrable) t-J model
which includes the interaction between spin and charge degrees of  freedom,
\cite{Links,Abad, FeiYue1994}.

On the side of LCFT, models related with the quantum group $\qSL{2|1}$ are
constructed in \cite{SemTip2013CFT} by the approach based on intersection of
kernels of the screening operators. Rational  $W$-algebras
$\mathscr{W}$ containing as a subalgebra $\hat{sl}(2)_k$ at a rational level $k$
naturally occur in these models. At the same time models over $\hat{sl}(2)_k$
are not rational. In this case $\hat{sl}(2)_k$ is an analog of the Virasoro
algebra in $(1,p)$-models. More on LCFT with $\hat{sl}(2)_k$ see in
\cite{Ridout2012, Ridout2013, RidoutWood_2015}.

In order to investigate how LCFT with quantum group $\qSL{2|1}$ appears in the scaling
limit of the spin-chain,
it is natural to follow the approach proposed in \cite{AzatSalTip2012}. In the
present paper we study $\qSL{2|1}$ mixed tensor product which is
the space of 
states for the spin-chains with $\qSL{2|1}$ symmetry. But the case when $\q$
is a 
root of
unity $\q = e^{i\pi/p}$, is more complicated and we leave it for a separate work.
Therefore in the present paper we
consider only the algebra with a generic value of the parameter~$\q$.

It is useful to make an analogy with the Heisenberg $\qSL{2}$-spin-chain with
generic $\q$.
Its centralizer $\mathscr{C}(\qSL{2})$ on the chain is the Temperley-Lieb algebra
$\mathscr{C}(\qSL{2})=TL_N$ with the same value of the parameter $\q$. Thus, the
spin-chain space of states can be expressed as a bimodule $T_N = \bigoplus_i
V_i \boxtimes M_i$, where $V_i$ and $M_i$ are some simple $\qSL{2}$- and
$TL_N$-modules. When $N \to \infty$, the algebra $TL_N$ conjecturally converges to the Virasoro algebra.
When $\q$ is a root of unity the centralizer of the Temperley--Lieb algebra
$TL_N$ is the Lusztig limit $\mathscr{L}\qSL{2}$ of $\qSL{2}$.
In this case the bimodule decomposition of the spin chain contains non
semisimple 
summands \cite{AzatVasseur2013}.

When $\q$ is a root of unity, the algebra $\mathscr{L}\qSL{2}$ contains the restricted quantum group $\overline{\mathscr{U}}_{\q} s\ell(2)$ as a subalgebra, see details in \cite{BGT}.
In \cite{AzatSalTip2012} it is shown that the centralizer of $\overline{\mathscr{U}}_{\q} s\ell(2)$ on the spin-chain $T_N$ is the algebra $\mathscr{W}_N$, which contains the algebra
$TL_N$. In the limit $N \to \infty$ the algebra $\mathscr{W}_N$ gives the triplet algebra
$\mathscr{W}$ built by the lattice VOA construction.

In case of $\qSL{2|1}$ we take its (mutually dual) fundamental representations
which are three-dimensional and denote them by $\three$ and $\bthree$.
We study the mixed tensor product
\begin{equation}\label{chain}
\Chain_{m,n} = \three^{\tensor m} \tensor \bthree^{\tensor n}.
\end{equation}
The tensor product $\Chain_{m,n}$ is the space of states of different integrable spin-chains with $\qSL{2|1}$ symmetric hamiltonians, examples of which are considered in 
\cite{Saleur2005, ReadSaleur2007, Frahm2011, Frahm2012}.
We let $\Alg_{m,n}$ denote the centralizer of $\qSL{2|1}$ on $\Chain_{m,n}$,
$\mathscr{C}(\qSL{2|1})=\Alg_{m,n}$. It is shown in \cite{ShaderMoon02, DDS-1, DDS-2, Stroppel12} that $\Alg_{m,n}$ is isomorphic to some quotient of the quantum walled Brauer algebra $\qwb_{m,n}$.
In this paper we do not give an explicit description of $\Alg_{m,n}$ itself, but describe simple and projective modules over $\Alg_{m,n}$.
We find the decomposition of the chain $\Chain_{m,n}$ as a bimodule over
$\qSL{2|1}$ and $\Alg_{m,n}$.
Even for generic values of $\q$, the bimodule is not semisimple. We give the
bimodule in an explicit form in Theorem \bref{bimodulethm}.

The quantum walled Brauer algebra $\qwb_{m,n}$ was introduced in \cite{Leduc,Halv, KM1993}.
The two-parame\-tric algebra $\qwb_{m,n}$ was introduced in \cite{En} and the structure of
the simple modules was described implicitly. Modules over $\qwb_{m,n}$ and its
classical analogue $\mathsf{w}\kern-1.8pt\mathscr{B}_{m,n}$ were investigated in
\cite{CDM2007, Cox2010, RuiSong_1, RuiSong_2,  ST2015}.

For arbitrary values $M, N$ the algebra $\qSL{M|N}$ on the appropriate mixed
tensor product
(which is the tensor product of its fundamental representations) is centralized by
some quotient of $\qwb_{m,n}$, see also \cite{Stroppel12, Stroppel2014}.
If $N=0$ the bimodule is semisimple. 
We study the simplest non-semisimple case  $N=1$.

The outline of the article is as follows. In Sec.~\ref{sec:sl(2|1)} we define
the algebra $\qSL{2|1}$ and classify its finite-dimensional simple and projective
modules. In Sec.~\ref{sec:chain} we
describe the mixed tensor product and introduce the centralizer $\Alg_{m,n}$. First, we
prove the formulas for the tensor products of modules needed to the mixed
tensor product
decomposition. Next, we show that the centralizer is a quotient of the algebra
$\qwB_{m,n}$. In Sec.~\ref{sec:modules} we describe simple and projective
modules over $\Alg_{m,n}$ and the restriction functors on them. 
In the last Sec.~\ref{sec:bimodule} we describe the
bimodule structure and give a sketch of a proof for the bimodule
decomposition formula.

\section{The Hopf algebra $\qSL{2|1}$}\label{sec:sl(2|1)}
\subsection{Definition of $\qSL{2|1}$}
Quantum analogues of superalgebras $\SL{2|1}$ and $\GL{2|1}$ was studied intensively in \cite{pal1,pal2,zhang,ky}.
We describe the Hopf algebra $\qSL{2|1}$ by a system of generators and
relations.  
In this section and in the entire paper we assume that the parameter $\q$ is not a root of unity.
We choose the generators adapted to the Hopf subalgebra
structure $\qSL{2|1}\supset\qGL{2}\supset\qSL{2}$ (such that the
embeddings become tautological); we extensively use these subalgebras
while working with $\qSL{2|1}$ modules in the sequel.  The Hopf
subalgebra $\qSL{2}$ in $\qSL{2|1}$ is generated as an associative
algebra by $\Eii$, $K$, and $\Fii$ with the relations
\begin{equation}
  \begin{gathered}
    K\Fii=\q^{-2}\Fii K,\quad \Eii \Fii-\Fii \Eii
    =\mfrac{K-K^{-1}}{\q-\q^{-1}},\quad K\Eii=\q^{2}\Eii K.
  \end{gathered}
\end{equation}
The larger algebra $\qGL{2}$ contains an additional generator $k$
satisfying the relations
\begin{gather}
  k\Fii= \q \Fii k,
  \quad
  k\Eii = \q^{-1}\Eii k,
  \quad
   kK = Kk.
\end{gather}
We call the generators $\Eii$, $\Fii$, $K$ and $k$ bosonic.  There are
two additional generators $B$ and $C$, which extend $\qGL{2}$ to
$\qSL{2|1}$, and which we call fermionic, or simply fermions.  The
relations that involve the fermions $B$ and $C$ are
\begin{equation}
  \begin{gathered}
  k\Fi =-\Fi k,\quad K\Fi =
  \q \Fi K, \quad k\Ei =-\Ei k,\quad
  K\Ei = 
  \q^{-1}\Ei K,
  \\[-2pt]
  \Fi^2 = 0, \quad \Fi \Ei -\Ei \Fi =\mfrac{k-k^{-1}}{\q-\q^{-1}},
  \quad \Ei^2 = 0,
  \\
  \Fii \Ei -\Ei \Fii=0, \qquad \Fi \Eii -\Eii \Fi =0,
  \\
  \Fii\Fii\Fi - 
  \qint{2}\Fii\Fi\Fii + \Fi\Fii\Fii = 0, \quad
  \Eii\Eii\Ei - 
  \qint{2} \Eii\Ei\Eii + \Ei\Eii\Eii = 0,
\end{gathered}
\end{equation}
where we use $q$-integers defined as
\begin{equation*}
  \qint{n}=\ffrac{q^{n}-q^{-n}}{q-q^{-1}}.
\end{equation*}
The Hopf-algebra structure of $\qSL{2|1}$ (the coproduct, the
antipode, and the counit) is given by
\begin{align}
  &\begin{alignedat}{2}
    \Delta(\Fii)&=\Fii\otimes1 + K^{-1}\otimes \Fii,\qquad
    &\Delta(\Eii)&=\Eii\tensor K+\one\tensor\Eii,
    \\
    \Delta(\Fi)&= \Fi\otimes1 + k^{-1}\otimes \Fi,
    &\Delta(\Ei)&=\Ei\tensor k+\one\tensor\Ei,
  \end{alignedat}
  \\
  &\begin{alignedat}{2}
    S(\Fi)&=-k\Fi,\quad S(\Fii)=-K\Fii,    
    \quad&S(\Ei)&=-\Ei k^{-1}, \quad S(\Eii)=-\Eii K^{-1},
  \end{alignedat}
\\
  &\begin{alignedat}{2}
    \epsilon(\Fi)&=0,\quad\epsilon(\Fii)=0,
    &\quad
    \epsilon(\Ei)&=0,\quad \epsilon(\Eii)=0,
  \end{alignedat}
\end{align}
with $k$ and $K$ being group-like.

\subsection{Simple $\qSL{2|1}$ modules}
We consider a subcategory of $\qSL{2|1}$-modules with $k$ eigenvalues
of the form $\q^{-n}$ for $n\in\oZ$.  The subcategory is closed under
tensor products.  The simple finite-dimensional $\qSL{2|1}$-modules
can be labeled as
\begin{gather*}
  \repZ^{\alpha,\beta}_{s,r},\qquad \alpha,\beta=\pm 1,\quad
  s\geq 1,\quad r\in \mathbb{Z}.
\end{gather*}
They have dimensions
\begin{equation}
  \dim\repZ^{\alpha,\beta}_{s,r}=
  \begin{cases}
    2s-1,& r=0,\\
    2s+1,& r=s,\\
    4s,& r\neq0,s.
  \end{cases}
\end{equation}
The modules with $r=0$ and $r=s$ are atypical, and others are typical. In \cite{pal1} it was shown that every finite-dimensional irreducible module over the general linear Lie superalgebra $\GL{n|1}$ can be deformed into an irreducible module over $\qGL{n|1}$. Notations "typical" and "atypical" for modules in the present work are inherited from the theory of Lie superalgebras (see, for example \cite{su}).

\subsubsection{$\qSL{2|1}$-action on simple modules}\label{basis_in_sl21_modules}
We describe (following \cite{TR-reps}) the action of $\qSL{2|1}$ on its simple modules
explicitly, using the basis adapted to the decomposition into
$\qGL{2}$-modules.  Each $\qSL{2|1}$-module decomposes into a direct
sum of simple $\qGL{2}$-modules $\repX^{\alpha,\beta}_{s,r}$, where
$\alpha ,\,\beta=\pm$, $s\geq 1,\, r\in \mathbb{Z}$.
Their dimensions are 
$\dim \repX^{\alpha,\beta}_{s,r} = s$. 
Eigenvalues of generators $K$ and $k$ on the highest weight vector in the module 
$\repX^{\alpha,\beta}_{s,r}$ are 
$\alpha \q^{s-1}$ and $\beta q^{-r}$ correspondingly.  

\begin{description}\addtolength{\itemsep}{6pt}
\item[Atypical modules with $r=0$, \ $\repZold^{\alpha,\beta}_{s,0}$] As
  $\qGL{2}$-modules, these modules decompose as
\begin{equation}\label{aty}
 \repZold^{\alpha,\beta}_{s,0}=
 \repX^{\alpha,\beta}_{s,0}\oplus\repX^{\alpha,\beta}_{s-1,-1},
\end{equation}
and we choose a basis in $\repZold^{\alpha,\beta}_{s,0}$ in accordance
with this decomposition, as
\begin{equation*}
  \left(
    \ZC{\alpha,s;\beta,0}_n\in\repX^{\alpha,\beta}_{s,0}\right)_{
    0\leq n\leq s-1},\qquad
  \left(
    \ZB{\alpha,s;\beta,0}_m\in\repX^{\alpha,\beta}_{s-1,-1}\right)_{
    0\leq m\leq s-2}.
\end{equation*}
The fermionic generators relate these two types of vectors as
\begin{alignat*}{2}
    B \ZC{\alpha,s,\beta,0}_n&=-\qint{n}\ZB{\alpha,s,\beta,0}_{n-1},
    \qquad
    &C\ZB{\alpha,s,\beta,0}_m&=\beta\ZC{\alpha,s,\beta,0}_{m+1}.
  \end{alignat*}

\item[Atypical modules with $s=r$, \ $\repZold^{\alpha,\beta}_{s,s}$]
  The modules decompose as
\begin{equation}\label{atyp}
  \repZold^{\alpha,\beta}_{s,s}=\repX^{\alpha,\beta}_{s,s}\oplus\repX^{\alpha,-\beta}_{s+1,s},
\end{equation}
and we choose a basis in $\repZold^{\alpha,\beta}_{s,s}$ accordingly,
as
\begin{equation*}
  \left(
    \ZC{\alpha,s;\beta,s}_n\in\repX^{\alpha,\beta}_{s,s}\right)_{
    0\leq n\leq s-1},\qquad
  \left(
    \ZB{\alpha,s;\beta,s}_m\in\repX^{\alpha,-\beta}_{s+1,s}\right)_{
    0\leq m\leq s}.
\end{equation*}
The fermions act between these two sets of basis vectors as
\begin{alignat*}{2}
    B \ZC{\alpha,s,\beta,s}_n&=\qint{s-n}\ZB{\alpha,s,\beta,s}_{n},
    \qquad
    &C \ZB{\alpha,s,\beta,s}_m&=\beta\ZC{\alpha,s,\beta,s}_{m}.
  \end{alignat*}

\item[Typical modules ($r\neq0,s$)] The modules decompose as
\begin{equation}\label{ty}
  \repZold^{\alpha,\beta}_{s,r}=\repX^{\alpha,\beta}_{s,r}
  \oplus\repX^{\alpha,-\beta}_{s+1,r}
  \oplus\repX^{\alpha,-\beta}_{s-1,r-1}
  \oplus\repX^{\alpha,\beta}_{s,r-1},
\end{equation}
and we choose the basis in $\repZold^{\alpha,\beta}_{s,r}$ as
\begin{multline*}
  \bigl(\ZC{\alpha,s;\beta,r}_j\bigr)_{0\leq j\leq s-1},
  \ \bigl(\ZU{\alpha,s;\beta,r}_m\bigr)_{0\leq m\leq s},
    \\
  \bigl(\ZD{\alpha,s;\beta,r}_n \bigr)_{0\leq n\leq s-2}, \
  \bigl(\ZB{\alpha,s;\beta,r}_j\bigr)_{0\leq j\leq s-1}.
\end{multline*}
The fermions act on these vectors as
\begin{alignat*}{2}
    B \ZC{\alpha, s, \beta, r}_{j} &= \fffrac{\qint{j}}{\qint{s}}
    \ZD{\alpha, s, \beta, r}_{j - 1} + \beta \fffrac{\qint{r}\qint{s -
        j}}{\qint{s}} \ZU{\alpha, s, \beta, r}_{j},\kern-200pt
    \\
    B \ZU{\alpha, s, \beta, r}_{m} &= \qint{m} \ZB{\alpha, s, \beta,
      r}_{m - 1}, & C \ZU{\alpha, s, \beta, r}_{m} &= \ZC{\alpha, s,
      \beta, r}_{m},
    \\
    B \ZD{\alpha, s, \beta, r}_{n} &= \beta \qint{r}
    \qint{n\!+\!1\!-\!s} \ZB{\alpha, s, \beta, r}_{n}, \quad& C
    \ZD{\alpha, s, \beta, r}_{n} &= \beta \qint{r\!-\!s} \ZC{\alpha,
      s, \beta, r}_{n + 1},
    \\
    C \ZB{\alpha, s, \beta, r}_{j} &= \fffrac{1}{\qint{s}} \ZD{\alpha,
      s, \beta, r}_{j} + \beta \fffrac{\qint{s - r}}{\qint{s}}
    \ZU{\alpha, s, \beta, r}_{j + 1}.\kern-100pt
  \end{alignat*}
\end{description}

\subsection{$\bExti$ spaces for atypical modules}\label{sec:ext}
For two modules $\repZold_1$ and $\repZold_2$, we define
$\Exti(\repZold_2,\repZold_1)$ as a linear space with basis
identified with nontrivial short exact sequences
\begin{equation*}
  0\to\repZold_1\to\repZold_1\oright\repZold_2\to\repZold_2\to0.
\end{equation*}
modulo a certain equivalence relation \cite{Maclane}.

The groups $\Exti$ vanish for the typical $\qSL{2|1}$ modules.  For
the atypical modules, the $\Exti(\repZold_1,\repZold_2)$ group is at
most $1$-dimensional.  Whenever $\Exti(\repZold_1,\repZold_2)$ is
nontrivial, we describe the algebra action in terms of generators: the
action of a $\qSL{2|1}$-generator $A$ on $\repZold_1\oright\repZold_2$
is given by
\begin{equation*}
  \rho_{A}=\rho^{(0)}_{A}+\xi_A,
\end{equation*}
where $\rho^{(0)}_{A}$ is the direct sum of actions of $\qSL{2|1}$-generators on the simple modules and
$\xi_A=\xi^{\repZold_1,\repZold_2}_A:\repZold_1\to\repZold_2$ are
linear maps.

We list the $\xi_A$ maps in terms of the bases introduced above.  The
formulas can be somewhat uniformized by adopting the following
convention for the $1$-dimensional modules
$\repZold^{\alpha,-\beta}_{1,0}$: we denote this module also by $\repZold^{\alpha,\beta}_{0,0}$, with a basis vector
$\ZB{\alpha,0,\beta,0}_0=\ZC{\alpha, 1,-\beta, 0}_{0}$ (and, formally,
with $\ZC{\alpha,0,\beta,0}_m=0$, $m\neq0$).  We then have
\begin{align*}
  \myatop{
    \Exti(\repZold^{\alpha,\beta}_{s,0},\repZold^{\alpha,-\beta}_{s+1,0})=\{b_{s+1}\},}{s\geq 1}\quad&
  \begin{aligned}
    \xi_{\Fi}&:\ZC{\alpha,s,\beta,0}_m\mapsto-\qint{s - m}\ZC{\alpha, s+1 ,-\beta, 0}_m,\\
    \xi_{\Fi}&:\ZB{\alpha,s,\beta,0}_m\mapsto \qint{s - m-1}\ZB{\alpha, s+1, -\beta, 0}_m,
  \end{aligned}
  \\[4pt]
  \myatop{
    \Exti(\repZold^{\alpha,\beta}_{s,0},\repZold^{\alpha,-\beta}_{s-1,0})=\{c_{s-1}\},}{s\geq 2}\quad&
  \begin{aligned}
    \xi_{\Ei}&:\ZC{\alpha,s,\beta,0}_m\mapsto \ZC{\alpha, s-1, -\beta, 0}_m,\\
    \xi_{\Ei}&:\ZB{\alpha,s,\beta,0}_m\mapsto \ZB{\alpha, s-1, -\beta, 0}_m,
  \end{aligned}
\\[4pt]
  \myatop{
    \Exti(\repZold^{\alpha,\beta}_{s,s},\repZold^{\alpha,-\beta}_{s-1,s-1})=\{\bar{b}_{s-1}\},}{s\geq 1}\quad&
  \begin{aligned}
    \xi_{\Fi}&:\ZC{\alpha,s,\beta,s}_m\mapsto-\qint{m}\ZC{\alpha, s-1 ,-\beta, s-1}_{m-1},\\
    \xi_{\Fi}&:\ZB{\alpha,s,\beta,s}_m\mapsto\qint{m}\ZB{\alpha, s-1, -\beta, s-1}_{m-1},
  \end{aligned}
  \\[4pt]
  \myatop{
    \Exti(\repZold^{\alpha,\beta}_{s,s},\repZold^{\alpha,-\beta}_{s+1,s+1})=\{\bar{c}_{s+1}\},}{s\geq 0}\quad&
  \begin{aligned}
    \xi_{\Ei}&:\ZC{\alpha,s,\beta,s}_m\mapsto \ZC{\alpha, s+1, -\beta, s+1}_{m+1},\\
    \xi_{\Ei}&:\ZB{\alpha,s,\beta,s}_m\mapsto \ZB{\alpha, s+1, -\beta, s+1}_{m+1}.
  \end{aligned}
\end{align*}

\subsection{Projective $\qSL{2|1}$-modules}\label{projective}
There are two types of projective $\qSL{2|1}$-modules.

\subsubsection{Simple projective modules}\label{simple-projective}
All simple typical modules described in \bref{basis_in_sl21_modules}
are projective.

\subsubsection{Projective covers of atypical modules}
\label{sec:R} 
We use the notation $\repR^{\alpha,\beta}_{s,0}$ and
$\repR^{\alpha,\beta}_{s,s}$ for projective covers of
$\repZold^{\alpha,\beta}_{s,0}$ and $\repZold^{\alpha,\beta}_{s,s}$
(where, as before, $\alpha,\beta=\pm 1 $ and $s\geq 1$). We
describe the projective covers in terms of Loewy graphs.  
The reconstruction of the $\qSL{2|1}$-action on a projective module from
its Loewy graph is described in detail in \cite[Sec.~6]{TR-reps}.  The
action $\rho_A(\nu)$ of a generator $A$ on a vector $\nu$ has three
parts:
\begin{equation*}
  \rho_A(\nu)=\rho_A^{(0)}(\nu)+ c(\nu) \xi_A(\nu)+\eta_A(\nu),
\end{equation*}
where $\rho_A^{(0)}(\nu)$ is the action of $A$ in the irreducible
subquotient, $\xi_A$ is determined in \bref{sec:ext},
and for the map $\eta_A$ we give explicit formulas after each Loewy graph
(whenever $\eta_A$ is nonzero). Here $c(\nu)$ are some coefficients depending on a pair of simple subquotients in the projective module in question. We write them on edges in Loewy graphs 
(see \cite{TR-reps} for a detailed explanation).

It is convenient to distinguish between two series and two exceptional
cases of projective covers.  The first series is
$\repR^{\alpha,\beta}_{s,0}$, $s\geq2$, with the Loewy graph
\begin{equation}\label{proj1}  
\xymatrix@=30pt{
    && {\repZ^{\alpha,\beta}_{s,0}}
    \ar[dl]_{-\qint{s-1}}
    \ar[dr]^{-\qint{s}}
    &\\
    &  {\repZ^{\alpha,-\beta}_{s+1,0}}
    \ar[dr]_{1}
    &
    & {\repZ^{\alpha,-\beta}_{s-1,0}}
    \ar[dl]^{1}
    \\
    && {\repZ^{\alpha,\beta}_{s,0}}
    &
  }
\end{equation}
where
\begin{equation*}
\eta_B:\ZC{\alpha,s,\beta,0}_n\TOP \mapsto
-\beta\qint{n}\ZB{\alpha, s ,\beta, 0}_{n-1}\BOT.
\end{equation*}
Here $v\TOP$ denotes the vector $v$ from the top subquotient, and $v\BOT$ denotes vector $v$ from the bottom subquotient.

The second series is $\repR^{\alpha,\beta}_{s,s}$, $s\geq 2$, with the
Loewy graph
\begin{equation} \label{proj2} 
\xymatrix@=30pt{
    && {\repZ^{\alpha,\beta}_{s,s}}
    \ar[dl]_{-\qint{s}}
    \ar[dr]^{-\qint{s+1}}
    &\\
    &  {\repZ^{\alpha,-\beta}_{s+1,s+1}}
    \ar[dr]_{1}
    &
    & {\repZ^{\alpha,-\beta}_{s-1,s-1}}
    \ar[dl]^{1}
    \\
    && {\repZ^{\alpha,\beta}_{s,s}}
    &
  }
\end{equation}

and with
\begin{equation*}
\eta_C:\ZB{\alpha,s,\beta,s}_n\TOP
\mapsto\ZC{\alpha, s ,\beta, s}_{n}\BOT.
\end{equation*}

The two exceptional cases are $\repR^{\alpha,\beta}_{1,0}$ and
$\repR^{\alpha,\beta}_{1,1}$, with the respective Loewy graphs
\begin{equation}\label{proj3}
  \mbox{}\kern-30pt
  \xymatrix@=30pt{
    && {\repZ^{\alpha,\beta}_{1,0}}
    \ar[dl]_{1}
    \ar[dr]^{-1}
    &\\
    &  {\repZ^{\alpha,-\beta}_{2,0}}
    \ar[dr]_{1}
    &
    & {\repZ^{\alpha,\beta}_{1,1}}
    \ar[dl]^{1}
    \\
    && {\repZ^{\alpha,\beta}_{1,0}}
    &
  }
\xymatrix@=30pt{
    && {\repZ^{\alpha,\beta}_{1,1}}
    \ar[dl]_{-\qint{1}}
    \ar[dr]^{-\qint{2}}
    &\\
    &  {\repZ^{\alpha,-\beta}_{2,2}}
    \ar[dr]_{1}
    &
    & {\repZ^{\alpha,\beta}_{1,0}}
    \ar[dl]^{1}
    \\
    && {\repZ^{\alpha,\beta}_{1,1}}
    &
  }
\end{equation}
These modules have dimensions
\begin{align*}
  \dim\repR^{\alpha,\beta}_{s,0} &= 8s-4, \quad s>1, \\
  \dim\repR^{\alpha,\beta}_{s,s} &= 8s+4, \quad s \geq 1, \\
  \dim\repR^{\alpha,\beta}_{1,0} &= 8.
\end{align*}

\section{The mixed tensor product}\label{sec:chain}
We study the mixed tensor product (``spin-chain'')~\eqref{chain},
where $\;\three=\repZ^{1,-1}_{1,1}$ and
$\;\bthree=\repZ^{1,1}_{2,0}$ are the two
three-dimensional simple $\qSL{2|1}$-modules.  We are interested in decomposing
$\Chain_{m,n}$ as a bimodule over $\qSL{2|1}$ and its
centralizer $\Alg_{m,n}$.  As a necessary first step, we
decompose tensor products of relevant $\qSL{2|1}$-modules with the
fundamental modules $\repZ^{\alpha,\beta}_{1,1}$ and
$\repZ^{\alpha,\beta}_{2,0}$.

\begin{Thm}\label{tp-decomp}
  Tensor products $\repZold\tensor\repZold^{\alpha,\beta}_{1,1}$,
  where $\repZold$ ranges the atypical and typical simple modules and
  their projective covers, decompose as follows: 
  \begin{align*}
    &\repZ^{\alpha_1,\beta_1}_{s,0}\tensor\repZold^{\alpha_2,\beta_2}_{1,1}
    =\repZ^{\alpha_{12},-\beta_{12}}_{s-1,0}+
      \repZ^{\alpha_{12},\beta_{12}}_{s,1}, \quad s \geq 2,
      \\
    &\repZ^{\alpha_1,\beta_1}_{s,s} \tensor\repZold^{\alpha_2,\beta_2}_{1,1}
    =  \repZold^{\alpha_{12},-\beta_{12}}_{s+1,s+1}+
      \repZold^{\alpha_{12},\beta_{12}}_{s,s+1}, \quad s \geq 1, 
    \\
&\left. \begin{array}{l}
    \repZold^{\alpha_1,\beta_1}_{s,r} \tensor\repZold^{\alpha_2,\beta_2}_{1,1}
    =
      \begin{cases} 
        \repR^{\alpha_{12},-\beta_{12}}_{s+1,0}+ \repZold^{\alpha_{12},-\beta_{12}}_{s-1,-1}, & r=-1, \\
        \repR^{\alpha_{12},-\beta_{12}}_{s-1,s-1}+ \repZold^{\alpha_{12},-\beta_{12}}_{s+1,s}, & r=s-1, \\
        \repZold^{\alpha_{12},\beta_{12}}_{s,r+1}+\repZold^{\alpha_{12},-\beta_{12}}_{s+1,r+1}+\repZold^{\alpha_{12},-\beta_{12}}_{s-1,r}
        &
        \text{otherwise},
      \end{cases}
\end{array} \right\}
	\qquad
	s \geq 2
    \\
    \intertext{and}
    &\repR^{\alpha_{1},\beta_{1}}_{s,0} \tensor \repZold^{\alpha_2,\beta_2}_{1,1}
    = \repR^{\alpha_{12},-\beta_{12}}_{s-1,0} +
      2\repZold^{\alpha_{12},\beta_{12}}_{s,1} +
      \repZold^{\alpha_{12},-\beta_{12}}_{s-1,1} +
      \repZold^{\alpha_{12},-\beta_{12}}_{s+1,1}, \quad s \geq 3,
    \\
    &\repR^{\alpha_{1},\beta_{1}}_{s,s} \tensor \repZold^{\alpha_2,\beta_2}_{1,1}
    = \repR^{\alpha_{12},-\beta_{12}}_{s+1,s+1} +
      2\repZold^{\alpha_{12},\beta_{12}}_{s,s+1} +
      \repZold^{\alpha_{12},-\beta_{12}}_{s-1,s} +
      \repZold^{\alpha_{12},-\beta_{12}}_{s+1,s+2}, \quad s \geq 2,
  \end{align*}
  where we write $\alpha_{12}=\alpha_1\alpha_2$ and
  $\beta_{12}=\beta_1\beta_2$. \\
The exceptional cases are listed below:
  \begin{align*}
    \repZ^{\alpha_1,\beta_1}_{1,0}\tensor\repZ^{\alpha_2,\beta_2}_{1,1}
    &=\repZ^{\alpha_{12},\beta_{12}}_{1,1}, 
      \\ 
    \repZ^{\alpha_1,\beta_1}_{1,-1}\tensor\repZ^{\alpha_2,\beta_2}_{1,1}
    &=\repR^{\alpha_{12},-\beta_{12}}_{2,0},  
       \\ 
    \repZ^{\alpha_1,\beta_1}_{1,r}\tensor\repZ^{\alpha_2,\beta_2}_{1,1}
    &=\repZ^{\alpha_{12},\beta_{12}}_{1,r+1} + \repZ^{\alpha_{12},-\beta_{12}}_{2,r+1}, 
    \quad r\ne -1,0,1,
      \\ 
    \repR^{\alpha_1,\beta_1}_{2,0}\tensor\repZ^{\alpha_2,\beta_2}_{1,1}
    &=\repR^{\alpha_{12},-\beta_{12}}_{1,0} + 2 \repZ^{\alpha_{12},\beta_{12}}_{2,1} 
    + \repZ^{\alpha_{12},-\beta_{12}}_{3,1},
	\\      
    \repR^{\alpha_1,\beta_1}_{1,0}\tensor\repZ^{\alpha_2,\beta_2}_{1,1}
    &=\repR^{\alpha_{12},\beta_{12}}_{1,1} + \repZ^{\alpha_{12},\beta_{12}}_{1,2} 
	+ \repZ^{\alpha_{12},-\beta_{12}}_{2,1},
	\\
    \repR^{\alpha_1,\beta_1}_{1,1}\tensor\repZ^{\alpha_2,\beta_2}_{1,1}
    &=\repR^{\alpha_{12},-\beta_{12}}_{2,2} + 2 \repZ^{\alpha_{12},\beta_{12}}_{1,2} 
	+ \repZ^{\alpha_{12},-\beta_{12}}_{2,3}.
  \end{align*}

 The tensor products
  $\repZold\tensor\repZold^{\alpha,\beta}_{2,0}$ decompose as:
  \begin{align*}
&\left. \begin{array}{l}
   \repZold^{\alpha_{1},\beta_{1}}_{s,0} \tensor\repZold^{\alpha_2,\beta_2}_{2,0}
    = \repZold^{\alpha_{12},\beta_{12}}_{s+1,0}+
      \repZold^{\alpha_{12},\beta_{12}}_{s-1,-1} ,
    \\
    \repZold^{\alpha_{1},\beta_{1}}_{s,s} \tensor\repZold^{\alpha_2,\beta_2}_{2,0}
    =
      \repZold^{\alpha_{12},\beta_{12}}_{s-1,s-1}+
      \repZold^{\alpha_{12},\beta_{12}}_{s+1,s} ,
    \\
    \repZold^{\alpha_{1},\beta_{1}}_{s,r} \tensor\repZold^{\alpha_2,\beta_2}_{2,0}
    =
      \begin{cases}
        \repR^{\alpha_{12},-\beta_{12}}_{s,0}+
        \repZold^{\alpha_{12},\beta_{12}}_{s+1,1}, & r=1,
        \\
        \repR^{\alpha_{12},-\beta_{12}}_{s,s}+
        \repZold^{\alpha_{12},\beta_{12}}_{s-1,s}, & r=s+1,
        \\
         \repZold^{\alpha_{12},\beta_{12}}_{s+1,r}+\repZold^{\alpha_{12},-\beta_{12}}_{s,r-1}+\repZold^{\alpha_{12},\beta_{12}}_{s-1,r-1},
         &
         \text{otherwise},
       \end{cases}
\end{array} \right\}
	\qquad
	s \geq 2
    \\
    \intertext{and}
    &\repR^{\alpha_{1},\beta_{1}}_{s,0} \tensor \repZold^{\alpha_2,\beta_2}_{2,0}
    = \repR^{\alpha_{12},\beta_{12}}_{s+1,0} +
      2\repZold^{\alpha_{12},\beta_{12}}_{s-1,-1}+
      \repZold^{\alpha_{12},-\beta_{12}}_{s,-1}  +
      \repZold^{\alpha_{12},-\beta_{12}}_{s-2,-1}, \quad s \geq 3,
    \\
    &\repR^{\alpha_{1},\beta_{1}}_{s,s} \tensor \repZold^{\alpha_2,\beta_2}_{2,0}
    = \repR^{\alpha_{12},\beta_{12}}_{s-1,s-1} +
      2\repZold^{\alpha_{12},\beta_{12}}_{s+1,s}+
      \repZold^{\alpha_{12},-\beta_{12}}_{s+2,s+1}  +
      \repZold^{\alpha_{12},-\beta_{12}}_{s,s-1}, \quad s \geq 2.
  \end{align*}
The exceptional cases are:
  \begin{align*}
    \repZ^{\alpha_1,\beta_1}_{1,0}\tensor\repZ^{\alpha_2,\beta_2}_{2,0}
    &=\repZ^{\alpha_{12},\beta_{12}}_{2,0},
      \\ 
    \repZ^{\alpha_1,\beta_1}_{1,1}\tensor\repZ^{\alpha_2,\beta_2}_{2,0}
    &=\repZ^{\alpha_{12},-\beta_{12}}_{1,0} + \repZ^{\alpha_{12},\beta_{12}}_{2,1},
      \\ 
    \repZ^{\alpha_1,\beta_1}_{1,2}\tensor\repZ^{\alpha_2,\beta_2}_{2,0}
    &=\repR^{\alpha_{12},-\beta_{12}}_{1,1},
      \\ 
    \repZ^{\alpha_1,\beta_1}_{1,r}\tensor\repZ^{\alpha_2,\beta_2}_{2,0}
    &=\repZ^{\alpha_{12},-\beta_{12}}_{1,r-1} + \repZ^{\alpha_{12},\beta_{12}}_{2,r}, 
    \quad r\ne 0,1,2,
      \\ 
    \repR^{\alpha_1,\beta_1}_{2,0}\tensor\repZ^{\alpha_2,\beta_2}_{2,0}
    &=\repR^{\alpha_{12},\beta_{12}}_{3,0} + 2 \repZ^{\alpha_{12},\beta_{12}}_{1,-1} 
    + \repZ^{\alpha_{12},-\beta_{12}}_{2,-1},
	\\      
    \repR^{\alpha_1,\beta_1}_{1,0}\tensor\repZ^{\alpha_2,\beta_2}_{2,0}
    &=\repR^{\alpha_{12},\beta_{12}}_{2,0} + \repZ^{\alpha_{12},-\beta_{12}}_{1,-1} 
    + \repZ^{\alpha_{12},\beta_{12}}_{2,1},
    \\
    \repR^{\alpha_1,\beta_1}_{1,1}\tensor\repZ^{\alpha_2,\beta_2}_{2,0}
    &=\repR^{\alpha_{12},-\beta_{12}}_{1,0} + 2 \repZ^{\alpha_{12},\beta_{12}}_{2,1} 
	+ \repZ^{\alpha_{12},-\beta_{12}}_{3,2}.    
  \end{align*}
\end{Thm}

\subsubsection{}It follows, in particular, that the set of simple
modules and their projective covers is closed under tensor product
decompositions.

\begin{proof} 
We discuss two cases:
$\repZold^{\alpha_1,\beta_1}_{s,s}\tensor\repZold^{\alpha_2,\beta_2}_{1,1}$
and
$\repZold^{\alpha_1,\beta_1}_{s,s-1}\tensor\repZold^{\alpha_2,\beta_2}_{1,1}$. 
Other cases are similar.

We consider the $\qSL{2|1}$-modules in the left-hand side of the
tensor product as $\qGL{2}$-modules (as explained
in~\bref{basis_in_sl21_modules}) and calculate their tensor product
using the results in~\cite{BGT}.  For the tensor product
$\repZold^{\alpha_1,\beta_1}_{s,s}\tensor\repZold^{\alpha_2,\beta_2}_{1,1}$,
we have
\begin{align}\label{6X}
  \repZold^{\alpha_1,\beta_1}_{s,s}\tensor
  \repZold^{\alpha_2,\beta_2}_{1,1}&=\left(\repX^{\alpha_1,\beta_1}_{s,s}\oplus
  \repX^{\alpha_1,-\beta_1}_{s+1,s}\right)\tensor\left(\repX^{\alpha_2,\beta_2}_{1,1}\oplus
   \repX^{\alpha_2,-\beta_2}_{2,1}\right)
  \\
  \notag
  &=\repX^{\alpha_{12}\beta_{12}}_{s,s+1}\oplus\repX^{\alpha_{12},-\beta_{12}}_{s+1,s+1}
  \oplus\repX^{\alpha_{12},-\beta_{12}}_{s-1,s}
  \oplus\repX^{\alpha_{12},-\beta_{12}}_{s+1,s+1}\oplus
  \repX^{\alpha_{12},\beta_{12}}_{s+2,s+1}\oplus\repX^{\alpha_{12},\beta_{12}}_{s,s}.
\end{align}
Decomposition~\eqref{6X} contains six $\qGL{2}$-modules. Taking into
account that a typical module contains four $\qGL{2}$-summands and an
atypical one contains two, the module in~\eqref{6X} can be the direct
sum of either three atypical $\qSL{2|1}$-modules or one typical and
one atypical module.  Explicitly writing the decompositions of
possible $\qSL{2|1}$-modules shows that there exists only one
$\qSL{2|1}$-module that has the decomposition~\eqref{6X}. The
second and the fifth summands can be combined into
$\repZold^{\alpha_{12},-\beta_{12}}_{s+1,s+1}$ and the other four
summands give $\repZold^{\alpha_{12},\beta_{12}}_{s,s+1}$.  Thus, we
have
\begin{gather*}
  \repZold^{\alpha_1,\beta_1}_{s,s}\tensor
  \repZold^{\alpha_2,\beta_2}_{1,1} 
  =\repZold^{\alpha_{12},-\beta_{12}}_{s+1,s+1}\oplus\repZold^{\alpha_{12},\beta_{12}}_{s,s+1}.
\end{gather*}

We next consider the product
$\repZold^{\alpha_1,\beta_1}_{s,s-1}\tensor\repZold^{\alpha_2,\beta_2}_{1,1}$.
The $\qGL{2}$-decomposition is
\begin{align}
  \label{12X} 
  \repZold^{\alpha_1,\beta_1}_{s,s-1}\tensor
  \repZold^{\alpha_2,\beta_2}_{1,1}=&\left(\repX^{\alpha,\beta}_{s,s-1}
  \oplus\repX^{\alpha_1,-\beta_1}_{s+1,s-1}
  \oplus\repX^{\alpha_1,-\beta_1}_{s-1,s-2}
  \oplus\repX^{\alpha_1,\beta_1}_{s,s-2}\right)\tensor\left(\repX^{\alpha_2,\beta_2}_{1,1}\oplus\repX^{\alpha_2,-\beta_2}_{2,1}\right)
  \\
  \notag
  =&\repX^{\alpha_{12},\beta_{12}}_{s,s}
  \oplus\repX^{\alpha_{12},-\beta_{12}}_{s+1,s}
  \oplus\repX^{\alpha_{12},-\beta_{12}}_{s-1,s-1}
  \oplus\repX^{\alpha_{12},\beta_{12}}_{s,s-1}\oplus\repX^{\alpha_{12},-\beta_{12}}_{s+1,s}
     \oplus\repX^{\alpha_{12},\beta_{12}}_{s+2,s}
  \\
  \notag
  &\oplus\repX^{\alpha_{12},\beta_{12}}_{s,s-1}
  \oplus\repX^{\alpha_{12},-\beta_{12}}_{s+1,s-1}\oplus\repX^{\alpha,-\beta}_{s-1,s-1}
  \oplus\repX^{\alpha_{12},\beta_{12}}_{s,s-1}
  \oplus\repX^{\alpha_{12},\beta_{12}}_{s-2,s-2}
  \oplus\repX^{\alpha_{12},-\beta_{12}}_{s-1,s-2}.
\end{align}
Because $\repZold^{\alpha_1,\beta_1}_{s,s-1}$ is a projective simple
module (see~\bref{simple-projective}), the decomposition of
$\repZold^{\alpha_1,\beta_1}_{s,s-1}~\tensor~\repZold^{\alpha_2,\beta_2}_{1,1}$
involves only projective modules, which, as we recall
from~\bref{sec:R}, consist of all typical simple modules and the
$\repR^{\alpha,\beta}_{s,r}$.  There are several $\qSL{2|1}$-modules
that have the $\qGL{2}$-decomposition~\eqref{12X}, but only one of
them is projective.\footnote{For example, the direct sum of simple
  $\qSL{2|1}$-modules
  $2\repZold^{\alpha_{12},-\beta_{12}}_{s-1,s-1}\oplus\repZold^{\alpha_{12},\beta_{12}}_{s,s}
  \oplus\repZold^{\alpha_{12},\beta_{12}}_{s-2,s-2}
  \oplus\repZold^{\alpha_{12},-\beta_{12}}_{s+1,s}$
  is compatible with the $\qGL{2}$-decomposition \eqref{12X}, but is not a
  projective $\qSL{2|1}$-module.}  Thus, we have
\begin{equation*}
  \repZold^{\alpha_1,\beta_1}_{s,s-1}\tensor
  \repZold^{\alpha_2,\beta_2}_{1,1}=\repR^{\alpha_{12},-\beta_{12}}_{s-1,s-1}
  \oplus\repZold^{\alpha_{12},-\beta_{12}}_{s+1,s}.
\end{equation*}

The cases
$\repR^{\alpha_{1},\beta_{1}}_{s,0}\tensor\repZold^{\alpha_2,\beta_2}_{1,1}$
and
$\repR^{\alpha_{1},\beta_{1}}_{s,s}\tensor\repZold^{\alpha_2,\beta_2}_{1,1}$
are worked out similarly.  We consider $\qGL{2}$-decompositions of
both tensorands and calculate tensor products of $\qGL{2}$-modules. This gives a long direct sum of simple and projective
$\qGL{2}$-modules that each time are combined uniquely into a sum of
projective $\qSL{2|1}$-modules.
\end{proof}

\begin{rem}
Decomposition of all tensor products of finite dimensional $\SL{2|1}$-representations into their indecomposable building blocks was found in \cite{gotz}. 
\end{rem}

\subsubsection{}
We calculate decomposition of $\Chain_{m,n}$ iteratively using Theorem \ref{tp-decomp}. The multiplicities of $\qSL{2|1}$-indecomposable modules are dimensions of 
$\Alg_{m,n}$-modules, which we discuss below.

\subsection{The centralizer of $\qSL{2|1}$ on the mixed tensor product}\label{qwB-on-chain}
We fix bases in the $\three$ and $\bthree$ modules in accordance
with~\bref{basis_in_sl21_modules} and introduce a shorthand
notation for them:
\begin{alignat*}{3}
  f_1&=\ZC{1,1;-1,1}_0,&\quad
  f_2&=\ZB{1,1;-1,1}_0,&\quad
  f_3&=\ZB{1,1;-1,1}_1,\\
  v_1&=\ZB{1,2;1,0}_0,&\quad
  v_2&=\ZC{1,2;1,0}_1,&\quad
  v_3&=\ZC{1,2;1,0}_0.
\end{alignat*}
In the tensor products of two $\qSL{2|1}$ modules, we then have the
operators
\begin{gather*}
  g: \three \tensor \three \mapsto   \three \tensor \three,\qquad
  \EE: \three \tensor \bthree \mapsto   \three \tensor \bthree,\qquad
  h: \bthree \tensor \bthree \mapsto   \bthree \tensor \bthree,
\end{gather*}
that commute with $\qSL{2|1}$ and are explicitly given by
\begin{multline*}
  g:
  \begin{pmatrix}
    f_1\tensor f_1 &
    f_1\tensor f_2 & 
    f_1\tensor f_3
    \\
    f_2\tensor f_1 &
    f_2\tensor f_2 & 
    f_2\tensor f_3
    \\
    f_3\tensor f_1 &
    f_3\tensor f_2 & 
    f_3\tensor f_3
  \end{pmatrix}\mapsto
  \\
  {\textstyle\begin{pmatrix}
    \q^{-2}f_1\tensor f_1&
    -\q^{-1}f_2\tensor f_1&
    -\q^{-1}f_3\tensor f_1
    \\
    (\q^{-2}\!-\!1) f_2\tensor f_1 - \q^{-1} f_1\tensor f_2  &
    -f_2\tensor  f_2&
     -\q^{-1} f_3\tensor  f_2
    \\
    (\q^{-2}\!-\!1) f_3\tensor   f_1 - \q^{-1}f_1\tensor f_3&
    (\q^{-2}\!-\!1) f_3\tensor f_2 - \q^{-1} f_2\tensor f_3&
     -f_3\tensor f_3
  \end{pmatrix}},
\end{multline*}
\begin{gather*}
  \EE:
  \begin{pmatrix}
    f_1\tensor v_1 &
    f_1\tensor v_2 & 
    f_1\tensor v_3
    \\
    f_2\tensor v_1 &
    f_2\tensor v_2 & 
    f_2\tensor v_3
    \\
    f_3\tensor v_1 &
    f_3\tensor v_2 & 
    f_3\tensor v_3
  \end{pmatrix}\mapsto
  {\textstyle\begin{pmatrix}
    1& 0& 0 \\
    0& -\q& 0 \\
	0& 0& 1
  \end{pmatrix} \cdot ( \q^2 f_1 \tensor v_1 + \q f_2 \tensor v_2 - f_3 \tensor v_3)  
  },
\end{gather*}%
and
\begin{multline*}
  h:
  \begin{pmatrix}
    v_1\tensor v_1 &
    v_1\tensor v_2 & 
    v_1\tensor v_3
    \\
    v_2\tensor v_1 &
    v_2\tensor v_2 & 
    v_2\tensor v_3
    \\
    v_3\tensor v_1 &
    v_3\tensor v_2 & 
    v_3\tensor v_3
  \end{pmatrix}\mapsto
  \\
  {\textstyle\begin{pmatrix}
    \q^{-2}v_1\tensor  v_1 &
	(\q^{-2}\!-\!1) v_1\tensor v_2 - \q^{-1} v_2\tensor v_1 &
	(\q^{-2}\!-\!1) v_1\tensor v_3 - \q^{-1} v_3\tensor v_1 &
    \\
    - \q^{-1} v_1\tensor v_2 &
    -v_2\tensor v_2 &  
    (\q^{-2}\!-\!1) v_2\tensor v_3 - \q^{-1} v_3\tensor v_2 
    \\
	-\q^{-1} v_1\tensor v_3 &
    -\q^{-1} v_2\tensor v_3 &
    -v_3\tensor  v_3
  \end{pmatrix}}.
\end{multline*}

On $\Chain_{m,n}$, we define the operators
\begin{align*}
  g_j &= \underbracket{\one\tensor\dots\tensor\one}_{m-j-1}\tensor g\tensor \underbracket{\one \tensor \dots \tensor \one}_{n+j-1}, \\ 
  h_i&=\underbracket{\one\tensor\dots\tensor\one}_{m+i-1}\tensor h\tensor \underbracket{\one \tensor \dots \tensor \one}_{n-i-1},\\
  \EE&= \underbracket{\one\tensor\dots\tensor\one}_{m-1}\tensor \EE\tensor \underbracket{\one \tensor \dots \tensor \one}_{n-1}.
\end{align*}
These are the generators of a quantum walled Brauer algebra, which we
discuss in the next subsection.

\subsection{The quantum walled Brauer algebra}
\subsubsection{}\label{qwb_relations}
The algebra $\qwB_{m,n}$ is the associative unital algebra generated
by $g_i$, $\EE$, $h_j$, where $1\leq i \leq m-1$ and
$1\leq j \leq n-1$, with relations (see \cite{Leduc,Halv,KM1993})
\begin{gather*}
 g_i h_j = h_j g_i,
 \\
 (g_i-\gamma)(g_i - \delta) =0, \qquad 
 (h_j-\gamma)(h_j -\delta) = 0,
 \\
 g_i g_j = g_j g_i, \quad |i-j|>1, \qquad
 h_i h_j = h_j h_i, \quad |i-j|>1,
 \\
 g_i g_{i+1} g_i = g_{i+1} g_i g_{i+1}, \qquad
 h_i h_{i+1} h_i = h_{i+1} h_i h_{i+1},
\end{gather*}
\begin{gather*}
  \EE \EE = \ffrac{\theta +1}{ \gamma +\delta}\EE,
  \\
  \EE g_1 \EE = \EE,\qquad
  \EE h_1 \EE = \EE,
  \\
  \EE g_i=g_i \EE,\quad 2\leq i\leq m-1,\qquad
  \EE h_j=h_j \EE,\quad 2\leq j\leq n-1,
  \\
  \EE  g_1  h_1^{-1} \EE  (g_1 - h_1)=0,\qquad
  (g_1 - h_1) \EE  g_1  h_1^{-1}  \EE = 0.
\end{gather*}
These relations involve complex parameters $\gamma$, $\delta$,
and $\theta$, and we sometimes use the notation
$\qwB_{m,n}(\gamma,\delta,\theta)$ for the algebra, although
one parameter can be eliminated from the relations by renormalizing
the generators. We write the relations in the present form for more
convenient comparison with different choices in literature.

\begin{rem}
The algebra $\qwB_{m,n}$ has a presentation by tangle diagrams, see \cite{ST2015}. 
\end{rem}

\begin{rem}
In \cite{Stroppel12, CDM2007, Cox2010} the one-parameter walled Brauer algebra is discussed.
It can be considered as a classical limit of quantum walled Brauer algebra $\qwb_{m,n}$.
To get this limit from the algebra with relations \bref{qwb_relations} we can do the following. 
By renormalization of the generators, parameter $\gamma$ can be set to $\gamma=-1$.
We introduce a complex parameter~$r$:
\begin{equation*}
\theta = - \delta^r
\end{equation*}
so that the relation reads
$\EE \EE = - \ffrac{\delta^r -1}{\delta-1}\EE$.
Then we consider the limit $\delta \to 1$. The dependent on parameters algebra relations become 
\begin{align*}
&g_i^2=h_i^2=1, \\
&\EE \EE = - r\EE.
\end{align*}
Such an algebra is called the (classical) walled Brauer algebra with (only one) parameter $r$. We use the notation $\wb_{m,n}(-r)$ for it. 
\end{rem}

\begin{thm}
The generators defined in~\bref{qwB-on-chain} satisfy the $\qwB_{m,n}$
relations with the parameters
\begin{align}\label{ChainNichRules}
&\gamma = -1,  \nonumber \\ 
&\delta  =  \q^{-2}, \\
&\theta = -\q^{-2}. \nonumber
\end{align}
\end{thm}

\begin{rem}
By choice of normalization in matrices, the parameters $\gamma$ and $\delta$ can be changed, however the relation 
\begin{equation}\label{NichRules}
    \theta = \ffrac{\delta}{\gamma}
\end{equation}
remains invariant. This relation means that we consider a degenerate case in which the algebra becomes non-semisimple as we discuss below.
\end{rem}

\begin{cor}
The endomorphism algebra of $\,\qSL{2|1}$-module $\Chain_{m,n}$ is isomorphic to the quotient of the
algebra $\qwb_{m,n}$ with special parameters (\ref{ChainNichRules}).
\end{cor}

One can consider an algebra $\qSL{M|N}$ for arbitrary positive integers $M$ and $N$.
Let $V$ and $V^\star$ be fundamental representation of $\qSL{M|N}$ and 
its dual. 
We let $\Alg^{M,N}_{m,n}$ denote the algebra of endomorphisms of $\qSL{M|N}$ on mixed tensor product 
${V^\star}^{\tensor m} \tensor V^{\tensor n}$.
As was shown in \cite{ShaderMoon07}
(see also \cite{DDS-1, DDS-2, ShaderMoon02, Stroppel12})
there is a surjective homomorphism
\begin{equation}
\Psi^{M,N}_{m,n}: 
\qwb_{m,n}(\gamma = -1, \delta  =  \q^{-2}, \theta = -\q^{-2(M-N)}) \to \Alg^{M,N}_{m,n}.
\end{equation}
Here the parameter $\q$ is the same as in the algebra $\qSL{M|N}$.
In the classical limit we conclude that the algebra of endomorphisms of $\SL{M|N}$ on mixed tensor product of its fundamental representations is a quotient of the algebra $\wb_{m,n}(r)$ with $r=N-M$. This is consistent with the results of 
\cite{ShaderMoon02, Stroppel12} because classical algebras $\wb_{m,n}(r)$ and $\wb_{m,n}(-r)$ are isomorhic to each other. Indeed, the isomorphism is given by the formulas 
$g_i' = -g_i$, ${h_j' = -h_j}$ and $\EE' = -\EE$.

We note that for ${N=0}$ the algebra $\Alg^{M,0}_{m,n}$ is semisimple
and $\mathrm{ker}\, \Psi^{M,0}_{m,n}$ contains the whole radical of 
$\qwb_{m,n}$, see \cite{Halv}.

At the end of this section we formulate two statements important for the sequel. 

\begin{Conj1}\label{conj_1}
Representation categories of the algebra $\qwB_{m,n}$ with generic values of parameter 
$\ffrac{\delta}{\gamma}$ and of the (classical) walled Brauer algebra are equivalent as abelian categories.
\end{Conj1}

The walled Brauer algebra has quasihereditary structure, see \cite{Cox2010}. According to our first conjecture we suppose $\qwB_{m,n}$ with generic values of the parameter $\ffrac{\delta}{\gamma}$ to be also quasihereditary.

In the following sections we consider only the case $M=2$, $N=1$ and use the notation 
$\Alg_{m,n}$ for $\Alg^{2,1}_{m,n}$. The second important statement is (see also \cite{StollWerth2014}) 
\begin{Conj2}\label{conj_2}
The algebra $\Alg_{m,n}$ is quasihereditary.\footnote{The conjecture about quasihereditary structure in the general case $\Alg^{M,N}_{m,n}$ can apparently be formulated but is beyond the scope of this paper.}
\end{Conj2}

\section{Modules over $\qwB_{m,n}$ and $\Alg_{m,n}$}\label{sec:modules}
In this section we describe Specht and simple modules for $\qwB_{m,n}$ and simple and projective modules for algebra $\Alg_{m,n}$.

\subsection{$\qwB_{m,n}$ Specht modules}

\subsubsection{}
A finite integer sequence 
$\mu = (\mu_1, \mu_2, \dots \mu_r)$ is called a partition, if 
$\mu_1 \geq \mu_2 \geq \dots \mu_r >0$.

A bipartition is a pair of partitions
$\la=(\la^L, \la^R)$.  Let~$\Lambda$ be the set of all bipartitions.
For each integer $0 \leq f \leq \min(m,n)$, we set
\begin{equation}
  \Lambda_{m,n}(f) = \{ \la \in \Lambda \mid  m-|\la^L| = n-|\la^R| = f \},
\end{equation}
where $|\lambda|$ is the sum of elements of a partition, and
\begin{equation}
\Lambda_{m,n} = \bigcup_{f=0}^{\min(m,n)} \, \Lambda_{m,n}(f).
\end{equation}

The set $\Lambda_{m,n}$ is in bijective correspondence with the set of
$\qwB_{m,n}$ Specht modules~\cite{En}.
We let $S(\la)$ denote the $\qwb_{m,n}$-Specht module
corresponding to the bipartition~$\lambda$.

The following claim is given in \cite{Stroppel12}
\begin{thm}\label{simpleQWBmodules}
  For generic values of the $\qwB_{m,n}$ parameters, each Specht
  module is simple, and the sets of Specht and simple modules
  coincide.
\end{thm}

\subsection{Modules over $\qwb_{m,n}$ with special parameters}\label{modulqwb}
We now consider the cathegory of $\qwb_{m,n}$ modules with the parameters
related as in (\ref{NichRules}).  The algebra is then nonsemisimple,
and some of the Specht modules $S(\la)$ become reducible.  

Let $D(\la)$ and $K(\la)$ be the simple head and the projective cover for $S(\la)$. 
Below we also use the notation $D\big[\la^L,\la^R \big]$ and $K\big[\la^L,\la^R \big]$ for
$D((\la^L,\la^R))$ and $K((\la^L,\la^R))$ respectively.  

In \cite[Theorem~2.7]{CDM2007} the full classification of simple modules over the walled Brauer algebra is given. Thus, assuming the Conjecture 1 \eqref{conj_1} (but see also \cite[Theorem~8.1]{En}) we have the following
\begin{lemma} 
  If $\EE \EE \ne 0$, the modules $D(\la)$, $\la \in \Lambda_{m,n}$ give a complete set of
  simple modules for the algebra $\qwb_{m,n}$.
\end{lemma}
The decomposition multiplicities $d_{\la,\mu}=\big[S(\mu):D(\la)\big]$ for the $S(\la)$-modules in terms of their simple subquotients are determined  in~\cite{Cox2010}.
Because of the quasihereditary structure of $\qwb_{m,n}$ each projective module $K(\la)$ has a filtration by Specht modules. Let $\tilde{d}_{\la\mu}~=~\big[K(\la):S(\mu)\big]$ be the multiplicity of a given Specht module $S(\la)$ in the filtration; then, by the Brauer-Humphreys reciprocity  (see~\cite{Cox2010} and references therein)
\begin{equation}\label{B-H_rec}
\tilde{d}_{\la\mu}=d_{\la\mu}.
\end{equation}
 We use this statement to construct projective modules for $\Alg_{m,n}$ in the next subsection.

\subsection{Modules in the decomposition of the mixed tensor product}
As a $\Alg_{m,n}\boxtimes\qSL{2|1}$-bimodule, the mixed tensor product $\Chain_{m,n}$ decomposes into a direct sum of indecomposable bimodules. 
\begin{dfn}
For non-negative integers $p,q$, a partition $\mu$ is called a $(p,q)$-hook partition if it doesn't contain a box in the $(p+1,q+1)$-position, i.e. $\mu_{p+1}<q+1$.
\end{dfn}
\begin{dfn}
(see \cite{CW11}) 
For non-negative integers $p,q$ a bipartition $\la=(\la^L,\la^R)$ is called a
$(p,q)$-cross bipartition if 
there exist non-negative integers $p_1, p_2, q_1, q_2$ such that 
$\la^L$ is a $(p_1,q_1)$-hook partition, $\la^R$ is a $(p_2,q_2)$-hook partition and $p_1+p_2 \leq p$, $q_1+q_2 \leq q$.
\end{dfn}
Let $\Cr_{m,n}$ be the subset of all $(2,1)$-cross bipartitions in $\Lambda_{m,n}$. Assuming the Conjecture~1 (\ref{conj_1})  and applying the statements from \cite{Stroppel12}, \cite{CW11} 
for $M=2$, $N=1$, we have
\begin{prop}\label{crossbip}
If $\la\in\Cr_{m,n}$ then $\mathrm{ker}\, \Psi^{2,1}_{m,n}$ acts as zero on $D(\la)$. The modules $D(\la), \, \la\in\Cr_{m,n}$ give a complete set of simple $\Alg_{m,n}$-modules.
\end{prop}
\begin{prop}
Each $\Alg_{m,n}$-simple module $D(\la), \, \la\in\Cr_{m,n}$ occurs as a subquotient in the bimodule decomposition of $\Chain_{m,n}$.
\end{prop}

In the following we use notation $a=|m-n|$. For bipartitions from $\Cr_{m,n}$ we introduce the notation
\begin{description}\addtolength{\itemsep}{6pt}
  \item[for $m \geq n$] 
    \begin{align*}
  \Da{a}{s} &= ((a,1^s),(s)),  \quad a > 0, \quad  0 \leq s \leq n,\\
  \Db{a}{s} &= ((a,s),(1^s) ), \quad a > 0, \quad  1 \leq s \leq \min(a,n), \\
  \Dc{a}{s} &= ((s+1,a+1),(1^{s+2}) ), \quad a \leq s \leq n-2, \quad a \geq 0,
    \end{align*}
  \item[for $m \leq n$]
    \begin{align*}
  \Daa{a}{s} &= ((s), (a,1^s) ),  \quad a > 0, \quad  0 \leq s \leq m,\\
  \Da{0}{0} &=  \Daa{0}{0} =(\emp, \emp), \\
  \Dbb{a}{s} &= ((1^s), (a,s) ), \quad a > 0, \quad  1 \leq s \leq \min(a,m), \\
  \Dcc{a}{s} &= ((1^{s+2}), (s+1,a+1) ), \quad a \leq s \leq m-2, \quad a \geq 0.
    \end{align*}
\end{description}
We note that $\Db{a}{1}=\Da{a}{1}$ and $\Dcc{0}{0}=\Dc{0}{0}$.

For given $m$, $n$ we define a subset $\at_{m,n}$ of bipartitions in $\Cr_{m,n}$ as
\begin{equation}
\at_{m,n} = \begin{cases}
\{\Da{a}{s} | 0 \leq s \leq n\} \,\bigcup\,
		\{\Db{a}{s} | 2 \leq s \leq \min(a,n) \} \,\bigcup\,
		\{\Dc{a}{s} | a \leq s \leq n-2 \}, \quad m>n, \\
\{ \Da{0}{0} \} \,\bigcup\,
		\{\Dc{0}{s} | 0 \leq s \leq n-2 \} \,\bigcup\,
		\{\Dcc{0}{s} | 1 \leq s \leq n-2 \}, \quad m=n, \\
\{\Daa{a}{s} | 0 \leq s \leq m\} \,\bigcup\,
		\{\Dbb{a}{s} | 2 \leq s \leq \min(a,m) \} \,\bigcup\,
		\{\Dcc{a}{s} | a \leq s \leq m-2 \}, \quad m<n. 
\end{cases}		
\end{equation}
We call these bipartitions atypical. If $\la \in \at_{m,n}$ we call corresponding modules $S(\la)$ and $D(\la)$ atypical also.

We define the operation $\operS$ from the set of $\qwb_{m,n}$-modules to the
set of $\qwb_{n,m}$-modules. The operation $\operS$ acts on the simple $\qwb_{m,n}$-module by the formula
\begin{equation}\label{opS1}
\operS \Bigl( D\big[\la^L,\la^R\big] \Bigr) = D\big[\la^R,\la^L\big], 
\end{equation}
i.e. it changes left and right partitions in a bipartition.
We note that $\operS \Da{a}{s} = \Daa{a}{s}$, and similarly for $\Db{a}{s}$, $\Dc{a}{s}$.
When applied to projective modules, the operation $\operS$ acts on each
simple subquotient by the formula (\ref{opS1}) and does not change the
structure of the Loewy graph. 
It is obvious that
\begin{equation}\label{S_onK}
K\big[\la^R,\la^L\big]  = \operS \Bigl( K\big[\la^L,\la^R\big] \Bigr).
\end{equation}

The action of the algebra $\Alg_{m,n}$ on an arbitrary $\qwb_{m,n}$-module is not defined in general. 
In particular, it is not defined on some $\qwb_{m,n}$-Specht modules, that contain
$D(\la'), \,\, {\la' \notin \Cr_{m,n}}$ as a subquotient. 
For $\la\in\Cr_{m,n}$ we define a Specht module over $\Alg_{m,n}$ 
(abusing notation we use the same symbol $S(\la)$ for it)
as a factor of corresponding $\qwb_{m,n}$-Specht module $S(\la)$ over all suquotients $D(\la')$ with
$\la' \notin \Cr_{m,n}$.  

Similarly we let $K(\la)$ denote the projective cover for $\Alg_{m,n}$-module $S(\la)$. This projective cover is a subquotient of 
$\qwb_{m,n}$ projective module $K(\la)$.

Assuming the Conjecture~2 (\bref{conj_2}), we have the equality of multiplicities $\tilde{d}_{\la,\mu}=d_{\la,\mu}$ for $\Alg_{m,n}$ in analogy with \eqref{B-H_rec}. Using \cite{Cox2010} and Proposition \bref{crossbip}, we have the following Theorem.
We write down the structure of the Loewy graphs for $\Alg_{m,n}$-projective modules 
(analogously to the formulas \ref{proj1}--\ref{proj3} for $\qSL{2|1}$-projective modules). They are oriented graphs where arrows mean the action of the algebra $\Alg_{m,n}$. States from the subquotient at the beginning of an arrow are mapped to the states in the subquotient at the end of an arrow and (possibly) in the subquotients further the arrows. Investigation of $\bExti$ spaces for the algebra $\Alg_{m,n}$ and the detailed action of all $\Alg_{m,n}$-generators on projective modules are beyond the scope of this paper.  

\begin{thm}\label{K_structure}
For $\la \in \Cr_{m,n}$, $\la \notin \at_{m,n}$, the projective module over 
$\Alg_{m,n}$ coincides with the simple module: 
$K(\la)=D(\la)$. For $\la \in \at_{m,n}$, we have the following structure of projective modules over $\Alg_{m,n}$
\begin{description}\addtolength{\itemsep}{6pt}
  \item[for $m > n$]
\begin{align*}  
&\xymatrix@=20pt{
    && {D(\Da{a}{s})} \ar[dl]\ar[dr]
    &&\\
    K(\Da{a}{s}) =  \hspace{-0.7cm}
    &  D(\Da{a}{s-1}) \ar[dr]
    && D(\Da{a}{s+1}) \ar[dl] \text{\, ,}
    & \qquad 2 \leq s \leq n-1, \quad a \geq 1,
    \\
    && D(\Da{a}{s})
    &&
  } 
\\
&\xymatrix@=25pt{
    && D(\Da{a}{1})\ar[dl]\ar[d]_{}\ar[dr]
    &
    \\
	K(\Da{a}{1}) = \hspace{-0.7cm}
    & D(\Da{a}{2}) \ar[dr]
    & D(\Da{a}{0}) \ar[d]_{}
    & D(\Db{a}{2}) \ar[dl]  \text{\, ,}
    & \qquad  a \geq 2, \quad n \geq 2, 
    \\
    && D(\Da{a}{1})
    &&
  }
\\  
&\xymatrix@=25pt{
    && D(\Da{1}{1})\ar[dl]\ar[d]_{}\ar[dr]
    &
    \\
	K(\Da{1}{1}) = \hspace{-0.7cm}
    & D(\Da{1}{2}) \ar[dr]
    & D(\Da{1}{0}) \ar[d]_{}
    & D(\Dc{1}{1}) \ar[dl] \text{\, ,}
    & \qquad  n \geq 3,
    \\
    && D(\Da{a}{1})
    &&
  }
\\ 
&\xymatrix@=25pt{
    & D(\Da{a}{n})\ar[d]_{} &
    \\
    K(\Da{a}{n}) = \hspace{-0.7cm}
    &  D(\Da{a}{n-1}) \,\, ,
    \ar[d]_{}
    &  \qquad  a \geq 1, \quad n \geq 1,
    \\
    & D(\Da{a}{n})
    &
  }
\\
&\xymatrix@=25pt{
    K(\Da{a}{0}) = \hspace{-0.7cm}
    &  D(\Da{a}{0}) \,\, ,
    \ar[d]_{}
    & \qquad  a \geq 1, \quad n \geq 1,
    \\
    & D(\Da{a}{1})
    &
  }
\\
&\xymatrix@=20pt{
    && {D(\Db{a}{s})} \ar[dl]\ar[dr]
    &&\\
    K(\Db{a}{s}) =  \hspace{-0.7cm}
    &  D(\Db{a}{s-1}) \ar[dr]
    && D(\Db{a}{s+1}) \,\, , \ar[dl]
    &   \qquad 2 \leq s \leq \min(a,n)-1 \,\, ,
    \\
    && D(\Db{a}{s})
    &&
  }
\\
&\xymatrix@=20pt{
    && {D(\Db{a}{a})} \ar[dl]\ar[dr]
    &&\\
    K(\Db{a}{a}) =  \hspace{-0.7cm}
    &  D(\Db{a}{a-1}) \ar[dr]
    && D(\Dc{a}{a}) \,\, , \ar[dl]
    &   \qquad 2 \leq a \leq n-2 \,\, ,
    \\
    && D(\Db{a}{a})
    &&
  }
\\
&\xymatrix@=25pt{
    & D(\Db{a}{a})\ar[d]_{}
    &
    \\
    K(\Db{a}{a}) = \hspace{-0.7cm}
    &  D(\Db{a}{a-1}) \,\, , \ar[d]_{}
	&  \qquad a=n-1, \quad n \geq 3,   
    \\
    & D(\Db{a}{a})
    &
  }
\\
&\xymatrix@=25pt{
    & D(\Db{a}{n})\ar[d]_{}
    &
    \\
    K(\Db{a}{n}) = \hspace{-0.7cm}
    &  D(\Db{a}{n-1}) \,\, , \ar[d]_{}
	&  \qquad a \geq n, \quad n \geq 2,
    \\
    & D(\Db{a}{n})
    &
  }
\\
&\xymatrix@=20pt{
    && {D(\Dc{a}{s})} \ar[dl]\ar[dr]
    &&\\
    K(\Dc{a}{s}) =  \hspace{-0.7cm}
    &  D(\Dc{a}{s-1}) \ar[dr]
    && D(\Dc{a}{s+1}) \,\, , \ar[dl]
    &   \qquad a+1 \leq s \leq n-3, \quad a \geq 1,
    \\
    && D(\Dc{a}{s})
    &&
  }
\\
&\xymatrix@=20pt{
    && {D(\Dc{a}{a})} \ar[dl]\ar[dr]
    &&\\
    K(\Dc{a}{a}) =  \hspace{-0.7cm}
    &  D(\Db{a}{a}) \ar[dr]
    && D(\Dc{a}{a+1}) \,\, , \ar[dl]
    &  \qquad 2 \leq a \leq n-3,
    \\
    && D(\Dc{a}{a})
    &&
  }
\\
&\xymatrix@=20pt{
    && {D(\Dc{1}{1})} \ar[dl]\ar[dr]
    &&\\
    K(\Dc{1}{1}) =  \hspace{-0.7cm}
    &  D(\Da{1}{1}) \ar[dr]
    && D(\Dc{1}{2}) \,\, , \ar[dl]
    &  \qquad n \geq 4,
    \\
    && D(\Dc{1}{1})
    &&
  }
\\
&\xymatrix@=25pt{
    & D(\Dc{a}{n-2})\ar[d]_{}
    &
    \\
    K(\Dc{a}{n-2}) = \hspace{-0.7cm}
    &  D(\Dc{a}{n-3}) \,\, , \ar[d]_{}
    & \qquad 1 \leq a \leq n-3, 
    \\
    & D(\Dc{a}{n-2})
    &
  }
\\
&\xymatrix@=25pt{
    & D(\Dc{n-2}{n-2})\ar[d]_{}
    &
    \\
    K(\Dc{n-2}{n-2}) = \hspace{-0.7cm}
    &  D(\Db{n-2}{n-2}) \,\, , \ar[d]_{}
    & \qquad n \geq 3. 
    \\
    & D(\Dc{n-2}{n-2})
    &
  }
\end{align*} 
  \item[for $m = n$]
\begin{align*}
&\xymatrix@=25pt{
    K(\Da{0}{0}) = \hspace{-0.7cm}
    &  D(\Da{0}{0}) \,\, ,
    \ar[d]_{}
    & \qquad n\geq 2,    
    \\
    & D(\Dc{0}{0})
    &
  }
\\
&\xymatrix@=25pt{
    && D(\Dc{0}{0})\ar[dl]\ar[d]_{}\ar[dr]
    &&\\
	K(\Dc{0}{0}) = \hspace{-0.7cm}
    & D(\Dcc{0}{1}) \ar[dr]
    & D(\Da{0}{0}) \ar[d]_{}
    & D(\Dc{0}{1}) \,\, , \ar[dl]
    &  \qquad n \geq 3,
    \\
    && D(\Dc{0}{0})
    &&
  }
\\ 
&\xymatrix@=20pt{
    && {D(\Dc{0}{s})} \ar[dl]\ar[dr]
    &&\\
    K(\Dc{0}{s}) =  \hspace{-0.7cm}
    &  D(\Dc{0}{s-1}) \ar[dr]
    && D(\Dc{0}{s+1}) \,\, , \ar[dl]
    &   \qquad 1 \leq s \leq n-3, 
    \\
    && D(\Dc{0}{s})
    &&
  }
\\
&\xymatrix@=25pt{
    & D(\Dc{0}{n-2})\ar[d]_{}
    &
    \\
    K(\Dc{0}{n-2}) = \hspace{-0.7cm}
    &  D(\Dc{0}{n-3}) \,\, , \ar[d]_{}
    & \qquad n \geq 3 , 
    \\
    & D(\Dc{a}{n-2})
    &
  }
\\
\end{align*}
\end{description}
Structure of projective modules $K(\Dcc{0}{s}), 0 \leq s \leq n-2$ for $m=n$ and all projective modules for $m<n$
can be obtained from this using the formula (\ref{S_onK}).
\end{thm}

\subsection{The restriction functors}
\subsubsection{}\label{Res_definition}
There are two natural embeddings between quantum walled Brauer
algebras (see \cite{CDM2007})
\begin{equation}
 \qwb_{m-1,n} \to \qwb_{m,n}, \qquad
 \qwb_{m,n-1} \to \qwb_{m,n}.
\end{equation}
The first embedding acts by identification of the
corresponding generators 
$\EE$, $g_1, g_2, \dots g_{m-2}$, $h_1, h_2, \dots h_{n-1}$. 
The second embedding acts by identification of the generators 
$\EE$, $g_1, g_2, \dots g_{m-1}$, $h_1, h_2, \dots h_{n-2}$. 
These two maps induce two
restriction functors 
$\res^{m,n}_{m-1,n}$ and $\res^{m,n}_{m,n-1}$
from the category of $\qwb_{m,n}$-modules to the
categories of $\qwb_{m-1,n}$ and $\qwb_{m,n-1}$-modules respectively. 

Let $\add(\mu)$ be the set of boxes for a partition $\mu$, which can be added singly to $\mu$ such that 
the result $\mu+\Box$ is a partition.
Let $\remov (\mu)$ be a set of boxes which can be 
removed from $\mu$ such that $\mu/\Box$ is a partition.

In what follows the sign $\biguplus$ denotes the non-direct sum of modules.
Following \cite{CDM2007}, where the classical case $\q=1$ is considered, we have for modules over $\qwb_{m,n}$ 
\begin{prop}\label{restr_S_qwb}
For $\la \in \Lambda_{m,n}(f)$ with $n\geq 1$ we have
\begin{align*}
\resN\, S(\la) =& 
 \biguplus_{\Box \in \remov (\la^R)}  S( \la^L, \la^R-\Box ), \quad \text{for } f=0, \\
\resN\, S(\la) =& 
 \biguplus_{\Box \,\in \add (\la^L)}  S( \la^L +\Box, \la^R) \biguplus
 \biguplus_{\Box \,\in \remov (\la^R)}  S( \la^L, \la^R-\Box ),
 \quad \text{for } f>0.  
\end{align*}
\end{prop}
This statement is valid for the algebra $\qwb_{m,n}$ with either generic or special parameters. 
For $\qwb_{m,n}$ with generic parameters all $\biguplus$ become direct sums.

As a consequence of the previous statement and Proposition \bref{crossbip} 
we have for modules over $\Alg_{m,n}$ 
\begin{prop}\label{restr_S_A}
For $\la \in \Lambda_{m,n}(f) \bigcap \Cr_{m,n}$ with $n\geq 1$ we have
\begin{align*}
  \resN\, S(\la) =& 
 \sum_{\Box \in \remov (\la^R)}  S( \la^L, \la^R-\Box ), \quad \text{for } f=0,  \\
    \resN\, S(\la) =& 
 \biguplus_{\Box \,\in \add (\la^L),  (\la^L +\Box, \la^R) \in \Cr_{m,n-1} }  
 S( \la^L +\Box, \la^R) \biguplus
 \biguplus_{\Box \,\in \remov (\la^R)}  S( \la^L, \la^R-\Box ), \quad \text{for } f>0. 
\end{align*}
\end{prop}

\begin{conj}
Restriction for projective module $K(\la)$ over algebra $\Alg_{m,n}$ is a sum of projective modules. 
\end{conj}
\begin{thm}\label{restr_K_A}
Consider $n\geq 1$. For $\la \in \at_{m,n}$ the restrictions for projective modules $K(\la)$ over  the algebra $\Alg_{m,n}$ are
\begin{align*}
&\resN\, K(\Da{a}{s}) = K(\Da{a+1}{s}) \oplus D[(a,1^{s+1}),(s)] \oplus 2D[(a,1^s),(s-1)] \oplus D[(a,1^{s-1}),(s-2)], \\ 
&\quad 2 \leq s \leq n-1, \quad a \geq 1,
\\
&\resN\, K(\Da{a}{1}) = K(\Da{a+1}{1}) \oplus D[(a,1^{2}),(1)] \oplus +2D[(a,1),\emp] \oplus D[(a,2),(1)],\quad a\geq 2, n\geq 2,
\\
&\resN\, K(\Da{1}{1}) = K(\Da{2}{1}) \oplus D[(1^{3}),(1)] \oplus 2D[(1^{2}),\emp], \quad n \geq 2,
\\
&\resN\, K(\Da{a}{n}) = D(\Da{a+1}{n-1}) \oplus 2D[(a,1^n),(n-1)] \oplus D[(a,1^{n-1}),(n-2)],  \quad n \geq 2,
\\
&\resN\, K(\Da{a}{0}) = K(\Da{a+1}{0}) \oplus D[(a,1),\emp], \quad a \geq 1, \quad n \geq 1,  
\\
&\resN\, K(\Da{0}{0}) = K(\Da{1}{0}), \quad n \geq 1,
\\
\\
&\resN\, K(\Db{a}{s}) = K(\Db{a+1}{s}) \oplus D[(a,s+1),(1^s)] \oplus 2D[(a,s),(1^{s-1})] \oplus D[(a,s-1),(1^{s-2})], \\
&\quad 2 \leq s \leq \min(a,n)-1,
\\
&\resN\, K(\Db{a}{a}) = K(\Db{a+1}{a}) \oplus 2D[(a,a),(1^{a-1})] \oplus D[(a,a-1),(1^{a-2})], \quad 2 \leq a \leq n-1,
\\
&\resN\, K(\Db{a}{n}) = D(\Db{a+1}{n-1}) \oplus 2D[(a,n),(1^{n-1})] \oplus D[(a,n-1),(1^{n-2})], \, 2 \leq n \leq a,
\\
\\
&\resN\, K(\Dc{a}{s}) = K(\Dc{a+1}{s})  \oplus D[(s+2,a+1),(1^{s+2})] \oplus 2D[(s+1,a+1),(1^{s+1})] \oplus \\ 
&\quad \oplus D[(s,a+1),(1^{s})], \qquad 
\quad a+2 \leq s \leq n-3,
\\
&\resN\, K(\Dc{a}{a}) = K(\Db{a+1}{a+1}) \oplus D[(a+2,a+1),(1^{a+2})] \oplus D[(a,a),(1^{a-1})], \quad 2 \leq a \leq n-3,
\\
&\resN\, K(\Dc{a}{a+1}) = K(\Dc{a+1}{a+1}) \oplus D[(a+3,a+1),(1^{a+3})] \oplus 2D[(a+2,a+1),(1^{a+2})], \quad a \leq n-4,
\\
&\resN\, K(\Dc{a}{n-2}) = D(\Dc{a+1}{n-3}) \oplus 2 D[(n-1,a+1),(1^{n-1})] \oplus D[(n-2,a+1),(1^{n-2})], \\
 &\quad a \leq n-4,
\\
&\resN\, K(\Dc{n-3}{n-2}) = D(\Db{n-2}{n-2}) \oplus 2 D[(n-1,n-2),(1^{n-1})], \quad n \geq 3,
\\
&\resN\, K(\Dc{n-2}{n-2}) = K(\Db{n-1}{n-1}) \oplus D[(n-2,n-2),(1^{n-3})], \quad n \geq 3,
\\
&\resN\, K(\Dc{1}{1}) = K(\Db{2}{2}) \oplus D[(3,2),(1^{3})] \oplus D[(1^2),\emp],
\quad n \geq 4, 
\\
&\resN\, K(\Dc{0}{0}) = K(\Da{1}{1}) \oplus D[(2,1),(1^{2})] \oplus D[(1^3),(1^{2})],
\quad n \geq 3,
\\
\\
&\resN\, K(\Daa{a}{s}) = K(\Daa{a-1}{s}) \oplus D[(s+1),(a,1^{s})] \oplus 2D[(s),(a,1^{s-1})] \oplus D[(s-1),(a,1^{s-2})], \\
&\quad 2 \leq s \leq m-1, \quad a \geq 2,
\\
&\resN\, K(\Daa{a}{1}) = K(\Daa{a-1}{1}) \oplus D[(2),(a,1)] \oplus 2D[(1),(a)] \oplus D[(1^2),(a,1)],
\quad a \geq 2, \quad m\geq 2,
\\
&\resN\, K(\Daa{a}{m}) = K(\Daa{a-1}{m}) \oplus 2D[(m),(a,1^{m-1})] \oplus D[(m-1),(a,1^{m-2})], 
\quad m \geq 2, \quad a \geq 2,
\\
&\resN\, K(\Daa{1}{s}) = K(\Dc{0}{s-1}) \oplus D[(s+1),(1^{s+1})] \oplus 2D[(s),(1^{s})] \oplus D[(s-1),(1^{s-1})], \quad 2 \leq s < m,
\\
&\resN\, K(\Daa{1}{m}) = D(\Dc{0}{m-2}) \oplus 2D[(m),(1^{m})] \oplus D[(m-1),(1^{m-1})], 
\quad m \geq 2,
\\
&\resN\, K(\Daa{1}{1}) = K(\Dc{0}{0}) \oplus D[(2),(1^2)] \oplus 2D[(1),(1)], \quad m \geq 2,
\\
&\resN\, K(\Daa{a}{0}) = K(\Daa{a-1}{0}) \oplus D[(1),(a)], 
\quad a \geq 1, \quad m \geq 1,
\\
\\
&\resN\, K(\Dbb{a}{s}) = K(\Dbb{a-1}{s}) \oplus D[(1^{s+1}),(a,s)] \oplus 2D[(1^{s}),(a,s-1)] \oplus D[(1^{s-1}),(a,s-2)],\\
&\quad 2 \leq s \leq \min(a,m)-1,
\\
&\resN\, K(\Dbb{a}{a}) = K(\Dcc{a-1}{a-1}) \oplus 2D[(1^{a}),(a,a-1)] \oplus D[(1^{a-1}),(a,a-2)],
\quad 2 \leq a \leq m-1,
\\
&\resN\, K(\Dbb{a}{m}) = K(\Dbb{a-1}{m}) \oplus 2D[(1^{m}),(a,m-1)] \oplus D[(1^{m-1}),(a,m-2)],   \quad 2 \leq m < a,
\\
&\resN\, K(\Dbb{m}{m}) = D(\Dbb{m-1}{m-1}) \oplus 2D[(1^{m}),(m,m-1)] \oplus D[(1^{m-1}),(m,m-2)],  \quad m \geq 2,
\\
\\
&\resN\, K(\Dcc{a}{s}) = K(\Dcc{a-1}{s}) \oplus D[(1^{s+3}),(s+1,a+1)] \oplus 2D[(1^{s+2}),(s,a+1)] \\ &\qquad \qquad \oplus D[(1^{s+1}),(s-1,a+1)], \quad a+2 \leq s \leq m-3, \quad a \geq 1,
\\
&\resN\, K(\Dcc{0}{s}) = K(\Da{1}{s+1}) \oplus D[(1^{s+3}),(s+1,1)] \oplus 2D[(1^{s+2}),(s,1)] \oplus D[(1^{s+1}),(s-1,1)], \\
&\quad 2 \leq s \leq m-3, 
\\
&\resN\, K(\Dcc{a}{a}) = K(\Dcc{a-1}{a}) \oplus D[(1^{a+3}),(a+1,a+1)] \oplus D[(1^{a}),(a,a-1)], \quad 1 \leq a \leq m-3,
\\
&\resN\, K(\Dcc{a}{a+1}) = K(\Dcc{a-1}{a+1}) \oplus D[(1^{a+4}),(a+2,a+1)] \oplus 2D[(1^{a+3}),(a+1,a+1)], \quad 1 \leq a \leq m-4,
\\
&\resN\, K(\Dcc{a}{m-2}) = K(\Dcc{a-1}{m-2}) \oplus 2D[(1^{m}),(m-2,a+1)] \oplus D[(1^{m-1}),(m-3,a+1)], \quad 1 \leq a \leq m-4,
\\
&\resN\, K(\Dcc{0}{1}) = K(\Da{1}{2}) \oplus D[(1^{4}),(2,1)] \oplus 2D[(1^{3}),(1,1)], \quad m \geq 4, 
\\
&\resN\, K(\Dcc{0}{m-2}) = K(\Da{1}{m-1}) \oplus 2D[(1^{m}),(m-2,1)] \oplus D[(1^{m-1}),(m-3,1)], \quad m \geq 4, 
\\
&\resN\, K(\Dcc{m-3}{m-2}) = K(\Dcc{m-4}{m-2}) \oplus 2D[(1^{m}),(m-2,m-2)] , \quad m \geq 4,
\\
&\resN\, K(\Dcc{m-2}{m-2}) = K(\Dcc{m-3}{m-2}) \oplus D[(1^{m-2}),(m-2,m-3)] , \quad m \geq 3.
\end{align*}
where we imply $(0)=(1^0)=\emp$ and $(s,0)=(s)$.
\end {thm}
\begin{proof}
We discuss the case $K(\Da{a}{s})$ for $2 \leq s \leq n-1$, $a \geq 1$. Other cases are similar. The projective module $K(\Da{a}{s})$ has a filtration by two atypical Specht modules, so one can write it as a non direct sum 
\begin{equation*}
 K(\Da{a}{s}) = S(\Da{a}{s}) \bup S(\Da{a}{s-1}).
\end{equation*}
Applying the Proposition \bref{restr_S_A} one obtains the sum of simple and atypical Specht modules:
\begin{multline*}
\resN\, K(\Da{a}{s}) = \resN\, \big(S(\Da{a}{s}) \bup S(\Da{a}{s-1})\big) = \\ 
=S(\Da{a+1}{s}) \bup S[(a,1^{s+1}),(s)] \bup S[(a,1^{s}),(s-1)] \\ 
\bup S(\Da{a+1}{s-1}) \bup S[(a,1^{s}),(s-1)] \bup S[(a,1^{s-1}),(s-2)].
\end{multline*}
In this sum only two modules are atypical, other modules are simple 
\begin{multline*}
\resN\, K(\Da{a}{s}) = S(\Da{a+1}{s}) \bup S(\Da{a+1}{s-1}) \bup \\
D[(a,1^{s+1}),(s)] \bup 2D[(a,1^{s}),(s-1)] \bup D[(a,1^{s-1}),(s-2)].
\end{multline*}
These two atypical Specht modules are glued uniquely into a projective module, thus
\begin{equation*}
 \resN\, K(\Da{a}{s})=K(\Da{a+1}{s}) \,\oplus\, D[(a,1^{s+1}),(s)] \,\oplus\, 2D[(a,1^{s}),(s-1)] \oplus 
 D[(a,1^{s-1}),(s-2)].
\end{equation*}
\end{proof}

To formulate the next theorem we introduce notation $a' = |m-n+1|$. 
\begin{thm}\label{restr_D_A}
Consider $\la \in \Lambda_{m,n}(f) \bigcap \Cr_{m,n}$ for $n\geq 1$. 
The restrictions for simple modules $D(\la)$ over the algebra $\Alg_{m,n}$ are
\\
for $\la \in \at_{m,n}$:
\begin{align*}
&\resN D(\Da{a}{s})=
D(\Da{a+1}{s}) \oplus D[ (a, 1^{s}), (s-1)], 
\quad a \geq 1, \quad 1\leq s \leq n-1, 
\\
&\resN D(\Da{a}{n})= D[ (a, 1^{n}), (n-1)], 
\quad a \geq 1, \quad n \geq 1,
\\
&\resN D(\Da{a}{0})=
D(\Da{a+1}{0}), \quad a \geq 0, \quad n \geq 0, 
\\
&\resN D(\Db{a}{s})=
D(\Db{a+1}{s})\oplus D[ (a,s), (1^{s-1}) ],
\quad 1\leq s \leq n-1, \quad s \leq a, 
\\
&\resN D(\Db{a}{n})= D[ (a,n), (1^{n-1}) ],
\quad 1\leq n \leq a, 
\\
&\resN D(\Dc{a}{s})=
D(\Dc{a+1}{s})\oplus D[ (s+1,a+1), (1^{s+1}) ], 
\quad a+1 \leq s \leq n-3, \quad a \geq 0,
\\
&\resN D(\Dc{a}{n-2})= D[ (n-1,a+1), (1^{n-1}) ], 
\quad 0 \leq a \leq n-3,
\\
&\resN D(\Dc{a}{a})= D( \Db{a+1}{a+1}), 
\quad 0 \leq a \leq n-2,
\\
\\
&\resN D(\Daa{a}{s})=
D(\Daa{a-1}{s})\oplus D[ (s), (a, 1^{s-1}) ],
\quad 1 \leq s \leq m, \quad a \geq 2, 
\\
&\resN D(\Daa{a}{0})= D(\Daa{a-1}{0}),
\quad a \geq 1, \quad m \geq 0, 
\\
&\resN D(\Daa{1}{s})= D(\Dc{0}{s-1}) \oplus D[ (s), (1^{s}) ],
\quad 1 \leq s \leq m-1, 
\\
&\resN D(\Daa{1}{m})= D[ (m), (1^{m}) ], \quad m \geq 1, 
\\
&\resN D(\Dbb{a}{s})=
D(\Dbb{a-1}{s})\oplus D\big[ (1^s), (a,s-1)  \big],
\quad 2 \leq s \leq m, \quad s<a,
\\
&\resN D(\Dbb{a}{a})=
D(\Dcc{a-1}{a-1})\oplus D\big[ (1^a), (a,a-1)  \big] , \quad 2 \leq a \leq m-1,
\\
&\resN D(\Dbb{m}{m})=
D\big[ (1^m), (m,m-1)  \big] , \quad 1 \leq m,
\\
&\resN D(\Dcc{a}{s})=
D(\Dcc{a-1}{s}) \oplus D\big[ (1^{s+2}), (s,a+1) \big],
\quad a+1 \leq s \leq m-2, \quad a \geq 1,
\\
&\resN D(\Dcc{a}{a})= D(\Dcc{a-1}{a}),
\quad 1 \leq a \leq m-2,
\\
&\resN D(\Dcc{0}{s})=
D(\Da{1}{s+1}) \oplus D\big[ (1^{s+2}), (s,1) \big],
\quad 1 \leq s \leq m-2.
\end{align*}
For $\la \notin \at_{m,n}$ first we list all exceptional cases (the generic rule will be given below):
\begin{align*}
&\resN D\big[ (a', 1^{s-1}), (s)\big]=
K(\Da{a'}{s})\oplus D\big[ (a'+1, 1^{s-1}), (s)\big] ,
\quad 1 \leq s \leq n-1, \quad a' \geq 1,
\\
&\resN D\big[ (a',s), (1^{s+1}) \big]=
K(\Db{a'}{s+1})\oplus D\big[ (a'+1,s), (1^{s+1}) \big],   
\quad 1 \leq s \leq a'-1, \quad s \leq n-2,
\\
&\resN D\big[ (s,a'+1), (1^{s+2}) \big]=
 K(\Dc{a'}{s})\oplus D\big[ (s,a'+2), (1^{s+2}) \big] , \quad a'+2 \leq s \leq n-3,
\\
&\resN D\big[ (a'+1,a'+1), (1^{a'+3}) \big]=
 K(\Dc{a'}{a'+1}), \quad a' \leq n-4,
\\
&\resN D\big[ (s), (a', 1^{s+1})\big]=
K(\Daa{a'}{s+1})\oplus D\big[ (s), (a'-1, 1^{s+1}) \big], 
\quad 0 \leq s \leq m-1, \quad a' \geq 2,
\\
&\resN D\big[ (s), (1^{s+2})\big]=
K(\Daa{1}{s+1})\oplus D\big[ (s,1), (1^{s+2}) \big], 
\quad 1 \leq s \leq m-1,
\\
&\resN D\big[ (1^{s-1}), (a',s) \big]=
 K(\Dbb{a'}{s})\oplus D\big[ (1^{s-1}), (a'-1,s)  \big] , 
\quad 2 \leq s \leq a'-1, \quad s \leq m,
\\
&\resN D\big[ (1^{a'-1}), (a',a') \big]=
 K(\Dbb{a'}{a'}) , 
\quad 1 \leq a' \leq m,
\\
&\resN D\big[ (1^s), (s,a'+1) \big]=
K(\Dcc{a'}{s-1})\oplus D\big[ (1^{s}), (s,a')  \big],
\quad a'+1 \leq s \leq m-1,
\end{align*}
where we imply $(0)=(1^0)=\emp$ and $(s,0)=(s)$.\\
For $\la \notin \at_{m,n}$ the generic rule is: \\
for $f=0$
\begin{equation*}
  \resN\, D(\la) = 
 \sum_{\Box \in \remov (\la^R)}  D( \la^L, \la^R-\Box ),
\end{equation*}
for $f>0$ 
\begin{equation*}
  \resN\, D(\la) = 
 \bigoplus_{\Box \,\in \add (\la^L),  (\la^L +\Box, \la^R) \in \Cr_{m,n-1} }  
 D( \la^L +\Box, \la^R) \oplus
 \bigoplus_{\Box \,\in \remov (\la^R)}  D( \la^L, \la^R-\Box ). 
\end{equation*}
\end{thm}
\begin{proof}
If $\la \notin \at_{m,n}$ then $D(\la) = S(\la)$, and the proof follows from \bref{restr_S_A} similarly to the proof of Theorem \bref{restr_K_A}. 

Now we consider $\la \in \at_{m,n}$. 
We discuss only $D(\Da{a}{s})$ for $a \geq 1$, $1\leq s \leq n-1$, other cases are similar. 
We prove that
\begin{equation*}
\resN D(\Da{a}{s})=
D(\Da{a+1}{s}) \oplus D\big[ (a, 1^{s}), (s-1)\big], \quad a \geq 1, s \leq n-1, \\
\end{equation*}
by induction on $s$.
First, we prove the induction base for $s=n-1$, then we check the induction step from $s$ to $s-1$.
The $\Alg_{m,n}$-module $S(\Da{a}{n})$ is simple: $S(\Da{a}{n})=D(\Da{a}{n})$, so we have from \bref{restr_S_A}
\begin{equation}\label{resProf1}
  \resN\, D(\Da{a}{n}) = \resN\, S(\Da{a}{n}) =  S\big[ (a,1^{n}), (n-1) \big] = 
  D\big[ (a,1^{n}), (n-1) \big].
\end{equation}
According to \bref{restr_S_A} we have for $s<n$
\begin{equation}\label{resProf2}
  \resN\, S(\Da{a}{s}) = S(\Da{a+1}{s}) \oplus D\big[ (a,1^{s+1}), (s) \big] \oplus D\big[ (a,1^s), (s-1) \big].
\end{equation}
We write $\Alg_{m,n}$-Specht modules as a non-direct sum $S(\Da{a}{s}) =D(\Da{a}{s})\bup D(\Da{a}{s+1})$ for $s<n$.
The $\Alg_{m,n-1}$-module  $S(\Da{a+1}{n-1}) = D(\Da{a+1}{n-1})$, so from (\ref{resProf2}) for $s=n-1$ we get
\begin{equation*}
  \resN\, \Big(D(\Da{a}{n-1})\bup D(\Da{a}{n})\Big) =D(\Da{a+1}{n-1}) \bup
  D\big[ (a,1^{n}), (n-1) \big] \bup D\big[ (a,1^{n-1}), (n-2) \big].
\end{equation*}
Now having in mind (\ref{resProf1}) we get the induction base
\begin{equation*}
  \resN\, D(\Da{a}{n-1}) =D(\Da{a+1}{n-1}) \bup D\big[ (a,1^{n-1}), (n-2) \big].
\end{equation*}
We also note that $\Alg_{m,n-1}$ module 
$S(\Da{a+1}{s}) = D(\Da{a+1}{s}) \bup D(\Da{a+1}{s+1})$ for $s<n-1$, so from (\ref{resProf2}) we get
\begin{equation*}
  \resN\, \Big(D(\Da{a}{s})\bup D(\Da{a}{s+1})\Big) =D(\Da{a+1}{s})\bup D(\Da{a+1}{s+1}) \bup D\big[ (a,1^s), (s-1) \big] 
  \bup D\big[ (a,1^{s+1}), (s) \big],
\end{equation*}
and now the induction step is straightforward.
\end{proof}

\begin{rem}
The second restriction functor $\res^{m,n}_{m-1,n}$ can be calculated from the first one. Actually
\begin{align}
&\res^{m,n}_{m-1,n} K(\la) = \operS \,\res^{n,m}_{n,m-1} \operS K(\la), \\
&\res^{m,n}_{m-1,n} D(\la) = \operS \,\res^{n,m}_{n,m-1} \operS D(\la).
\end{align}
\end{rem}

We can also make generalization to the $\qwb_{m,n}$ modules.
\begin{conj}
Consider the algebra $\qwb_{m,n}$ with special parameter
 $\theta = -(-\frac{\delta}{\gamma})^{M-N}$. Let $\la \in \Lambda_{m,n}$ be an $(M,N)$-cross bipartition, 
then $\resN\, D(\la)$ contains only subquotients $D(\la')$ 
for which $\la' \in \Lambda_{m,n-1}$ is an $(M,N)$-cross bipartition.  
\end{conj}
In other words, the restriction functor for $\qwb_{m,n}$ with special parameters preserves the class of all $(M,N)$-cross bipartitions.
In particular we have the next important consequence for $M=2, \, N=1$.
\begin{conj}
For $\la\in\Cr_{m,n}$ the restrictions $\res^{m,n}_{m,n-1} D(\la)$ for simple modules 
over $\qwb_{m,n}$ with $\theta = \frac{\delta}{\gamma}$ are explicitly given by the formulas from 
theorem~\bref{restr_D_A} without any changes.
\end{conj}
This conjecture was directly checked for all $\qwb_{m,n}$-modules whenever $m+n \leq 8$.

\section{The mixed tensor product as a bimodule}\label{sec:bimodule}
We introduce new notation in order to simplify the formula for the bimodule decomposition.

\subsection{Notation}
We introduce the notation $\repZZ^{p}_{t,r}$ for simple $\qSL{2|1}$ modules:
\begin{align*}
  \repZZ^{p}_{t,r} &= \repZ^{1,(-1)^p}_{t+r,r}, \quad r \ne 0,\\
  \repZZ^{p}_{t,0} &= \repZ^{1,(-1)^{p+1}}_{t+1,0}, \quad t\geq 0. 
\end{align*}
We also introduce the notation $\repRR^{p}_{t,r}$ for projective covers of atypical modules $\repZZ^{p}_{t,r}$. Namely, 
\begin{align*}
  \repRR^{p}_{0,r} &= \repR^{1,(-1)^p}_{r,r}, \quad r \geq 1, \\
  \repRR^{p}_{t,0} &= \repR^{1,(-1)^{p+1}}_{t+1,0},  \quad t\geq 0.
\end{align*}
Typical modules $\repZZ^{p}_{t,r}$ coincide with their projective covers, so we do not introduce any new notation for them.
We rewrite the formulas \ref{proj1}--\ref{proj3} in the new notation:
\begin{equation}  
\xymatrix@=20pt{
    && {\repZZ^{p}_{t,0}} \ar[dl]\ar[dr]  
    &\\
	\repRR^{p}_{t,0}=  \hspace{-0.7cm}  
    &  {\repZZ^{p+1}_{t+1,0}}\ar[dr]
    && {\repZZ^{p-1}_{t-1,0}}, \ar[dl]
    \\
    && {\repZZ^{p}_{t,0}}
    &
  }\hspace{1cm}
\xymatrix@=20pt{
    && {\repZZ^{p}_{0,t}} \ar[dl]\ar[dr]  
    &&\\
	\repRR^{p}_{0,t}=  \hspace{-0.7cm}   
    &  {\repZZ^{p+1}_{0,t+1}}\ar[dr]
    && {\repZZ^{p-1}_{0,t-1}}, \ar[dl]
    &\quad t \geq 1,
    \\
    && {\repZZ^{p}_{0,t}}
    &&
  }
\end{equation}
and the exeptional case is 
\begin{equation}  
\xymatrix@=20pt{
    && {\repZZ^{p}_{0,0}} \ar[dl]\ar[dr]  
    &\\
	\repRR^{p}_{0,0}=  \hspace{-0.7cm}  
    &  {\repZZ^{p+1}_{1,0}}\ar[dr]
    && {\repZZ^{p-1}_{0,1}}.\ar[dl]
    \\
    && {\repZZ^{p}_{0,0}}
    &
  }
\end{equation}
Then the dimensions are:
\begin{align*}
\dim\repRR^{p}_{r,0} =& \dim\repRR^{p}_{0,r} = 8r+4, \quad r>0, \\
&\dim\repRR^{p}_{0,0} = 8.
\end{align*}

\subsection{}
The bimodule is a direct sum of subbimodules
\begin{equation}\label{AP-decomp}
  \Chain_{m,n} = T^s_{m,n} \oplus \Chat_{m,n} ,
\end{equation}
where the $T^s_{m,n}$ part is the direct sum of simple $\Alg_{m,n} \boxtimes\qSL{2|1}$-bimodules, and
$\Chat_{m,n}$ is an indecomposable $\Alg_{m,n} \boxtimes\qSL{2|1}$-bimodule. 
Each subquotient in $T^s_{m,n}$ contains a typical $\qSL{2|1}$-module and a typical $\Alg_{m,n}$-module, 
and each subquotient in $\Chat_{m,n}$ contains an atypical $\qSL{2|1}$-module and an atypical $\Alg_{m,n}$-module.
We call $T^s_{m,n}$ the semisimple part and $\Chat_{m,n}$
the atypical part. 

\subsubsection{Examples}
Before giving a general formula for the decomposition of $\Chain_{m,n}$ in~\bref{bimodulethm}, we illustrate the structure of the semisimple part $T^s_{m,n}$ with two examples. 
$T^s_{m,n}$ has the structure
\begin{align}\label{eq:s-part}
  T^s_{m,n}=\bigoplus_{t,r} D(\la_{m,n}(t,r)) \boxtimes 
  \repZZ^{p(t,r)}_{t,r}
\end{align}
For given $m$, $n$, we represent the sum in~\eqref{eq:s-part} as a
table of bipartitions $\la_{m,n}(t,r)$ in coordinates $(t,r)$.
All parts of the sum outside the table vanish, and $0$ in the table means that the corresponding submodule in~\eqref{eq:s-part} vanishes.

For $m=5$ and $n=3$, the table of bipartitions $\la_{5,3}(t,r)$ reads
\begin{align*}
     &  \hspace{0cm} \tiny 
       \begin{array}{|c|c|c|c|c|c|c|} \hline
         r=5 & 0 & (\yng(1,1,1,1,1),\yng(3))\vspace{3pt} &  0 & (\yng(3,1,1),\yng(3)) & (\yng(4,1),\yng(3)) & (\yng(5),\yng(3)) \\ \hline
         r=4 & (\yng(1,1,1,1,1),\yng(2,1))\vspace{3pt} & (\yng(1,1,1,1),\yng(2)) & 
                      0 & (\yng(3,1),\yng(2)) & (\yng(4),\yng(2)) & (\yng(5),\yng(2,1)) \\ \hline
         r=3 & (\yng(1,1,1,1),\yng(1,1))\vspace{3pt} & (\yng(1,1,1),\yng(1)) & 0 & (\yng(3),\yng(1)) & (\yng(4),\yng(1,1)) & (\yng(5),\yng(1,1,1))  \\ \hline
         r=2 & 0& (\yng(1,1),\emp) & 0 & (\yng(3,1),\yng(1,1)) & (\yng(4,1),\yng(1,1,1))\vspace{3pt} & 0  \\ \hline
         r=1 & 0 & 0 & 0 & (\yng(3,2),\yng(1,1,1))\vspace{3pt} & 0 & 0 \\ \hline
             & t=-2 & t=-1 & t=0 & t=1 & t=2 & t=3  \\ \hline
       \end{array}
\end{align*}\\[0cm]
For $m=4$ and $n=4$, the table of bipartitions $\la_{4,4}(t,r)$ reads  
\begin{align*}
& \hspace{0cm} \tiny 
\begin{array}{|c|c|c|c|c|c|c|} \hline
 \text{r=4} & 0 & 0 & (\yng(1,1,1,1),\yng(4))\vspace{3pt} & (\yng(2,1,1),\yng(4)) & (\yng(3,1),\yng(4)) & (\yng(4),\yng(4)) \\ \hline
 \text{r=3} & 0 & 0 & (\yng(1,1,1),\yng(3))\vspace{3pt} & (\yng(2,1),\yng(3)) & (\yng(3),\yng(3)) & (\yng(4),\yng(3,1)) \\ \hline
 r=2 & (\yng(1,1,1,1),\yng(2,2))\vspace{3pt} & 0 & (\yng(1,1),\yng(2)) & (\yng(2),\yng(2)) & (\yng(3),\yng(2,1)) & (\yng(4),\yng(2,1,1))  \\ \hline
 \text{r=1} & 0 & 0 & (\yng(1),\yng(1)) & (\yng(2),\yng(1,1)) & (\yng(3),\yng(1,1,1)) & (\yng(4),\yng(1,1,1,1))\vspace{3pt}  \\ \hline
 \text{r=0} & 0 & 0 & 0 & 0 & 0 & 0 \\ \hline
 \text{r=-1} & 0 & 0 & 0 & (\yng(2,2),\yng(1,1,1,1))\vspace{3pt} & 0 & 0  \\ \hline
  & t=-1 & t=0 & t=1 & t=2 & t=3 & t=4  \\ \hline
\end{array}
\end{align*}

\subsubsection{}
In the next Theorem, we give explicit formulas for the decomposition
of $\Chain_{m,n}$ for $m\geq n$; the case $m<n$ can be easily
recovered from $m>n$ using operation $\operS$ interchanging $m$ with $n$
\begin{equation*}
  \Chain_{n,m}=\operS \Chain_{m,n}.
\end{equation*}
The operation is involutive, $\operS^2=1$, and additive,
$\operS(X\oplus Y)=\operS(X)\oplus \operS(Y)$. 
It acts on the indecomposable summands in the semisimple part $T^s_{m,n}$ by the formula
\begin{equation}\label{opS2}
  \operS \Bigl( D\big[\la^L,\la^R \big] \boxtimes \repZZ_{t,r}^{p} \Bigr)=
  \operS \Bigl( D\big[\la^L,\la^R\big] \Bigr) \boxtimes \operS \Bigl(  \repZZ_{t,r}^{p} \Bigr),
\end{equation}
where the action $\operS \Bigl( D\big[\la^L,\la^R\big] \Bigr)$ is defined in (\ref{opS1}) and 
\begin{equation}
  \operS \repZZ_{t,r}^{p} = \repZZ_{r,t}^{p}.
\end{equation}
When applied to the atypical part $\Chat_{m,n}$, the operation $\operS$ acts on each simple subquotient by the formula (\ref{opS2}) and does not change the structure of the Loewy graph.

\begin{Thm}\label{bimodulethm}
  The $\Alg_{m,n} \boxtimes\qSL{2|1}$-bimodule decomposition of 
  $\Chain_{m,n}$, $m\geq n$, has the form
  $\Chain_{m,n} = T^s_{m,n} \oplus \Chat_{m,n}$
  with the semisimple part 
  \begin{description}\addtolength{\itemsep}{6pt}
  \item[$m> n$] 
    \begin{align*}
      T^s_{m,n}=
      &\bigoplus_{s=1}^{n}\bigoplus_{\substack{k=1,\\k\neq a}}^{a+s} D\big[(k,1^{s-k+a}),\, (s) \big]\boxtimes \repZZ^{s+k+a+1}_{k-a,s+a}\oplus      \\
      &\bigoplus_{s=a+2}^{m}\bigoplus_{k=1}^{s-a-1} D\big[(s),\,(k,1^{s-k-a})\big]\boxtimes
        \repZZ^{s+k+a+1}_{s-a,k+a}\oplus	 \\
      &\bigoplus_{s=1}^{n-1}\bigoplus_{k=1}^{\min(s,\,n-s)} D\big[(1^{s+k+a}),\,(s,k)\big]\boxtimes 	\repZZ^{s+k+a}_{1-k-a,s+a}\oplus      \\
      &\bigoplus_{s=a+1}^{m-1}\bigoplus_{\substack{k=1,\\k\neq a+1}}^{\min(s,\,m-s)} D\big[(s,k),\,(1^{s+k-a})\big]\boxtimes \repZZ^{s+k+a}_{s-a,1-k+a}\oplus     \\
      &\bigoplus_{k=1}^{\lfloor\frac{a}{2}\rfloor}\bigoplus_{s=k}^{a-k} D\big[(s,k, 1^{a-s-k}),\,\emp \big]\boxtimes
        \repZZ^{s+k+a}_{s-a,1-k+a}\oplus    \\
      &\bigoplus_{s=\lfloor\frac{a}{2}\rfloor + 1}^{a-1}\bigoplus_{k=1-s+a}^{\min(s,\,m-s)} D\big[(s,k),
        \,  (1^{s+k-a})\big]\boxtimes
        \repZZ^{s+k+a}_{s-a,1-k+a},
    \end{align*}
    \item[$m=n$]
    \begin{align*}
      T^s_{m,m}=
      &\bigoplus_{s=1}^{m}\bigoplus_{k=1}^{s} D\big[(k,1^{s-k}),\,(s)\big]\boxtimes
        \repZZ^{s+k+1}_{k,s}\oplus
      \\
      &\bigoplus_{s=2}^{m}\bigoplus_{k=1}^{s-1} D\big[(s),\,(k,1^{s-k})\big]\boxtimes
        \repZZ^{s+k+1}_{s,k}\oplus
      \\[2mm]
      &\bigoplus_{s=2}^{m-1}\bigoplus_{k=2}^{\min(s,\,m-s)} 
      D\big[(1^{s+k}),\,(s,k)\big]\boxtimes
        \repZZ^{s+k}_{1-k,s}\oplus\\
      &\bigoplus_{s=1}^{m-1}\bigoplus_{k=2}^{\min(s,\,m-s)} D\big[(s,k),\,(1^{s+k})\big]\boxtimes
        \repZZ^{s+k}_{s,1-k},
    \end{align*}
  \end{description}
and the atypical part $\Chat_{m,n}$ is given by figures \ref{R1}--\ref{R5} in Appendix \bref{Bimodule}.\\
\end{Thm}

\subsection{Verification}
To check the decomposition formula for the bimodule
 we make two powerful verifications using formulas for
tensor product decompositions for $\qSL{2|1}$ modules and restrictions for
$\Alg_{m,n}$ modules. We check that $\Chain_{m,n}\tensor\bthree$ coincides
with $\res^{m,n+1}_{m,n}\, \Chain_{m,n+1}$ as $\qSL{2|1}$-module in the first verification and as $\Alg_{m,n}$-module in the second one. 
In order to do this we introduce two Grothendieck
(forgetful) functors $\ppp$ and $\qqq$.

We define the Grothendieck  functor $\ppp$ on the category of $\qSL{2|1}$-modules which maps an indecomposable module into a direct sum of its simple 
subquotients. 
The functor $\ppp$ on any $\qSL{2|1}$-module is known from
\bref{projective}. For example
\begin{align*}
\ppp \repRR^{p}_{t,0} &= 2\repZZ^{p}_{t,0} \oplus 
    \repZZ^{p+1}_{t+1,0}\oplus \repZZ^{p-1}_{t-1,0}, 
\quad t \geq 1, \\
\ppp \repZZ^{p}_{t,r} &= \repZZ^{p}_{t,r}, \quad \forall p,t,r.
\end{align*}
We define the other Grothendieck  functor $\qqq$ on the category of 
$\Alg_{m,n}$ modules which maps an indecomposable module into a direct sum of its
simple subquotients. The functor $\qqq$ on any $\Alg_{m,n}$-module is known
from \bref{K_structure}. For example 
\begin{align*}
&\qqq K(\Da{1}{1}) = 2D(\Da{1}{1}) 
\oplus D(\Da{1}{2}) \oplus D(\Da{1}{0}) \oplus D(\Dc{1}{1}), \quad n\geq 3,\\
&\qqq D(\la) = D(\la), \quad \forall m,n,\la.
\end{align*}
The functors $\ppp$ and~$\qqq$ do not change semisimple part of the bimodule:
\begin{equation*}
\ppp T^s_{m,n} = \qqq T^s_{m,n} = T^s_{m,n},
\end{equation*}
because semisimple part is a direct sum of simple bimodules.

\subsubsection{As $\qSL{2|1}$ module}
The action of $\qqq$ on the atypical part $\Chat_{m,n}$ has the form 
\begin{description}\addtolength{\itemsep}{6pt}
  \item[$m> n$] 
    \begin{align*}
 \qqq \Chat_{m,n} =
  \bigoplus_{s=1}^{n} & D(\Da{a}{s}) \boxtimes \repRR^{s}_{0,a+s-1} \oplus \\
  \bigoplus_{s=2}^{\min(a,\,n)} & D(\Db{a}{s}) \boxtimes \repRR^{s}_{0,a-s+1} \oplus \\
  \bigoplus_{s=a}^{n-2} & D(\Dc{a}{s}) \boxtimes \repRR^{s+1}_{s-a,0} \oplus \\
  & D(\Da{a}{0}) \boxtimes \repZZ^{1}_{0,a},
    \end{align*}
  \item[$m=n$]
    \begin{align*}
  \qqq \Chat_{m,m} =
  &\bigoplus_{s=1}^{m-2} D(\Dcc{0}{s}) \boxtimes \repRR^{s-1}_{0,s} \oplus \\
  &\bigoplus_{s=0}^{m-2} D(\Dc{0}{s}) \boxtimes \repRR^{s-1}_{s,0} \oplus \\
  & D(\Da{0}{0}) \boxtimes \repZZ^{1}_{0,0}.
    \end{align*}
\end{description}
We introduce the notation
$\overline{\Chain_{m,n}} = \qqq \Chain_{m,n}$. 
The following relation must hold:
\begin{equation}\label{Ch-induction1}
\overline{\Chain_{m,n}} \tensor \bthree = \qqq \, \res^{m,n+1}_{m,n}\, \overline{\Chain_{m,n+1}}.  
\end{equation}
Because $\overline{\Chain_{m,n}}$ has the form 
$\overline{\Chain_{m,n}} = \bigoplus D \boxtimes \repR \,\, \bigoplus\,\, D \boxtimes \repZ$, we can calculate 
$\overline{\Chain_{m,n}} \tensor \bthree$ 
using formulas from~\bref{tp-decomp}.
Because $\overline{\Chain_{m,n+1}}$ contains as subquotients only modules $D(\la)$ for $\la \in \Cr_{m,n+1}$, we can calculate $\res^{m,n+1}_{m,n}\, \overline{\Chain_{m,n+1}}$
using formulas from~\bref{restr_D_A}, and then apply the functor $\qqq$. We have
checked the validity of relation \ref{Ch-induction1} for all $m,n$ whenever
$m+n \leq 25$.

\subsubsection{As $\Alg_{m,n}$ module}
The action of $\ppp$ on the atypical part $\Chat_{m,n}$ has the form 
\begin{description}\addtolength{\itemsep}{6pt}
  \item[$m> n$] 
    \begin{align*}
 \ppp \Chat_{m,n} =
  \bigoplus_{s=1}^{n} & K(\Da{a}{s}) \boxtimes \repZZ^{s}_{0,a+s-1} \oplus \\
  \bigoplus_{s=2}^{\min(a,\,n)} & K(\Db{a}{s}) \boxtimes \repZZ^{s}_{0,a-s+1} \oplus \\
  \bigoplus_{s=a}^{n-2} & K(\Dc{a}{s}) \boxtimes \repZZ^{s+1}_{s-a,0} \\
  \oplus & D(\Da{a}{n}) \boxtimes \repZZ^{n+1}_{0,m} \\
  \oplus & PT^{\text{right}}_{m,n},
    \end{align*}
   where
\begin{equation*}
PT^{\text{right}}_{m,n} = 
\begin{cases}
	0, \qquad n=0, \\
	D(\Db{a}{n}) \boxtimes \repZZ^{n+1}_{0,m-2n}, \qquad 1 \leq n \leq \frac{m}{2}, \\
	D(\Db{a}{a}) \boxtimes \repZZ^{a+1}_{0,0}, \qquad n= \frac{m+1}{2}, 
	\quad n \geq 2, \\
	D(\Dc{a}{n-2}) \boxtimes \repZZ^{n}_{2n-m-1,0}, \qquad 
		\frac{m}{2}+1 \leq n \leq m-1, 
\end{cases}
\end{equation*}    
  \item[$m=n \geq 2$]
    \begin{align*}
  \ppp \Chat_{m,m} =
  &\bigoplus_{s=1}^{m-2} K(\Dcc{0}{s}) \boxtimes \repZZ^{s-1}_{0,s} \oplus \\
  &\bigoplus_{s=0}^{m-2} K(\Dc{0}{s}) \boxtimes \repZZ^{s-1}_{s,0} \\
  &\oplus  D(\Dcc{0}{m-2}) \boxtimes \repZZ^{m-2}_{0,m-1}\\ 
  &\oplus D(\Dc{0}{m-2}) \boxtimes \repZZ^{m-2}_{m-1,0},
    \end{align*}
   \item[$m=n=1$]
\begin{equation*}
	\ppp \Chat_{1,1} = D(\Da{0}{0}) \boxtimes \repZZ^{1}_{0,0}.
\end{equation*}
\end{description}
We introduce the notation
$\underline{\Chain_{m,n}} = \ppp \Chain_{m,n}$. 
The following relation must hold:
\begin{equation}\label{Ch-induction2}
\ppp \, (\underline{\Chain_{m,n}} \tensor \bthree) = \res^{m,n+1}_{m,n}\, \underline{\Chain_{m,n+1}}.  
\end{equation}
Because $\underline{\Chain_{m,n}}$ has the form 
$\underline{\Chain_{m,n}} = \bigoplus K \boxtimes \repZ \,\, \bigoplus\,\, D \boxtimes \repZ$, we can calculate 
$\underline{\Chain_{m,n}} \tensor \bthree$ 
using formulas from~\bref{tp-decomp}.
Because $\underline{\Chain_{m,n+1}}$ contains as subquotients only modules $K(\la)$ and $D(\la)$ for $\la \in \Cr_{m,n+1}$, we can calculate $\res^{m,n+1}_{m,n}\, \underline{\Chain_{m,n+1}}$
using formulas from~\bref{restr_D_A} and \bref{restr_K_A} and then apply the
functor $\ppp$. We have checked the validity of relation \ref{Ch-induction2}
for all $m,n$ whenever $m+n \leq 25$.

\section{Conclusion}
In the present work we have studied the $\qSL{2|1}$ mixed tensor product and found
its decomposition as a bimodule over $\Alg_{m,n}\boxtimes\qSL{2|1}$. 
These results are the basis for a further study of the $\qSL{2|1}$-spin-chain and
appropriate LCFT.

The next step is studying the mixed tensor product with parameter $\q$ at the root of
unity. We expect the appearance of the Lusztig limit of algebra $\qSL{2|1}$ in
that case. We anticipate that $\Alg_{m,n}$ will remain the centralizer of
$\mathscr{L}\qSL{2|1}$ on the mixed tensor product and some triplet extension of
$\Alg_{m,n}$ will be the centralizer of $\qSL{2|1}$.

Natural ways for further developments of the results presented in this paper:
\begin{enumerate}
\item
Describe explicitly the algebra $\Alg_{m,n}$ and identify it with some
quotient of $\qwb_{m,n}$. 
Similar problem is posed the algebras $\Alg^{M,N}_{m,n}$ of
$\qSL{M|N}$-endomorphisms. 
\item
Describe the structure of Specht and projective $\Alg^{M,N}_{m,n}$-modules
and perhaps $\qwb_{m,n}$-modules. This problem becomes significantly more
complicated when parameter $\q$ is a root of unity.
\item
Figure out the restriction functor on all simple and projective modules
of the algebra $\qwb_{m,n}$.
\item
Classify $\Exti$ spaces for modules over the algebra $\Alg_{m,n}$ and describe
explicitly the action of $\Alg_{m,n}$-generators on the basis of projective
modules $K(\la)$. The solution to this problem will allow one to describe explicitly the
$\Alg_{m,n}$-action in the bimodule $\Chain_{m,n}$. 
\end{enumerate}

\subsection*{Acknowledgments}
We thank A.~Davydov, B.~Feigin, S.~Lentner, and
H.~Saleur for useful discussions and suggestions. 
We thank A.~Semikhatov for considerable contribution to the work
on its early stage.   
We are grateful to
A.~Gainutdinov for careful reading the draft of the paper and useful
suggestions. 
DB offers special thanks to A.~Elishev for careful reading of the manuscript and helpful comments.
DB thanks IPhT Saclay for hospitality, and support from the ERC Advanced Grant NuQFT.
The work of AK was supported in part by Dynasty Foundation and the LPI Educational-Scientific Complex.

\newpage
\appendix
\section{Atypical part of the bimodule}\label{Bimodule}
In this section we represent the structure of Loewy graph for the indecomposable bimodule $\Chat_{m,n}$, see \bref{AP-decomp}.
Detailed investigation of $\Alg_{m,n}$ action on these bimodules are beyound the scope of this paper. See paper \cite{AzatVasseur2013}, where the spin chain over $\qSL{2}$ is investigated for comparison.

In each vertex of the graph there is some subquotient
$D(\la) \boxtimes \repZZ_{t,r}^{p}$. 
The meaning of the arrows is the same as in \bref{K_structure}. 
On the figures the action of algebra $\qSL{2|1}$ is denoted by solid lines, and the action of $\Alg_{m,n}$ is denoted by dash lines.

The subquotients connected by dash lines have the same $\qSL{2|1}$ module as a tensor multiplier.
The subquotients connected by solid lines have the same $\Alg_{m,n}$ module as a tensor multiplier.
To simplify the figures we omit $\Alg_{m,n}$ multiplier where it does not cause inconsistency. We also do not write symbol $D$ each time, and write only $\la$ for simple module $D(\la)$.

For example, the bimodule for $\Chat_{3,2}$ is
\begin{equation}
\xymatrix@R=18pt@C=1pt@W=0pt@M=0.5pt{
    & D(\Da{1}{2}) \boxtimes \repZZ^{2}_{0,2} \ar[dl]\ar[dr]\myar[drrr]  
    &&&& D(\Da{1}{1})\boxtimes \repZZ^{1}_{0,1} \ar[dl]\ar[dr]\myar[d]\myar[dlll]
    &
    \\    
    D(\Da{1}{2}) \boxtimes \repZZ^{3}_{0,3} \hspace{3pt} \ar[dr]
    && D(\Da{1}{2}) \boxtimes \repZZ^{1}_{0,1} \ar[dl]\myar[drrr]
    &\hspace{1cm}
    & D(\Da{1}{1})\boxtimes \repZZ^{2}_{0,2}  \ar[dr]\myar[dlll]
    & \red{D(\Da{1}{0})\boxtimes \repZZ^{1}_{0,1}} \myar[d]
    & D(\Da{1}{1})\boxtimes \repZZ^{2}_{0,0} \ar[dl]
    \\
	& D(\Da{1}{2}) \boxtimes \repZZ^{2}_{0,2}
    &&&& D(\Da{1}{1})\boxtimes \repZZ^{1}_{0,1}
    &
  }
\end{equation}
We use shorthand notation for $\Chat_{3,2}$:
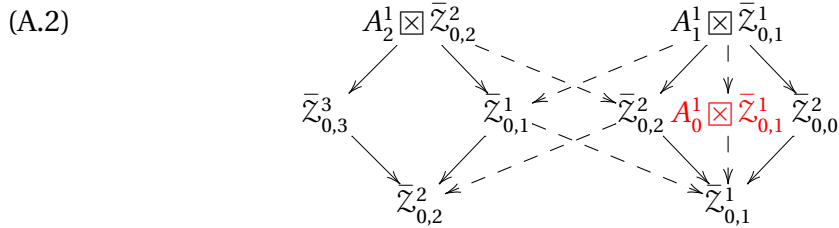
\begin{equation}
\xymatrix@R=18pt@C=1pt@W=0pt@M=0.5pt{
    & \Da{1}{2}\boxtimes \repZZ^{2}_{0,2} \ar[dl]\ar[dr]\myar[drrr]  
    &&&& \Da{1}{1}\boxtimes \repZZ^{1}_{0,1} \ar[dl]\ar[dr]\myar[d]\myar[dlll]
    &
    \\    
    \repZZ^{3}_{0,3} \hspace{3pt} \ar[dr]
    && \repZZ^{1}_{0,1} \ar[dl]\myar[drrr]
    &\hspace{1cm}
    & \repZZ^{2}_{0,2}  \ar[dr]\myar[dlll]
    & \red{\Da{1}{0}\boxtimes \repZZ^{1}_{0,1}} \myar[d]
    & \repZZ^{2}_{0,0} \ar[dl]
    \\
	& \repZZ^{2}_{0,2}
    &&&& \repZZ^{1}_{0,1}
    &
  }
\end{equation}
We mark in red the subquotient where the figure has irregular form.\\
The structure of $\Chat_{m,n}$ for the case $1 \leq n \leq \frac{m}{2}$ is shown in figure~\ref{R1}.\\
The case $\frac{m}{2}+1  \leq n \leq  m-2$ is shown in figure~\ref{R2}.\\
The case $n=\frac{m+1}{2}, n \geq 2$ is shown in figure~\ref{R3}.\\
The case $n=m-1, n \geq 1$ is shown in figure~\ref{R4}.\\
The case $n=m, n \geq 2$ is shown in figure~\ref{R5}.\\
Two exceptional cases are:
\begin{align*}
\Chat_{m,0} &= D(\Da{m}{0}) \boxtimes \repZZ^{1}_{0,m}, \\
\Chat_{1,1} &= D(\Da{0}{0}) \boxtimes \repZZ^{1}_{0,0}. 
\end{align*}


\newpage
\begin{landscape}
\thispagestyle{empty}
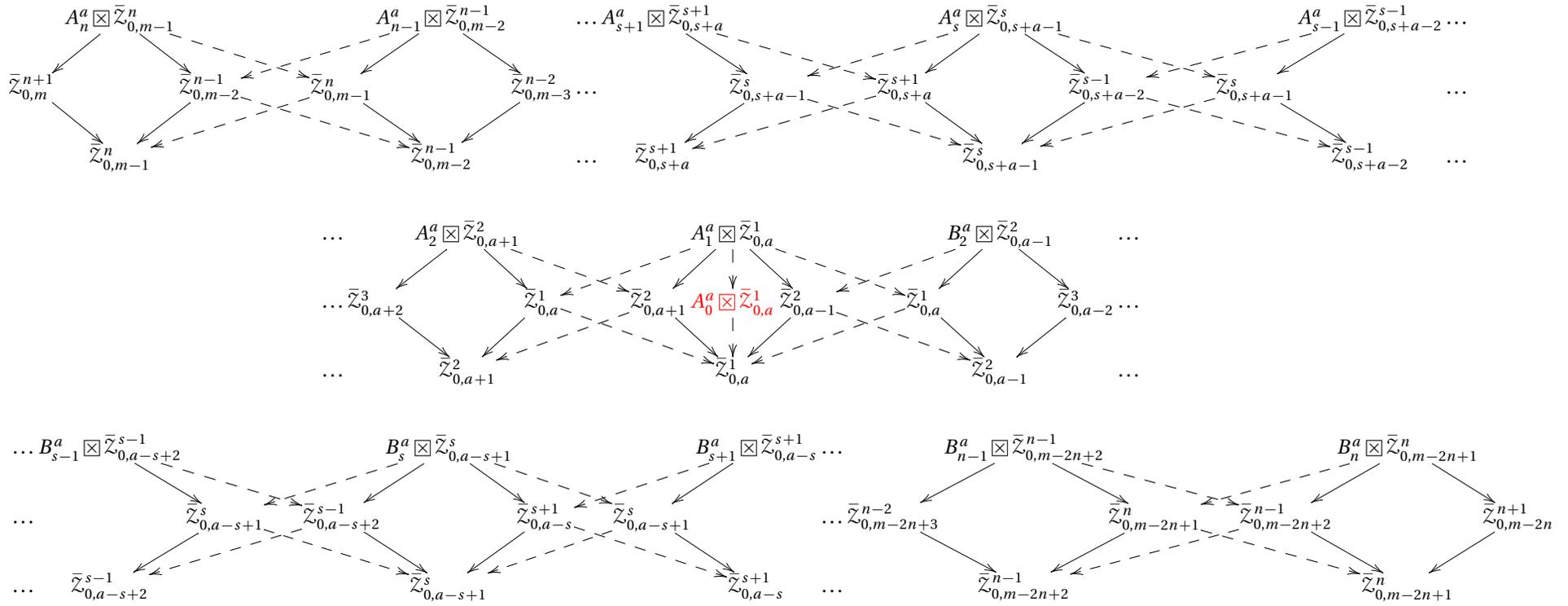
\begin{figure} 
{\footnotesize
\centering
\begin{multline*}
\hspace{-1cm}
 \xymatrix@R=18pt@C=1pt@W=0pt@M=0.5pt{
    &&\Da{a}{n} \boxtimes \repZZ^{n}_{0,m-1} \ar[dl]\ar[dr]\myar[drrr]  
    &&&& \Da{a}{n-1} \boxtimes \repZZ^{n-1}_{0,m-2} \ar[dl]\ar[dr]\myar[dlll]	
    &&
    \dots
    & \Da{a}{s+1}\boxtimes \repZZ^{s+1}_{0,s+a} \ar[dr]\myar[drrr]  
    &&&& \Da{a}{s} \boxtimes \repZZ^{s}_{0,s+a-1} \ar[dl]\ar[dr]\myar[dlll]\myar[drrr]	
    &&&& \Da{a}{s-1}\boxtimes \repZZ^{s-1}_{0,s+a-2} \ar[dl]\myar[dlll]
    & \dots
    \\    
    &\repZZ^{n+1}_{0,m} \hspace{3pt} \ar[dr]
    && \repZZ^{n-1}_{0,m-2} \ar[dl]\myar[drrr]
    &\hspace{1cm}
    & \repZZ^{n}_{0,m-1}  \ar[dr]\myar[dlll]
    && \repZZ^{n-2}_{0,m-3} \ar[dl]
    &
    \dots
    && \repZZ^{s}_{0,s+a-1} \ar[dl]\myar[drrr]
    &\hspace{1cm}
    & \repZZ^{s+1}_{0,s+a}  \ar[dr]\myar[dlll]
    & 
    & \repZZ^{s-1}_{0,s+a-2} \ar[dl]\myar[drrr]
    &\hspace{1cm}
    &\repZZ^{s}_{0,s+a-1} \ar[dr]\myar[dlll]
    &&\dots
    \\
    && \repZZ^{n}_{0,m-1}
    &&&& \repZZ^{n-1}_{0,m-2}
    &&
    \dots
    & \repZZ^{s+1}_{0,s+a}
    &&&& \repZZ^{s}_{0,s+a-1}
    &&&& \repZZ^{s-1}_{0,s+a-2}
    & \dots
  }\\\mbox{}\\
 \xymatrix@R=18pt@C=1pt@W=0pt@M=0.5pt{
    \dots&
    & \Da{a}{2}\boxtimes \repZZ^{2}_{0,a+1} \ar[dl]\ar[dr]\myar[drrr]  
    &&&& \Da{a}{1}\boxtimes \repZZ^{1}_{0,a} \ar[dl]\ar[dr]\myar[d]\myar[dlll]\myar[drrr]	
    &&&& \Db{a}{2}\boxtimes \repZZ^{2}_{0,a-1} \ar[dl]\ar[dr]\myar[dlll]
    && \dots
    \\    
    \dots
    &\repZZ^{3}_{0,a+2} \hspace{3pt} \ar[dr]
    && \repZZ^{1}_{0,a} \ar[dl]\myar[drrr]
    &\hspace{1cm}
    & \repZZ^{2}_{0,a+1}  \ar[dr]\myar[dlll]
    & \red{\Da{a}{0}\boxtimes \repZZ^{1}_{0,a}} \myar[d]
    & \repZZ^{2}_{0,a-1} \ar[dl]\myar[drrr]
    &\hspace{1cm}
    &\repZZ^{1}_{0,a} \ar[dr]\myar[dlll]
    && \repZZ^{3}_{0,a-2} \ar[dl]
    & \dots
    \\
    \dots
    && \repZZ^{2}_{0,a+1}
    &&&& \repZZ^{1}_{0,a}
    &&&& \repZZ^{2}_{0,a-1}
    && \dots
  }\\\mbox{}\\
\hspace{-0.5cm}
 \xymatrix@R=18pt@C=1pt@W=0pt@M=0.5pt{
    \dots
    & \Db{a}{s-1} \boxtimes \repZZ^{s-1}_{0,a-s+2} \ar[dr]\myar[drrr]  
    &&&& \Db{a}{s}\boxtimes \repZZ^{s}_{0,a-s+1} \ar[dl]\ar[dr]\myar[dlll]\myar[drrr]	
    &&&& \Db{a}{s+1}\boxtimes \repZZ^{s+1}_{0,a-s} \ar[dl]\myar[dlll]
    & \dots
	&& \Db{a}{n-1} \boxtimes \repZZ^{n-1}_{0,m-2n+2} \ar[dl]\ar[dr] \myar[drrr]
	&&&& \Db{a}{n}\boxtimes \repZZ^{n}_{0,m-2n+1} \ar[dl] \ar[dr]\myar[dlll] &
    \\    
    \dots
    && \repZZ^{s}_{0,a-s+1} \ar[dl]\myar[drrr] 
    &\hspace{0.5cm}
    & \repZZ^{s-1}_{0,a-s+2}  \ar[dr]\myar[dlll]
    & 
    & \repZZ^{s+1}_{0,a-s} \ar[dl]\myar[drrr]
    &\hspace{0.5cm}
    &\repZZ^{s}_{0,a-s+1} \ar[dr]\myar[dlll]
    && \dots
	& \repZZ^{n-2}_{0,m-2n+3} \ar[dr]
	&& \repZZ^{n}_{0,m-2n+1} \ar[dl]\myar[drrr]
	& \hspace{0.5cm}
	& \repZZ^{n-1}_{0,m-2n+2} \ar[dr]\myar[dlll]  
    && \repZZ^{n+1}_{0,m-2n} \ar[dl]  
    \\
    \dots
    & \repZZ^{s-1}_{0,a-s+2}
    &&&& \repZZ^{s}_{0,a-s+1}
    &&&& \repZZ^{s+1}_{0,a-s}
    & \dots
    && \repZZ^{n-1}_{0,m-2n+2}
    &&&& \repZZ^{n}_{0,m-2n+1}&
  }\\\mbox{}\\
\end{multline*}
}
\caption{$T^d_{m,n}, \quad 1 \leq n \leq \frac{m}{2}$ }\label{R1}
\end{figure}

\newpage
\thispagestyle{empty}
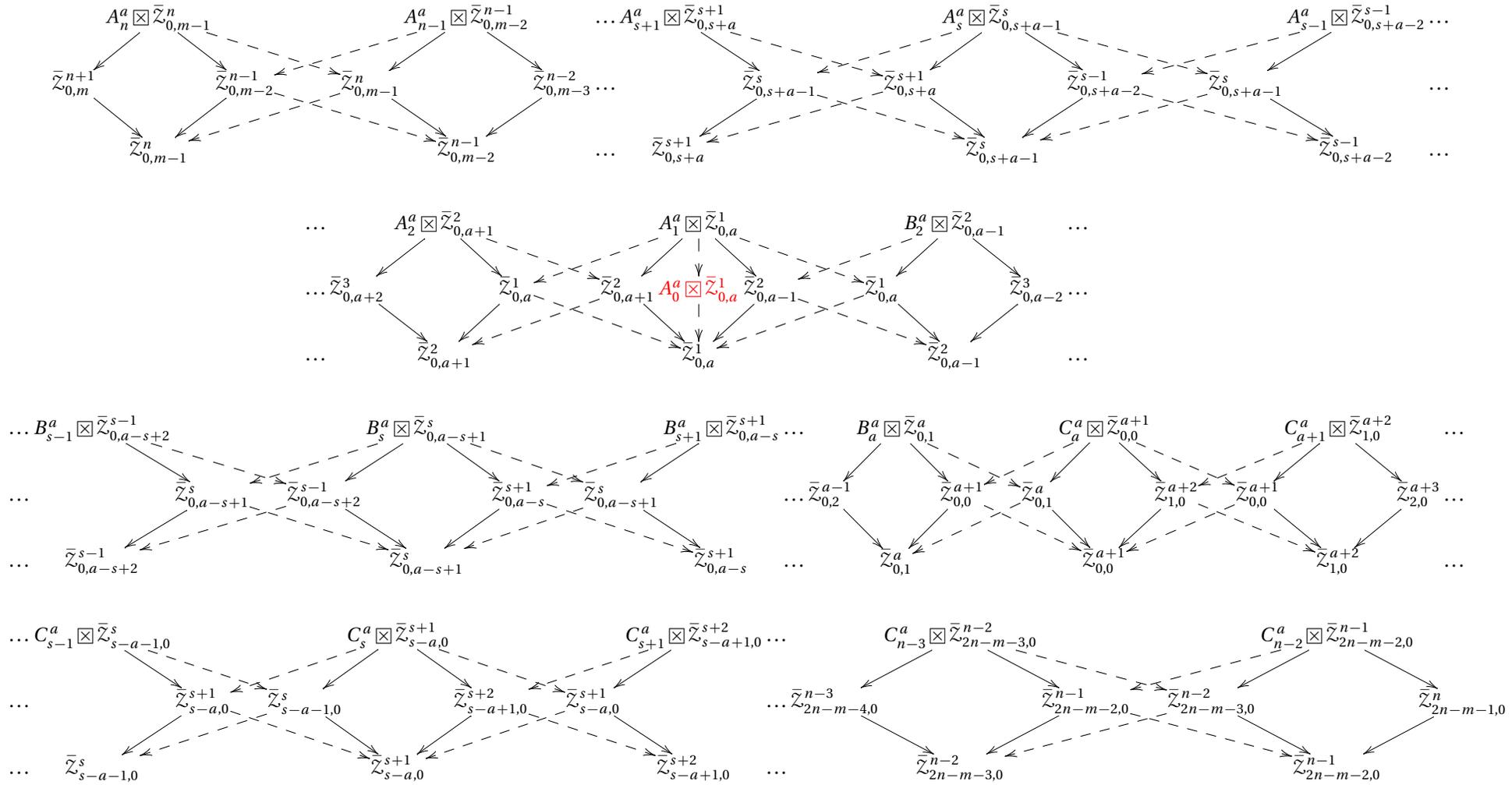
\begin{figure} 
{\footnotesize
\centering
\begin{multline*}
\hspace{-0.2cm}
 \xymatrix@R=18pt@C=1pt@W=0pt@M=0.5pt{
    &&\Da{a}{n} \boxtimes \repZZ^{n}_{0,m-1} \ar[dl]\ar[dr]\myar[drrr]  
    &&&& \Da{a}{n-1} \boxtimes \repZZ^{n-1}_{0,m-2} \ar[dl]\ar[dr]\myar[dlll]	
    &&
    \dots
    & \Da{a}{s+1}\boxtimes \repZZ^{s+1}_{0,s+a} \ar[dr]\myar[drrr]  
    &&&& \Da{a}{s} \boxtimes \repZZ^{s}_{0,s+a-1} \ar[dl]\ar[dr]\myar[dlll]\myar[drrr]	
    &&&& \Da{a}{s-1}\boxtimes \repZZ^{s-1}_{0,s+a-2} \ar[dl]\myar[dlll]
    & \dots
    \\    
    &\repZZ^{n+1}_{0,m} \hspace{3pt} \ar[dr]
    && \repZZ^{n-1}_{0,m-2} \ar[dl]\myar[drrr]
    &\hspace{1cm}
    & \repZZ^{n}_{0,m-1}  \ar[dr]\myar[dlll]
    && \repZZ^{n-2}_{0,m-3} \ar[dl]
    &
    \dots
    && \repZZ^{s}_{0,s+a-1} \ar[dl]\myar[drrr]
    &\hspace{1cm}
    & \repZZ^{s+1}_{0,s+a}  \ar[dr]\myar[dlll]
    & 
    & \repZZ^{s-1}_{0,s+a-2} \ar[dl]\myar[drrr]
    &\hspace{1cm}
    &\repZZ^{s}_{0,s+a-1} \ar[dr]\myar[dlll]
    &&\dots
    \\
    && \repZZ^{n}_{0,m-1}
    &&&& \repZZ^{n-1}_{0,m-2}
    &&
    \dots
    & \repZZ^{s+1}_{0,s+a}
    &&&& \repZZ^{s}_{0,s+a-1}
    &&&& \repZZ^{s-1}_{0,s+a-2}
    & \dots
  }\\\mbox{}\\
 \xymatrix@R=18pt@C=1pt@W=0pt@M=0.5pt{
    \dots&
    & \Da{a}{2}\boxtimes \repZZ^{2}_{0,a+1} \ar[dl]\ar[dr]\myar[drrr]  
    &&&& \Da{a}{1}\boxtimes \repZZ^{1}_{0,a} \ar[dl]\ar[dr]\myar[d]\myar[dlll]\myar[drrr]	
    &&&& \Db{a}{2}\boxtimes \repZZ^{2}_{0,a-1} \ar[dl]\ar[dr]\myar[dlll]
    && \dots
    \\    
    \dots
    &\repZZ^{3}_{0,a+2} \hspace{3pt} \ar[dr]
    && \repZZ^{1}_{0,a} \ar[dl]\myar[drrr]
    &\hspace{1cm}
    & \repZZ^{2}_{0,a+1}  \ar[dr]\myar[dlll]
    & \red{\Da{a}{0}\boxtimes \repZZ^{1}_{0,a}} \myar[d]
    & \repZZ^{2}_{0,a-1} \ar[dl]\myar[drrr]
    &\hspace{1cm}
    &\repZZ^{1}_{0,a} \ar[dr]\myar[dlll]
    && \repZZ^{3}_{0,a-2} \ar[dl]
    & \dots
    \\
    \dots
    && \repZZ^{2}_{0,a+1}
    &&&& \repZZ^{1}_{0,a}
    &&&& \repZZ^{2}_{0,a-1}
    && \dots
  }\\\mbox{}\\
\hspace{-0.5cm}
 \xymatrix@R=18pt@C=1pt@W=0pt@M=0.5pt{
    \dots
    & \Db{a}{s-1} \boxtimes \repZZ^{s-1}_{0,a-s+2} \ar[dr]\myar[drrr]  
    &&&& \Db{a}{s}\boxtimes \repZZ^{s}_{0,a-s+1} \ar[dl]\ar[dr]\myar[dlll]\myar[drrr]	
    &&&& \Db{a}{s+1}\boxtimes \repZZ^{s+1}_{0,a-s} \ar[dl]\myar[dlll]
    & \dots
	&& \Db{a}{a} \boxtimes \repZZ^{a}_{0,1} \ar[dl]\ar[dr] \myar[drrr]
	&&&& \Dc{a}{a}\boxtimes \repZZ^{a+1}_{0,0} \ar[dl] \ar[dr]\myar[dlll] \myar[drrr] 
	&&&& \Dc{a}{a+1}\boxtimes \repZZ^{a+2}_{1,0} \ar[dl] \ar[dr]\myar[dlll]  
	&& \dots
    \\    
    \dots
    && \repZZ^{s}_{0,a-s+1} \ar[dl]\myar[drrr] 
    &\hspace{0.5cm}
    & \repZZ^{s-1}_{0,a-s+2}  \ar[dr]\myar[dlll]
    & 
    & \repZZ^{s+1}_{0,a-s} \ar[dl]\myar[drrr]
    &\hspace{0.5cm}
    &\repZZ^{s}_{0,a-s+1} \ar[dr]\myar[dlll]
    && \dots
	& \repZZ^{a-1}_{0,2} \ar[dr]
	&& \repZZ^{a+1}_{0,0} \ar[dl]\myar[drrr]
	& \hspace{0.5cm}
	& \repZZ^{a}_{0,1} \ar[dr]\myar[dlll]  
    && \repZZ^{a+2}_{1,0} \ar[dl]\myar[drrr]    
    & \hspace{0.5cm}
    & \repZZ^{a+1}_{0,0} \ar[dr] \myar[dlll]
    && \repZZ^{a+3}_{2,0} \ar[dl]
    & \dots       
    \\
    \dots
    & \repZZ^{s-1}_{0,a-s+2}
    &&&& \repZZ^{s}_{0,a-s+1}
    &&&& \repZZ^{s+1}_{0,a-s}
    & \dots
    && \repZZ^{a}_{0,1}
    &&&& \repZZ^{a+1}_{0,0}
    &&&& \repZZ^{a+2}_{1,0} 
    && \dots
  }\\\mbox{}\\
\hspace{-0.5cm}
 \xymatrix@R=18pt@C=1pt@W=0pt@M=0.5pt{
    \dots
    & \Dc{a}{s-1} \boxtimes \repZZ^{s}_{s-a-1,0} \ar[dr]\myar[drrr]  
    &&&& \Dc{a}{s}\boxtimes \repZZ^{s+1}_{s-a,0} \ar[dl]\ar[dr]\myar[dlll]\myar[drrr]	
    &&&& \Dc{a}{s+1}\boxtimes \repZZ^{s+2}_{s-a+1,0} \ar[dl]\myar[dlll]
    & \dots
	&& \Dc{a}{n-3} \boxtimes \repZZ^{n-2}_{2n-m-3,0} \ar[dl]\ar[dr] \myar[drrr]
	&&&& \Dc{a}{n-2}\boxtimes \repZZ^{n-1}_{2n-m-2,0} \ar[dl] \ar[dr]\myar[dlll] &
    \\    
    \dots
    && \repZZ^{s+1}_{s-a,0} \ar[dl]\myar[drrr] 
    &\hspace{0.5cm}
    & \repZZ^{s}_{s-a-1,0}  \ar[dr]\myar[dlll]
    & 
    & \repZZ^{s+2}_{s-a+1,0} \ar[dl]\myar[drrr]
    &\hspace{0.5cm}
    &\repZZ^{s+1}_{s-a,0} \ar[dr]\myar[dlll]
    && \dots
	& \repZZ^{n-3}_{2n-m-4,0} \ar[dr]
	&& \repZZ^{n-1}_{2n-m-2,0} \ar[dl]\myar[drrr]
	& \hspace{0.5cm}
	& \repZZ^{n-2}_{2n-m-3,0} \ar[dr]\myar[dlll]  
    && \repZZ^{n}_{2n-m-1,0} \ar[dl]  
    \\
    \dots
    & \repZZ^{s}_{s-a-1,0}
    &&&& \repZZ^{s+1}_{s-a,0}
    &&&& \repZZ^{s+2}_{s-a+1,0}
    & \dots
    && \repZZ^{n-2}_{2n-m-3,0}
    &&&& \repZZ^{n-1}_{2n-m-2,0}&
  }\\\mbox{}\\
\end{multline*}
}
\caption{$T^d_{m,n}, \frac{m}{2}+1  \leq n \leq  m-2 $.}
\label{R2}
\end{figure}

\newpage
\thispagestyle{empty}
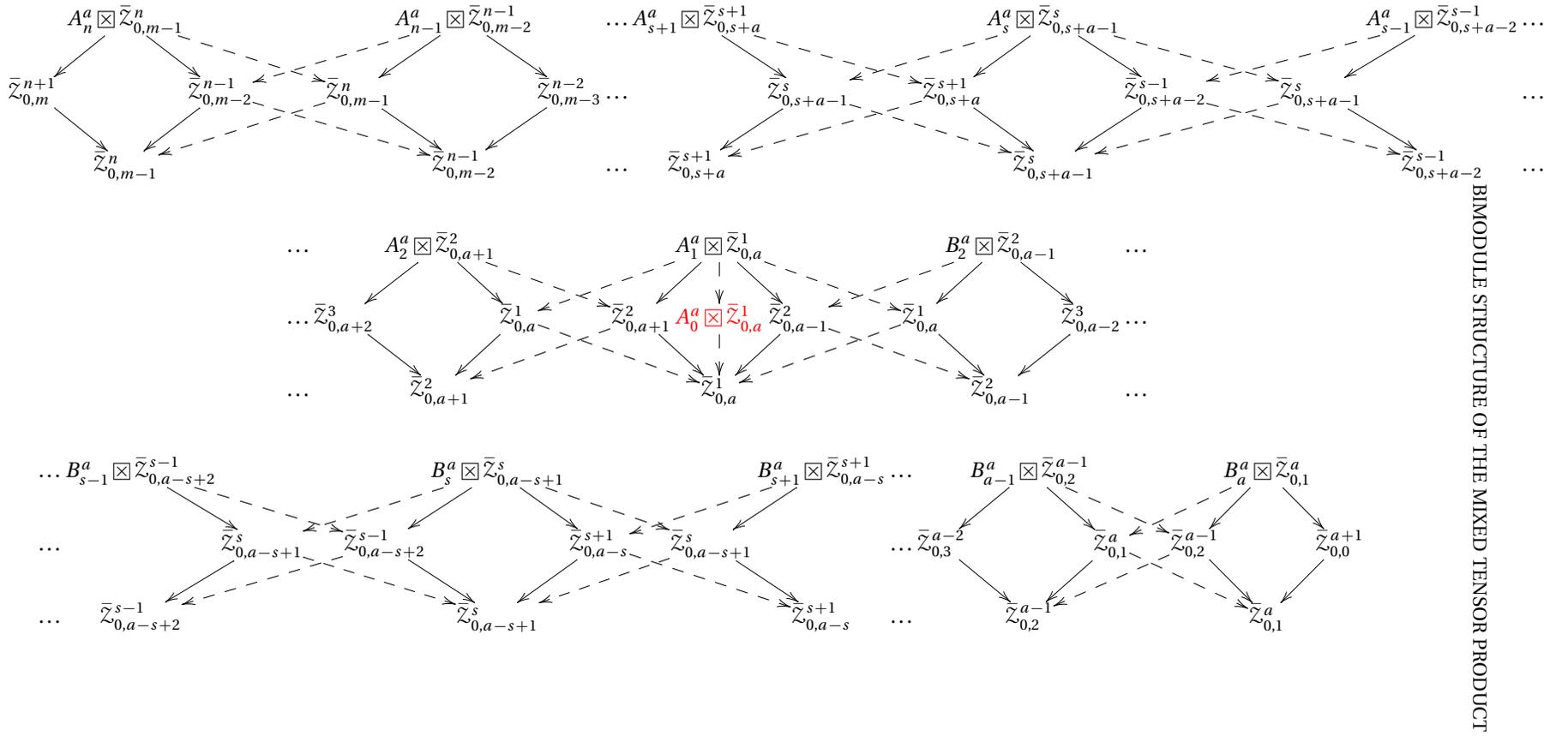
\begin{figure} 
{\footnotesize
\centering
\begin{multline*}
\hspace{-0.2cm}
 \xymatrix@R=18pt@C=1pt@W=0pt@M=0.5pt{
    &&\Da{a}{n} \boxtimes \repZZ^{n}_{0,m-1} \ar[dl]\ar[dr]\myar[drrr]  
    &&&& \Da{a}{n-1} \boxtimes \repZZ^{n-1}_{0,m-2} \ar[dl]\ar[dr]\myar[dlll]	
    &&
    \dots
    & \Da{a}{s+1}\boxtimes \repZZ^{s+1}_{0,s+a} \ar[dr]\myar[drrr]  
    &&&& \Da{a}{s} \boxtimes \repZZ^{s}_{0,s+a-1} \ar[dl]\ar[dr]\myar[dlll]\myar[drrr]	
    &&&& \Da{a}{s-1}\boxtimes \repZZ^{s-1}_{0,s+a-2} \ar[dl]\myar[dlll]
    & \dots
    \\    
    &\repZZ^{n+1}_{0,m} \hspace{3pt} \ar[dr]
    && \repZZ^{n-1}_{0,m-2} \ar[dl]\myar[drrr]
    &\hspace{1cm}
    & \repZZ^{n}_{0,m-1}  \ar[dr]\myar[dlll]
    && \repZZ^{n-2}_{0,m-3} \ar[dl]
    &
    \dots
    && \repZZ^{s}_{0,s+a-1} \ar[dl]\myar[drrr]
    &\hspace{1cm}
    & \repZZ^{s+1}_{0,s+a}  \ar[dr]\myar[dlll]
    & 
    & \repZZ^{s-1}_{0,s+a-2} \ar[dl]\myar[drrr]
    &\hspace{1cm}
    &\repZZ^{s}_{0,s+a-1} \ar[dr]\myar[dlll]
    &&\dots
    \\
    && \repZZ^{n}_{0,m-1}
    &&&& \repZZ^{n-1}_{0,m-2}
    &&
    \dots
    & \repZZ^{s+1}_{0,s+a}
    &&&& \repZZ^{s}_{0,s+a-1}
    &&&& \repZZ^{s-1}_{0,s+a-2}
    & \dots
  }\\\mbox{}\\
 \xymatrix@R=18pt@C=1pt@W=0pt@M=0.5pt{
    \dots&
    & \Da{a}{2}\boxtimes \repZZ^{2}_{0,a+1} \ar[dl]\ar[dr]\myar[drrr]  
    &&&& \Da{a}{1}\boxtimes \repZZ^{1}_{0,a} \ar[dl]\ar[dr]\myar[d]\myar[dlll]\myar[drrr]	
    &&&& \Db{a}{2}\boxtimes \repZZ^{2}_{0,a-1} \ar[dl]\ar[dr]\myar[dlll]
    && \dots
    \\    
    \dots
    &\repZZ^{3}_{0,a+2} \hspace{3pt} \ar[dr]
    && \repZZ^{1}_{0,a} \ar[dl]\myar[drrr]
    &\hspace{1cm}
    & \repZZ^{2}_{0,a+1}  \ar[dr]\myar[dlll]
    & \red{\Da{a}{0}\boxtimes \repZZ^{1}_{0,a}} \myar[d]
    & \repZZ^{2}_{0,a-1} \ar[dl]\myar[drrr]
    &\hspace{1cm}
    &\repZZ^{1}_{0,a} \ar[dr]\myar[dlll]
    && \repZZ^{3}_{0,a-2} \ar[dl]
    & \dots
    \\
    \dots
    && \repZZ^{2}_{0,a+1}
    &&&& \repZZ^{1}_{0,a}
    &&&& \repZZ^{2}_{0,a-1}
    && \dots
  }\\\mbox{}\\
\hspace{-0.5cm}
 \xymatrix@R=18pt@C=1pt@W=0pt@M=0.5pt{
    \dots
    & \Db{a}{s-1} \boxtimes \repZZ^{s-1}_{0,a-s+2} \ar[dr]\myar[drrr]  
    &&&& \Db{a}{s}\boxtimes \repZZ^{s}_{0,a-s+1} \ar[dl]\ar[dr]\myar[dlll]\myar[drrr]	
    &&&& \Db{a}{s+1}\boxtimes \repZZ^{s+1}_{0,a-s} \ar[dl]\myar[dlll]
    & \dots
	&& \Db{a}{a-1} \boxtimes \repZZ^{a-1}_{0,2} \ar[dl]\ar[dr] \myar[drrr]
	&&&& \Db{a}{a}\boxtimes \repZZ^{a}_{0,1} \ar[dl] \ar[dr]\myar[dlll] &
    \\    
    \dots
    && \repZZ^{s}_{0,a-s+1} \ar[dl]\myar[drrr] 
    &\hspace{0.5cm}
    & \repZZ^{s-1}_{0,a-s+2}  \ar[dr]\myar[dlll]
    & 
    & \repZZ^{s+1}_{0,a-s} \ar[dl]\myar[drrr]
    &\hspace{0.5cm}
    &\repZZ^{s}_{0,a-s+1} \ar[dr]\myar[dlll]
    && \dots
	& \repZZ^{a-2}_{0,3} \ar[dr]
	&& \repZZ^{a}_{0,1} \ar[dl]\myar[drrr]
	& \hspace{0.5cm}
	& \repZZ^{a-1}_{0,2} \ar[dr]\myar[dlll]  
    && \repZZ^{a+1}_{0,0} \ar[dl]  
    \\
    \dots
    & \repZZ^{s-1}_{0,a-s+2}
    &&&& \repZZ^{s}_{0,a-s+1}
    &&&& \repZZ^{s+1}_{0,a-s}
    & \dots
    && \repZZ^{a-1}_{0,2}
    &&&& \repZZ^{a}_{0,1}&
  }\\\mbox{}\\
\end{multline*}
}
\caption{$T^d_{m,n}, n=\frac{m+1}{2}, \quad n \geq 2$,  ($a=n-1$)}
\label{R3}
\end{figure}

\newpage
\thispagestyle{empty}
\begin{figure}
{\footnotesize
\centering
\begin{multline*}
\hspace{-0.2cm}
 \xymatrix@R=18pt@C=1pt@W=0pt@M=0.5pt{
    &&\Da{1}{m-1} \boxtimes \repZZ^{m-1}_{0,m-1} \ar[dl]\ar[dr]\myar[drrr]  
    &&&& \Da{1}{m-2} \boxtimes \repZZ^{m-2}_{0,m-2} \ar[dl]\ar[dr]\myar[dlll]	
    &&
    \dots
    & \Da{1}{s+1}\boxtimes \repZZ^{s+1}_{0,s+1} \ar[dr]\myar[drrr]  
    &&&& \Da{1}{s} \boxtimes \repZZ^{s}_{0,s} \ar[dl]\ar[dr]\myar[dlll]\myar[drrr]	
    &&&& \Da{1}{s-1}\boxtimes \repZZ^{s-1}_{0,s-1} \ar[dl]\myar[dlll]
    & \dots
    \\    
    &\repZZ^{m}_{0,m} \hspace{3pt} \ar[dr]
    && \repZZ^{m-2}_{0,m-2} \ar[dl]\myar[drrr]
    &\hspace{1cm}
    & \repZZ^{m-1}_{0,m-1}  \ar[dr]\myar[dlll]
    && \repZZ^{m-3}_{0,m-3} \ar[dl]
    &
    \dots
    && \repZZ^{s}_{0,s} \ar[dl]\myar[drrr]
    &\hspace{1cm}
    & \repZZ^{s+1}_{0,s+1}  \ar[dr]\myar[dlll]
    & 
    & \repZZ^{s-1}_{0,s-1} \ar[dl]\myar[drrr]
    &\hspace{1cm}
    &\repZZ^{s}_{0,s} \ar[dr]\myar[dlll]
    &&\dots
    \\
    && \repZZ^{m-1}_{0,m-1}
    &&&& \repZZ^{m-2}_{0,m-2}
    &&
    \dots
    & \repZZ^{s+1}_{0,s+1}
    &&&& \repZZ^{s}_{0,s}
    &&&& \repZZ^{s-1}_{0,s-1}
    & \dots
  }\\\mbox{}\\
 \xymatrix@R=18pt@C=1pt@W=0pt@M=0.5pt{
    \dots&
    & \Da{1}{2}\boxtimes \repZZ^{2}_{0,2} \ar[dl]\ar[dr]\myar[drrr]  
    &&&& \Da{1}{1}\boxtimes \repZZ^{1}_{0,1} \ar[dl]\ar[dr]\myar[d]\myar[dlll]\myar[drrr]	
    &&&& \Dc{1}{1}\boxtimes \repZZ^{2}_{0,0} \ar[dl]\ar[dr]\myar[dlll]
    && \dots
    \\    
    \dots
    &\repZZ^{3}_{0,3} \hspace{3pt} \ar[dr]
    && \repZZ^{1}_{0,1} \ar[dl]\myar[drrr]
    &\hspace{1cm}
    & \repZZ^{2}_{0,2}  \ar[dr]\myar[dlll]
    & \red{\Da{1}{0}\boxtimes \repZZ^{1}_{0,1}} \myar[d]
    & \repZZ^{2}_{0,0} \ar[dl]\myar[drrr]
    &\hspace{1cm}
    &\repZZ^{1}_{0,1} \ar[dr]\myar[dlll]
    && \repZZ^{3}_{1,0} \ar[dl]
    & \dots
    \\
    \dots
    && \repZZ^{2}_{0,2}
    &&&& \repZZ^{1}_{0,1}
    &&&& \repZZ^{2}_{0,0}
    && \dots
  }\\\mbox{}\\
\hspace{-0.5cm}
 \xymatrix@R=18pt@C=1pt@W=0pt@M=0.5pt{
    \dots
    & \Dc{1}{s-1} \boxtimes \repZZ^{s}_{s-2,0} \ar[dr]\myar[drrr]  
    &&&& \Dc{1}{s}\boxtimes \repZZ^{s+1}_{s-1,0} \ar[dl]\ar[dr]\myar[dlll]\myar[drrr]	
    &&&& \Dc{1}{s+1}\boxtimes \repZZ^{s+2}_{s,0} \ar[dl]\myar[dlll]
    & \dots
	&& \Dc{1}{m-4} \boxtimes \repZZ^{m-3}_{m-5,0} \ar[dl]\ar[dr] \myar[drrr]
	&&&& \Dc{1}{m-3}\boxtimes \repZZ^{m-2}_{m-4,0} \ar[dl] \ar[dr]\myar[dlll] &
    \\    
    \dots
    && \repZZ^{s+1}_{s-1,0} \ar[dl]\myar[drrr] 
    &\hspace{0.5cm}
    & \repZZ^{s}_{s-2,0}  \ar[dr]\myar[dlll]
    & 
    & \repZZ^{s+2}_{s,0} \ar[dl]\myar[drrr]
    &\hspace{0.5cm}
    &\repZZ^{s+1}_{s-1,0} \ar[dr]\myar[dlll]
    && \dots
	& \repZZ^{m-4}_{m-6,0} \ar[dr]
	&& \repZZ^{m-2}_{m-4,0} \ar[dl]\myar[drrr]
	& \hspace{0.5cm}
	& \repZZ^{m-3}_{m-5,0} \ar[dr]\myar[dlll]  
    && \repZZ^{m-1}_{m-3,0} \ar[dl]  
    \\
    \dots
    & \repZZ^{s}_{s-2,0}
    &&&& \repZZ^{s+1}_{s-1,0}
    &&&& \repZZ^{s+2}_{s,0}
    & \dots
    && \repZZ^{m-3}_{m-5,0}
    &&&& \repZZ^{m-2}_{m-4,0}&
  }\\\mbox{}\\
\end{multline*}
}
\caption{$T^d_{m,m-1}, \quad m \geq 2$ }\label{R4}
\end{figure}
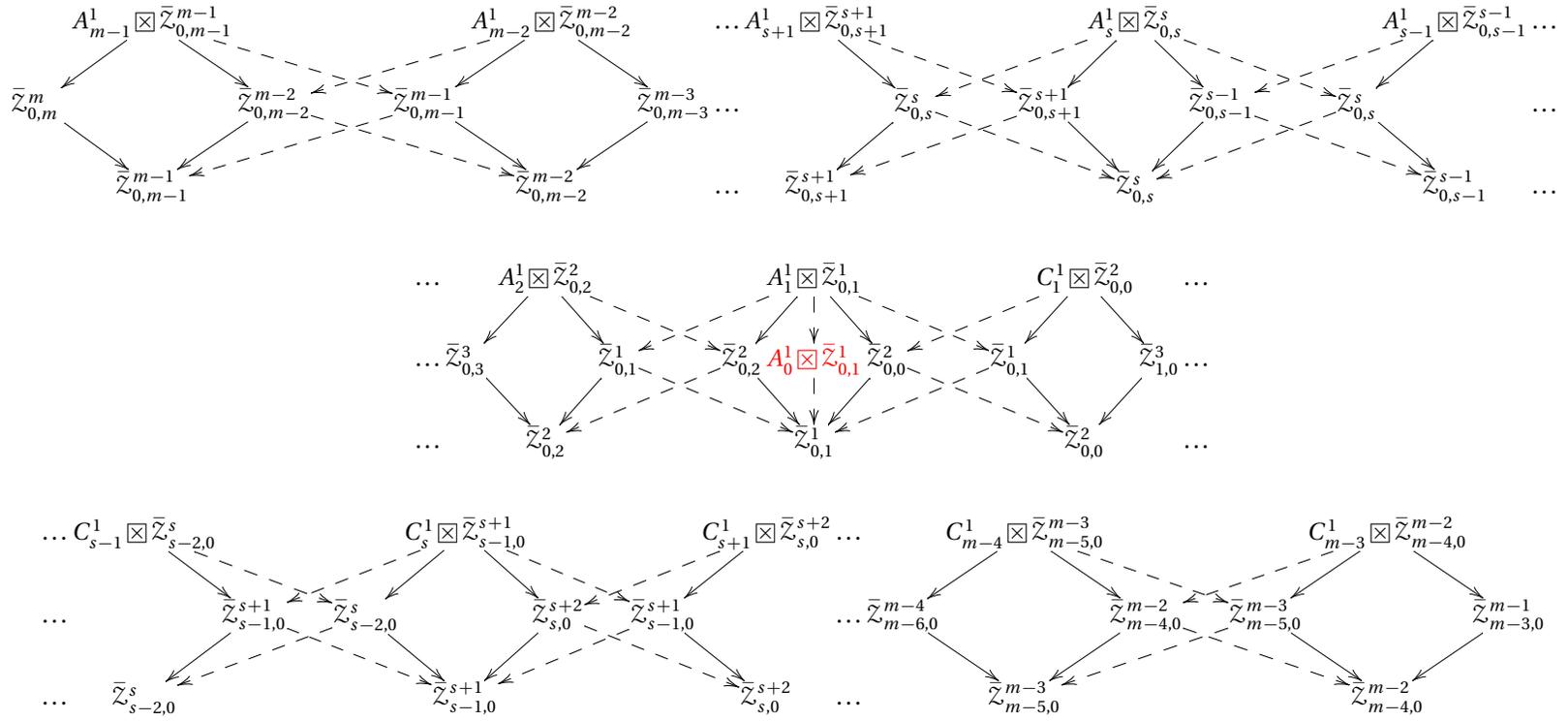

\newpage
\thispagestyle{empty}
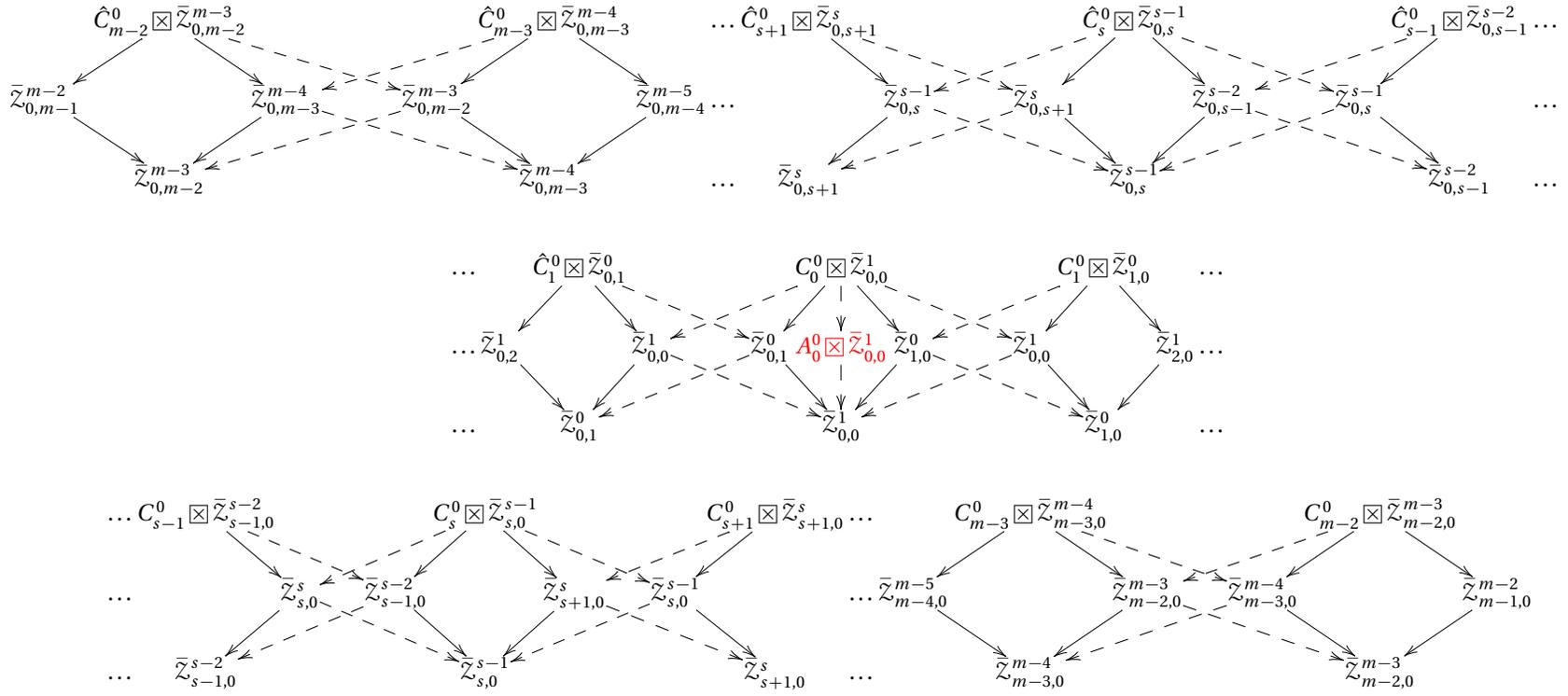
\begin{figure} 
{\footnotesize
\centering
\begin{multline*}
\hspace{-1cm}
 \xymatrix@R=18pt@C=1pt@W=0pt@M=0.5pt{
    &&\Dcc{0}{m-2} \boxtimes \repZZ^{m-3}_{0,m-2} \ar[dl]\ar[dr]\myar[drrr]  
    &&&& \Dcc{0}{m-3} \boxtimes \repZZ^{m-4}_{0,m-3} \ar[dl]\ar[dr]\myar[dlll]	
    &&
    \dots
    & \Dcc{0}{s+1}\boxtimes \repZZ^{s}_{0,s+1} \ar[dr]\myar[drrr]  
    &&&& \Dcc{0}{s} \boxtimes \repZZ^{s-1}_{0,s} \ar[dl]\ar[dr]\myar[dlll]\myar[drrr]	
    &&&& \Dcc{0}{s-1}\boxtimes \repZZ^{s-2}_{0,s-1} \ar[dl]\myar[dlll]
    & \dots
    \\    
    &\repZZ^{m-2}_{0,m-1} \hspace{3pt} \ar[dr]
    && \repZZ^{m-4}_{0,m-3} \ar[dl]\myar[drrr]
    &\hspace{1cm}
    & \repZZ^{m-3}_{0,m-2}  \ar[dr]\myar[dlll]
    && \repZZ^{m-5}_{0,m-4} \ar[dl]
    &
    \dots
    && \repZZ^{s-1}_{0,s} \ar[dl]\myar[drrr]
    &\hspace{1cm}
    & \repZZ^{s}_{0,s+1}  \ar[dr]\myar[dlll]
    && \repZZ^{s-2}_{0,s-1} \ar[dl]\myar[drrr]
    &\hspace{1cm}
    &\repZZ^{s-1}_{0,s} \ar[dr]\myar[dlll]
    &&\dots
    \\
    && \repZZ^{m-3}_{0,m-2}
    &&&& \repZZ^{m-4}_{0,m-3}
    &&
    \dots
    & \repZZ^{s}_{0,s+1}
    &&&& \repZZ^{s-1}_{0,s}
    &&&& \repZZ^{s-2}_{0,s-1}
    & \dots
  }\\\mbox{}\\
 \xymatrix@R=18pt@C=1pt@W=0pt@M=0.5pt{
    \dots&
    & \Dcc{0}{1}\boxtimes \repZZ^{0}_{0,1} \ar[dl]\ar[dr]\myar[drrr]  
    &&&& \Dc{0}{0}\boxtimes \repZZ^{1}_{0,0} \ar[dl]\ar[dr]\myar[d]\myar[dlll]\myar[drrr]	
    &&&& \Dc{0}{1}\boxtimes \repZZ^{0}_{1,0} \ar[dl]\ar[dr]\myar[dlll]
    && \dots
    \\    
    \dots
    &\repZZ^{1}_{0,2} \hspace{3pt} \ar[dr]
    && \repZZ^{1}_{0,0} \ar[dl]\myar[drrr]
    &\hspace{1cm}
    & \repZZ^{0}_{0,1}  \ar[dr]\myar[dlll]
    & \red{\Da{0}{0}\boxtimes \repZZ^{1}_{0,0}} \myar[d]
    & \repZZ^{0}_{1,0} \ar[dl]\myar[drrr]
    &\hspace{1cm}
    &\repZZ^{1}_{0,0} \ar[dr]\myar[dlll]
    && \repZZ^{1}_{2,0} \ar[dl]
    & \dots
    \\
    \dots
    && \repZZ^{0}_{0,1}
    &&&& \repZZ^{1}_{0,0}
    &&&& \repZZ^{0}_{1,0}
    && \dots
  }\\\mbox{}\\
\hspace{-0.5cm}
 \xymatrix@R=18pt@C=1pt@W=0pt@M=0.5pt{
    \dots
    & \Dc{0}{s-1} \boxtimes \repZZ^{s-2}_{s-1,0} \ar[dr]\myar[drrr]  
    &&&& \Dc{0}{s}\boxtimes \repZZ^{s-1}_{s,0} \ar[dl]\ar[dr]\myar[dlll]\myar[drrr]	
    &&&& \Dc{0}{s+1}\boxtimes \repZZ^{s}_{s+1,0} \ar[dl]\myar[dlll]
    & \dots
	&& \Dc{0}{m-3} \boxtimes \repZZ^{m-4}_{m-3,0} \ar[dl]\ar[dr] \myar[drrr]
	&&&& \Dc{0}{m-2}\boxtimes \repZZ^{m-3}_{m-2,0} \ar[dl] \ar[dr]\myar[dlll] &
    \\    
    \dots
    && \repZZ^{s}_{s,0} \ar[dl]\myar[drrr] 
    &\hspace{0.5cm}
    & \repZZ^{s-2}_{s-1,0}  \ar[dr]\myar[dlll]
    & 
    & \repZZ^{s}_{s+1,0} \ar[dl]\myar[drrr]
    &\hspace{0.5cm}
    &\repZZ^{s-1}_{s,0} \ar[dr]\myar[dlll]
    && \dots
	& \repZZ^{m-5}_{m-4,0} \ar[dr]
	&& \repZZ^{m-3}_{m-2,0} \ar[dl]\myar[drrr]
	& \hspace{0.5cm}
	& \repZZ^{m-4}_{m-3,0} \ar[dr]\myar[dlll]  
    && \repZZ^{m-2}_{m-1,0} \ar[dl]  
    \\
    \dots
    & \repZZ^{s-2}_{s-1,0}
    &&&& \repZZ^{s-1}_{s,0}
    &&&& \repZZ^{s}_{s+1,0}
    & \dots
    && \repZZ^{m-4}_{m-3,0}
    &&&& \repZZ^{m-3}_{m-2,0}&
  }\\\mbox{}\\
\end{multline*}
}
\caption{$T^d_{m,m}, \quad m \geq 2$ }\label{R5}
\end{figure}
\end{landscape}


\newpage
\begin{thebibliography}{99}

\bibitem{FGST2006_1}
B. L. Feigin, A. M. Gainutdinov, A. M. Semikhatov and I. Yu. Tipunin,
\textit{Modular group representations and fusion in logarithmic conformal field theories and in the quantum group center},
Commun.Math.Phys. 265 (2006) 47-93

\bibitem{FGST2006_3}
B. L. Feigin, A. M. Gainutdinov, A. M. Semikhatov and I. Yu. Tipunin,
\textit{Logarithmic extensions of minimal models: characters and modular transformations},
Nucl.Phys.B, 757 (2006) 303-343

\bibitem{FuchsHwangST}
J. Fjelstad, J. Fuchs, S. Hwang, A. M. Semikhatov and I. Yu. Tipunin,
\textit{Logarithmic Conformal Field Theories via Logarithmic Deformations},
Nucl.Phys. B 633 (2002) 379-413

\bibitem{SemTip2011}
A. M. Semikhatov, I. Yu. Tipunin,
\textit{The Nichols algebra of screenings},
Commun. Contemp. Math. 14 (2012) 1250029 

\bibitem{Simon2017} S. D. Lentner, 
\textit{Quantum groups and Nichols algebras acting on conformal field theories}, 
arXiv:1702.06431 [math.QA] (2017)

\bibitem{PRZ}  
P.~ Pearce, J.~ Rasmussen, and J.-B.~ Zuber, \textit{Logarithmic minimal models}, J. Stat. Mech. 
(2006).

\bibitem{ReadSaleur2007}
N. Read, H. Saleur,
\textit{Enlarged symmetry algebras of spin chains, loop models, and S-matrices},
Nucl.Phys.B 777:263-315 (2007)

\bibitem{ReadSaleur2007_2} N.~Read, H.~Saleur, 
\textit{Associative-algebraic approach to logarithmic conformal field theories}, Nucl. Phys. B { 777}, 316 (2007). 

\bibitem{Ridout2007}
Pierre Mathieu, D. Ridout,
\textit{From percolation to logarithmic conformal field theory}, 
Phys.Lett.B657 (2007) 120-129

\bibitem{Ridout2015}
A. Morin-Duchesne, J. Rasmussen and D. Ridout,
\textit{Boundary algebras and Kac modules for logarithmic minimal models}, 
arXiv:1503.07584 [hep-th] (2015)

\bibitem{ga} A.~ M.~ Gainutdinov, N.~ Read, H.~ Saleur and R.~
  Vasseur, {\textit{The periodic $sl(2|1)$ alternating spin chain and
      its continuum limit as a bulk Logarithmic Conformal Field Theory
      at $c=0$}}, JHEP 1505 (2015).

\bibitem{AzatVasseur2013}
A. M. Gainutdinov, R. Vasseur,
\textit{Lattice fusion rules and logarithmic operator product expansions},
Nucl. Phys. B 868, 223-270 (2013)

\bibitem{DiFrancesco}
Ph. Di Francesco, P. Mathieu and D. Senechal, 
\textit{Conformal Field Theory}, Springer, New York, 1997

\bibitem{AzatReadSaleur2016} A. M. Gainutdinov, N. Read and H. Saleur, 
\textit{Associative algebraic approach to logarithmic CFT in the bulk: the continuum limit of the $g\ell(1|1)$ periodic spin chain, Howe duality and the interchiral algebra},  Comm. Math. Phys., 341, 35 (2016).

\bibitem{DuplantierSaleur1987} B. Duplantier, H. Saleur,
\textit{Exact critical properties of two-dimensional dense self-avoiding walks},  
Nucl.Phys.B, 290 (1987) 291-326

\bibitem{FrancSaleurZuber1987} P. Di Francesco, H. Saleur, J.B. Zuber,
\textit{Critical Ising correlation functions in the plane and on the torus},  
Nucl.Phys.B, 290 (1987) 527-581

\bibitem{Cardy1991}
J. Cardy,
\textit{Critical Percolation in Finite Geometries},
J.Phys. A25 (1992) L201-L206

\bibitem{Ruelle2013}
Ph. Ruelle,
\textit{Logarithmic conformal invariance in the Abelian sandpile model},
J. Phys. A: Math. Theor. 46 (2013) 494014 (39pp)

\bibitem{Rasmussen2007_1}
J. Rasmussen, P. Pearce,
\textit{Fusion algebra of critical percolation},
J. Stat. Mech. P09002 (2007)

\bibitem{Rasmussen2007_2}
J. Rasmussen, P. Pearce,
\textit{Fusion algebras of logarithmic minimal models}, 
J. Phys. A40 (2007) 13711 

\bibitem{AzatSaleur2016} A.M. Gainutdinov, H. Saleur, 
\textit{Fusion and braiding in finite and affine Temperley-Lieb categories}, 
arXiv:1606.04530 [math.QA] (2016)

\bibitem{Baxter}	
R. J. Baxter, \textit{Exactly Solved Models in Statistical Mechanics}, Academic Press, 1982

\bibitem{Quantum_Groups_2dim}	
C. Gómez , M. Ruiz-Altaba and G. Sierra,
\textit{Quantum Groups in Two-Dimensional Physics}, Cambridge University Press, 1996

\bibitem{Dubail2010} J.~Dubail, J.~L. Jacobsen and H.~Saleur, 
\textit{Conformal field theory at central
  charge c = 0: A measure of the indecomposability (b) parameters},
Nucl.Phys.B, 834 (2010) 399-422 

\bibitem{VasseurJacobsen2011} R.~Vasseur, J.~L. Jacobsen, and H.~Saleur, 
\textit{Indecomposability parameters in chiral Logarithmic Conformal Field Theory},
Nucl.Phys.B, 851 (2011) 314-345 

\bibitem{MorinDuchesne2011} A.~Morin-Duchesne and Y.~Saint-Aubin, 
\textit{The Jordan Structure of Two Dimensional Loop Models},
J.Stat.Mech.1104:P04007 (2011)

\bibitem{PearceRasmussen2013} P. Pearce, J. Rasmussen and E. Tartaglia, 
\textit{Logarithmic superconformal minimal models}, 
arXiv:1312.6763 [hep-th] (2013)

\bibitem{BPPR_2014}
J. Brankov, V. Poghosyan, V. Priezzhev and P. Ruelle,
\textit{Transfer matrix for spanning trees, webs and colored forests},
J. Stat. Mech. (2014) P09031

\bibitem{AzatNepomechie2016} A.M. Gainutdinov, R.I. Nepomechie, 
\textit{Algebraic Bethe ansatz for the quantum group invariant open XXZ chain at roots of unity},
arXiv:1603.09249 [math-ph] (2016)

\bibitem{FGST2006_2}
B. L. Feigin, A. M. Gainutdinov, A. M. Semikhatov and I. Yu. Tipunin,
\textit{Kazhdan--Lusztig correspondence for the representation category of the triplet W-algebra in logarithmic CFT},
Theor.Math.Phys. 148 (2006) 1210-1235; Teor.Mat.Fiz. 148 (2006) 398-427

\bibitem{FGST2007}
B. L. Feigin, A. M. Gainutdinov, A. M. Semikhatov and I. Yu. Tipunin,
\textit{Kazhdan--Lusztig-dual quantum group for logarithmic extensions of Virasoro minimal models},
J.Math.Phys.48 (2007) 032303 

\bibitem{SemTip2013CFT}
A. M. Semikhatov, I. Yu. Tipunin,
\textit{Logarithmic $\hat{s\ell(2)}$ CFT models from Nichols algebras. 1},
J. Phys. A: Math. Theor. 46 (2013) 494011

\bibitem{GaberdielKausch} M.R. Gaberdiel, H.G. Kausch, 
\textit{A rational logarithmic conformal field theory},
Phys. Lett. B, 386, (1996) 131 

\bibitem{FHST2004} J. Fuchs, S. Hwang, A.M. Semikhatov and I. Yu. Tipunin, 
\textit{Nonsemisimple Fusion Algebras and the Verlinde Formula},
Commun.Math.Phys. 247 (2004) 713-742 

\bibitem{AdamovicMilas2007_1}
D. Adamovic, A. Milas,
\textit{Logarithmic intertwining operators and $\mathscr{W}(2,2p-1)$-algebras},
J.Math.Phys. 48:073503 (2007)

\bibitem{AdamovicMilas2007_2}
D. Adamovic, A. Milas,
\textit{On the triplet vertex algebra $\mathscr{W}(p)$},
Advances in Mathematics, 217 (2008) 2664-2699

\bibitem{AzatSalTip2012}
A.M. Gainutdinov, H. Saleur and I.Yu. Tipunin, \textit{Lattice W-algebras and logarithmic CFTs}, J. Phys. A: Math. Theor. 47 (2014) 495401.

\bibitem{Candu2010}
C. Candu,
\textit{Continuum limit of $g\ell(M|N)$ spin chains},
JHEP 1107:069 (2011)

\bibitem{Links} J.~Links, A.~Foerster, {\textit{Integrability of a t-J model with impurities}}, Journal of Physics A General Physics 32(1) (1998).

\bibitem{Abad} J. ~Abad, M. ~Rios,
  {\textit{Exact solution of a electron system combining two different t-J models}}, Journal of Physics A: Mathematical and General 32 (19) (1999).

\bibitem{FeiYue1994}
S.-M. Fei, R.-H. Yue,
\textit{Generalized t-j Model},
J. Phys. A: Math. Gen. 27 3715 (1994), arXiv:cond-mat/9405016

\bibitem{Ridout2012}
T. Creutzig, D. Ridout,
\textit{Modular Data and Verlinde Formulae for Fractional Level WZW Models I},
arXiv:1205.6513 [hep-th] (2012)

\bibitem{Ridout2013}
T. Creutzig, D. Ridout,
\textit{Modular Data and Verlinde Formulae for Fractional Level WZW Models II},
arXiv:1306.4388 [hep-th] (2013)

\bibitem{RidoutWood_2015}
D. Ridout, S. Wood,
\textit{Relaxed singular vectors, Jack symmetric functions and fractional level $\hat{sl}(2)$ models},
arXiv:1306.4388 [hep-th] (2013)

\bibitem{BGT} P.~ V.~Bushlanov, A.~M.~Gainutdinov, I.~Yu.~Tipunin,
  {\textit{Kazhdan-Lusztig equivalence and fusion of Kac modules in
      Virasoro logarithmic models}}, Nucl.Phys. B862 (2012).

\bibitem{Saleur2005}
F.H.L. Essler, H. Frahm, H. Saleur,
\textit{Continuum Limit of the Integrable $s\ell(2|1)$ $3-\bar{3}$ Superspin Chain},
Nucl.Phys. B712 (2005) 513-572

\bibitem{Frahm2011}
H. Frahm, M.J. Martins,
\textit{Finite size properties of staggered $\qSL{2|1}$ superspin chains},
Nucl.Phys.B847:220-246 (2011)

\bibitem{Frahm2012}
H. Frahm, M.J. Martins,
\textit{Phase diagram of an integrable alternating  $\qSL{2|1}$ superspin chain},
Nucl. Phys. B 862 [FS] (2012) 504-552

\bibitem{ShaderMoon02}
C.~L.~Shader and D.~Moon, {\textit{Mixed tensor representations and rational representations for the general linear Lie superalgebras}}, Commun. Algebra 30:2 (2002), 839-857.

\bibitem{DDS-1} R.~Dipper, S.~Doty, F.~Stoll \textit{The quantized
    walled Brauer algebra and mixed tensor space}, F. Algebr. Represent. Theor. (2014) 17: 675. 

\bibitem{DDS-2}R.~Dipper, S.~Doty, F.~Stoll, \textit{Quantized mixed
    tensor space and Schur-Weyl duality}, Algebra \& Number Theory 7
  (2013) 1121--1146, arXiv:0810.1227 [math.RT]

\bibitem{Stroppel12}
J.~Brundan and C.~Stroppel, {\textit{Gradings on walled Brauer algebras and Khovanov's arc algebra}},
Adv. Math. 231:2 (2012), 709-773.

\bibitem{Leduc} R.~Leduc, {\textit{A two-parameter version of the centralizer algebra of the mixed tensor representation of the general linear group and quantum general linear group}}, thesis, University of Wisconsin-Madison (1994).

\bibitem{Halv} T.~Halverson, \textit{Characters of the centralizer
    algebras of mixed tensor representations of $GL(r,\oC)$ and the quantum
    group $U_q (g\ell(r,\oC))$}, Pacific Jornal of Mathematics, 174 (1996) 259-410.

\bibitem{KM1993} M.~Kosuda and J.~Murakami, \textit{Centralizer algebras
    of the mixed tensor representations of quantum group
    $\qGL{m, \oC}$}, Osaka J. Math. 30 (1993) 475--507.

\bibitem{En} Enyang, \textit{Cellular bases of the two-parameter
    version of the centraliser algebra for the mixed tensor
    representations of the quantum general linear group},
  Combinatorial representation theory and related topics (Kyoto, 2002)
  1310 (2003) 134--153.

\bibitem{CDM2007} A.~Cox, M.~De~Visscher, S.~Doty and P.~Martin, 
\textit{On the blocks of the walled Brauer algebra}, J. Algebra 320 (2008), 169-212.

\bibitem{Cox2010} A.~Cox and M.~De~Visscher, 
{\textit{Diagrammatic Kazhdan-Lusztig theory for the (walled) Brauer algebra}}, Journal of Algebra 340(1) (2011).

\bibitem{RuiSong_1}
H. Rui, L. Song
\textit{The representations of quantized walled Brauer algebras},
arXiv:1403.7722 [math.QA] (2014)

\bibitem{RuiSong_2}
H. Rui, L. Song
\textit{Decomposition numbers of quantized walled Brauer algebras},
arXiv:1403.7740 [math.QA] (2014)

\bibitem{ST2015} A. ~M.~ Semikhatov, I. ~Yu.~ Tipunin, 
{\textit{Quantum walled Brauer algebra: commuting families, baxterization, and representations}}, 
arXiv:1512.06994v1 [math.QA].

\bibitem{Stroppel2014}
A. Sartori, C. Stroppel,
\textit{Walled Brauer algebras as idempotent truncations of level $2$ cyclotomic quotients},
arXiv:1411.2771 [math.RT] (2014)

\bibitem{pal1} T.~D.~ Palev and V.~N.~ Tolstoy, {\textit{Finite dimensional irreducible representations of the quantum superalgebra $U_q g\ell(n|1)$}}, Commun. Math. Phys. 141 (1991).

\bibitem{pal2} T.~D.~ Palev, N.~I.~ Stoilova and J.~ Van der Jeugt, {\textit{Finite-dimensional representations of the quantum superalgebra $U_q g\ell(n|m)$ and related q-identities}}, Commun. Math. Phys.
166 (1994).

\bibitem{zhang} R.~ B.~ Zhang, {\textit{Finite dimensional irreducible representations of the quantum supergroup $U_q g\ell(m|n)$}}, J. Math. Phys. 34 (3) (1993).

\bibitem{ky} Nguyen Anh Ky and Nguyen thi Hong Van, {\textit{Finite-dimensional representations of $U_q g\ell(2|1)$ in a basis of $U_q [g\ell(2)\oplus g\ell(1)]$}}, Advances in Natural Sciences 5 (2004), 
p.1.

\bibitem{su} Yucai Su, {\textit{Classification of finite dimensional modules of the Lie superalgebra $s\ell(2|1)$}}, Communications in Algebra 20:11 3259-3277 (1992). 

\bibitem{TR-reps} A.~ M. ~Semikhatov, I. ~Yu.~ Tipunin,
  {\textit{Representations of $\bar{U}_q s\ell(2|1)$ at even roots of
      unity}},   J. Math. Phys. 57 (2), (2016).

\bibitem{Maclane}
S. Maclane, \textit{Homology}, Springer Verlag, 1963.

\bibitem{gotz} Gerhard Gotz, Thomas Quella, Volker Schomerus, {\textit{Representation theory of $s\ell(2|1)$}}, Journal of Algebra 312:2 (2007).

\bibitem{ShaderMoon07}
C.~L.~Shader and D.~Moon,
\textit{Mixed tensor
representations of quantum superalgebra $\qGL{m,n}$}, 
Commun. Algebra, 35:3 (2007), 781-806.

\bibitem{StollWerth2014}
F. Stoll, M. Werth,
\textit{A cell filtration of mixed tensor space},
arXiv:1408.6720 [math.RT] (2014)
  
\bibitem{CW11}
J.~Comes and B.~Wilson, {\textit{Deligne's category \underline{Rep}($GL_\delta$) and representations of
general linear supergroups}}; arXiv:1108.0652.


\end{thebibliography}
\end{document}